\pdfoutput=1
\documentclass[letterpaper,12pt,oneside,final]{article}
\usepackage{amsmath}
\usepackage{amsthm}
\usepackage{amssymb}
\usepackage{enumitem}
\usepackage{graphicx}
\usepackage{float}
\usepackage{hyperref}
\usepackage{cleveref}
\usepackage{tikz,ifthen}
\usetikzlibrary{decorations.pathreplacing}
\usepackage{url}
\hypersetup{
    colorlinks,
    linkcolor={blue},
    citecolor={blue},
    urlcolor={blue}
}
\usepackage{subcaption}
\newtheorem{theorem}{Theorem}[section]

\newtheorem{lemma}[theorem]{Lemma}
\newtheorem{corollary}[theorem]{Corollary}
\theoremstyle{definition}
\newtheorem{definition}[theorem]{Definition}
\newtheorem{example}[theorem]{Example}
\newtheorem{remark}[theorem]{Remark}
\newtheorem{note}[theorem]{Note}

\title{Vertex models for the product of a permuted-basement Demazure atom and a Schur polynomial}
\author{Timothy C. Miller \\
    \small Department of Combinatorics and Optimization \\[-0.8ex]
    \small University of Waterloo \\ [-0.8ex]
    \small Waterloo, ON N2L 3G1 \\ [-0.8ex]
    \small Canada \\
    \small\tt \href{mailto:tcmiller@uwaterloo.ca}{tcmiller@uwaterloo.ca} \\
}

\newcommand{\al}{\alpha}
\newcommand{\be}{\beta}
\newcommand{\de}{\delta}

\newcommand{\la}{\lambda}
\newcommand{\si}{\sigma}

\newcommand{\Z}{\mathbb{Z}}

\newcommand{\A}{\mathcal{A}}

\newcommand{\SSAF}{{\operatorname{SSAF}}}

\newcommand{\id}{{\operatorname{id}}}

\newcommand{\len}{\operatorname{len}}
\newcommand{\BB}[1]{${\color{blue6}\mathbf{#1}}$}
\newcommand{\itc}[1]{\textbf{\textit{\color{blue6}#1}}}

\def\thickwidth{1mm}
\def\colourwidth{0.6mm}
\def\circwidth{0.3mm}
\def\circsize{1.1mm}
\definecolor{blue1}{rgb}{0,1.0,1.0}
\definecolor{blue2}{rgb}{0,0.8,0.9}
\definecolor{blue3}{rgb}{0,0.6,0.8}
\definecolor{blue4}{rgb}{0,0.4,0.7}
\definecolor{blue5}{rgb}{0,0.2,0.6}
\definecolor{blue6}{rgb}{0,0.0,0.5}
\definecolor{red}{rgb}{1,0,0}
\definecolor{purple1}{rgb}{0.9,0,0.9}
\definecolor{purple2}{rgb}{0.7,0,0.7}
\definecolor{purple3}{rgb}{0.5,0,0.5}
\definecolor{gray}{rgb}{0.4,0.4,0.4}
\definecolor{light gray}{rgb}{0.9,0.9,0.9}
\tikzset{thick/.style={line width=\thickwidth}}
\tikzset{thick dashed/.style={line width=\thickwidth,dashed}}
\tikzset{blue1 line/.style={line width=\colourwidth,blue1}}
\tikzset{blue2 line/.style={line width=\colourwidth,blue2}}
\tikzset{blue3 line/.style={line width=\colourwidth,blue3}}
\tikzset{blue4 line/.style={line width=\colourwidth,blue4}}
\tikzset{blue5 line/.style={line width=\colourwidth,blue5}}
\tikzset{blue6 line/.style={line width=\colourwidth,blue6}}
\tikzset{red line/.style={line width=\colourwidth,red}}
\tikzset{purple1 line/.style={line width=\colourwidth,purple1}}
\tikzset{purple2 line/.style={line width=\colourwidth,purple2}}
\tikzset{purple3 line/.style={line width=\colourwidth,purple3}}
\tikzset{red block/.style={line width=\colourwidth*2,color=red,fill=red,opacity=0.20}}
\tikzset{purple outline/.style={line width=\colourwidth*2,color=purple1,fill=purple1,opacity=0.20}}
\tikzset{dotted red line/.style={line width=\colourwidth,red,dotted}}
\tikzset{blue weight/.style={blue1, opacity=0.15}}
\tikzset{red weight/.style={red, opacity=0.15}}
\tikzset{gray fill/.style={fill=light gray, draw=gray}}
\def\ascale{4mm}
\def\bscale{3mm}
\def\cscale{2.5mm}
\def\dscale{2mm}
\def\v{{sqrt(3)}}
\def\vx{sqrt(3)}
\def\smfont{\footnotesize}

\newcommand{\vver}{--++(1,+\v)}
\newcommand{\ver}{--++(1/2,+\v/2)}
\newcommand{\hor}{--++(-1,0)}
\newcommand{\hhor}{--++(-2,0)}

\newcommand{\rhhor}{--++(2,0)}
\newcommand{\dia}{--++(-1/2,+\v/2)}
\newcommand{\ddia}{--++(-1,+\v)}
\newcommand{\lel}{--++(1/2,+\v/2)--++(-1/2,+\v/2)}
\newcommand{\rel}{--++(-1/2,+\v/2)--++(1/2,+\v/2)}

\newcommand{\blueweight}[2]{
    \filldraw[blue weight]({2*(#1-1)+#2-1},{+\v*(#2-1)})--++(2,0)--++(1,\v)--++(-2,0)--cycle;
}

\newcommand{\redcirc}[2]{
    \filldraw[color=white, fill=red, line width=\circwidth]({#1},{#2}) circle (\circsize);
}
\newcommand{\whitecirc}[2]{
    \filldraw[color=black, fill=white, line width=\circwidth]({#1},{#2}) circle (\circsize);
}
\newcommand{\colourcirc}[3]{
    \filldraw[color=white, fill=#3, line width=\circwidth]({#1},{#2}) circle (\circsize);
}
\newcommand{\halfcirc}[3]{
    \begin{scope}[shift={({#1},{#2})}]
        \fill[#3] (0,0) circle (\circsize);
        \fill[red] (0,0) -- (0:\circsize) arc (0:180:\circsize);
        \draw[white, line width=\circwidth] (0,0) circle (\circsize);
    \end{scope}
}
\newcommand{\whiteredcirc}[2]{
    \filldraw[color=red, fill=white, line width=\circwidth]({#1},{#2}) circle (\circsize);
    \draw({#1},{#2})node{\tiny$+$};
}
\newcommand{\blueredcirc}[2]{
    \filldraw[color=red, fill=blue1, line width=\circwidth]({#1},{#2}) circle (\circsize);
    \draw({#1},{#2})node{\tiny$\star$};
}

\newcommand{\rcirc}{{\tikz{\filldraw[color=red, fill=red, line width=\circwidth](0,0) circle (\circsize)}}}
\newcommand{\bacirc}{{\tikz{\filldraw[color=blue1, fill=blue1, line width=\circwidth](0,0) circle (\circsize)}}}

\newcommand{\wcirc}{
    \begin{tikzpicture}
        \filldraw[draw,color=black, fill=white, line width=\circwidth](0,0) circle (\circsize);
    \end{tikzpicture}
}
\newcommand{\wrcirc}{
    \begin{tikzpicture}
        \filldraw[color=red, fill=white, line width=\circwidth](0,0) circle (\circsize);
        \clip (0,0) circle (\circsize);
        \node at (0,0) {\tiny$+$};
    \end{tikzpicture}
}
\newcommand{\brcirc}{
    \begin{tikzpicture}
        \filldraw[color=red, fill=blue1, line width=\circwidth](0,0) circle (\circsize);
        \clip (0,0) circle (\circsize);
        \draw(0,0)node[draw]{\tiny$\star$};
    \end{tikzpicture}
}

\newcommand{\perparrow}[7]{
    \def\rd{(#5)/sqrt((#3-(#1))^2+(#4-(#2))^2)}
    \def\xd{(#4-(#2))*\rd}
    \def\yd{(#1-(#3))*\rd}
    \draw[-stealth] ({#1+\xd},{#2+\yd})--node[#6]{#7}({#3+\xd},{#4+\yd});
}

\newcommand{\double}[4]{
    \begin{scope}
        \draw[#3]({#1},{#2-1/15})--++(#4,0);
        \clip({#1},{#2})--++(#4,0)--++(0,1)--++(-#4,0)--++(0,-1);
        \draw[red line]({#1},{#2+1/15})--++(#4,0);
    \end{scope}
}

\newcommand{\LCHEV}[5]{
    \filldraw[gray fill](#1,+\v*#1)--++(2,0)--++(1,-\v)--node[right]{\smfont$#3$}++(1,+\v)--node[right]{\smfont$#5$}++(-1,+\v)--node[above]{\smfont$#4$}++(-2,0)--node[left]{\smfont$#2$}cycle;
    \draw[gray](#1+2,+\v*#1)--++(1,+\v);
}
\newcommand{\LDCHEV}[1]{
    \filldraw[light gray](#1,+\v*#1)--++(2,0)--++(1,-\v)--++(1,+\v)--++(-1,+\v)--++(-2,0)--cycle;
    \draw[gray,dashed](#1,+\v*#1)--++(1,+\v);
    \draw[gray,dashed](#1+2,+\v*#1)--++(1,+\v);
    \draw[gray,dashed](#1+3,{\vx*(#1-1)})--++(1,+\v);
    \draw[gray](#1,+\v*#1)--++(2,0)--++(1,-\v);
    \draw[gray](#1+1,{\vx*(#1+1)})--++(2,0)--++(1,-\v);
}
\newcommand{\RCHEV}[5]{
    \filldraw[gray fill](#1,+\v*#1)--node[left]{\smfont$#4$}++(1,-\v)--node[below]{\smfont$#5$}++(2,0)--node[right]{\smfont$#3$}++(1,+\v)--++(-2,0)--++(-1,+\v)--node[left]{\smfont$#2$}cycle;
    \draw[gray](#1+1,{\vx*(#1-1)})--++(1,+\v);
}
\newcommand{\RDCHEV}[1]{
    \filldraw[light gray](#1,+\v*#1)--++(1,-\v)--++(2,0)--++(1,+\v)--++(-2,0)--++(-1,+\v)--cycle;
    \draw[gray,dashed](#1,+\v*#1)--++(1,+\v);
    \draw[gray,dashed](#1+1,{\vx*(#1-1)})--++(1,+\v);
    \draw[gray,dashed](#1+3,{\vx*(#1-1)})--++(1,+\v);
    \draw[gray](#1,+\v*#1)--++(1,-\v)--++(2,0);
    \draw[gray](#1+1,{\vx*(#1+1)})--++(1,-\v)--++(2,0);
}
\newcommand{\BLOCK}[5]{
    \filldraw[gray fill](#1,+\v*#1)--node[left]{\smfont$#2$}++(1,-\v)--node[below]{\smfont$#3$}++(2,0)--node[right]{\smfont$#5$}++(-1,+\v)--node[above]{\smfont$#4$}cycle;
}
\newcommand{\LCHEVC}[5]{
    \draw[gray](#1,+\v*#1)--++(2,0)--++(1,-\v)--node[right,black]{\smfont$#3$}++(1,+\v)--node[right,black]{\smfont$#5$}++(-1,+\v)--node[above,black]{\smfont$#4$}++(-2,0)--node[left,black]{\smfont$#2$}cycle;
    \draw[gray](#1+2,+\v*#1)--++(1,+\v);
}
\newcommand{\LDCHEVC}[1]{
    \draw[gray,dashed](#1,+\v*#1)--++(1,+\v);
    \draw[gray,dashed](#1+2,+\v*#1)--++(1,+\v);
    \draw[gray,dashed](#1+3,{\vx*(#1-1)})--++(1,+\v);
    \draw[gray](#1,+\v*#1)--++(2,0)--++(1,-\v);
    \draw[gray](#1+1,{\vx*(#1+1)})--++(2,0)--++(1,-\v);
}
\newcommand{\RCHEVC}[5]{
    \draw[gray](#1,+\v*#1)--node[left,black]{\smfont$#4$}++(1,-\v)--node[below,black]{\smfont$#5$}++(2,0)--node[right,black]{\smfont$#3$}++(1,+\v)--++(-2,0)--++(-1,+\v)--node[left,black]{\smfont$#2$}cycle;
    \draw[gray](#1+1,{\vx*(#1-1)})--++(1,+\v);
}
\newcommand{\RDCHEVC}[1]{
    \draw[gray,dashed](#1,+\v*#1)--++(1,+\v);
    \draw[gray,dashed](#1+1,{\vx*(#1-1)})--++(1,+\v);
    \draw[gray,dashed](#1+3,{\vx*(#1-1)})--++(1,+\v);
    \draw[gray](#1,+\v*#1)--++(1,-\v)--++(2,0);
    \draw[gray](#1+1,{\vx*(#1+1)})--++(1,-\v)--++(2,0);
}
\newcommand{\BLOCKC}[5]{
    \draw[gray](#1,+\v*#1)--node[left,black]{\smfont$#2$}++(1,-\v)--node[below,black]{\smfont$#3$}++(2,0)--node[right,black]{\smfont$#5$}++(-1,+\v)--node[above,black]{\smfont$#4$}cycle;
}
\newcommand{\LCUT}[1]{
    \filldraw[light gray](#1,+\v*#1)--++(1,-\v)--++(2,0)--++(-1,-\v)--++(-1,\v)--++(-2,0)--cycle;
    \draw[gray,dotted](#1,+\v*#1)--++(1,-\v)--++(2,0);
}
\newcommand{\RCUT}[1]{
    \filldraw[light gray](#1,+\v*#1)--++(1,-\v)--++(2,0)--++(-1,-\v)--++(-1,\v)--++(-2,0)--cycle;
    \draw[gray,dotted](#1-1,{\vx*(#1-1)})--++(2,0)--++(1,-\v);
}
\newcommand{\LWTE}[1]{\whitecirc{{#1+1/2}}{{\vx*(#1+1/2)}}}
\newcommand{\RWTE}[1]{\whitecirc{{#1+7/2}}{{\vx*(#1-1/2)}}}

\newcommand{\LWRD}[1]{\whiteredcirc{{#1+1/2}}{{\vx*(#1-1/2)}}}
\newcommand{\RWRD}[1]{\whiteredcirc{{#1+5/2}}{{\vx*(#1-1/2)}}}
\newcommand{\LWHT}[1]{\whitecirc{{#1+1/2}}{{\vx*(#1-1/2)}}}
\newcommand{\RWHT}[1]{\whitecirc{{#1+5/2}}{{\vx*(#1-1/2)}}}
\newcommand{\LBRD}[1]{\blueredcirc{{#1+1/2}}{{\vx*(#1+1/2)}}}
\newcommand{\RBRD}[1]{\blueredcirc{{#1+7/2}}{{\vx*(#1-1/2)}}}
\newcommand{\clipcol}[2]{
    \begin{scope}
        \clip(0,0)--++(1,-\v)--++(2,0)--++(#1,+\v*#1)--++(-1,+\v)--++(-2,0)--cycle;
        #2
    \end{scope}
}
\newcommand{\clipdouble}[3]{
    \begin{scope}
        \clip(#1,#2)--++(2,0)--++(1,\v)--++(-2,0)--cycle;
        \begin{scope}[shift={(-1/2,+\v/2)}]
            \double{#1}{#2}{#3}{4}
        \end{scope}
    \end{scope}
}
\newcommand{\drawdiamond}[5]{
    \begin{scope}
        \clip(0,0)--++(1,-\v)--++(1,\v)--++(-1,\v)--cycle;
        #5
    \end{scope}
    \draw[gray](0,0)--node[left,black]{\smfont$#3$}++(1,-\v)--node[right,black]{\smfont$#4$}++(1,\v)--node[right,black]{\smfont$#2$}++(-1,\v)--node[left,black]{\smfont$#1$}cycle;
}
\newcommand{\Dthor}[1]{\draw[#1](0,+\v/2)--++(2,0);}
\newcommand{\Dbhor}[1]{\draw[#1](0,-\v/2)--++(2,0);}
\newcommand{\Dthord}[1]{\double{0}{+\v/2}{#1}{2}}
\newcommand{\Dbhord}[1]{\double{0}{-\v/2}{#1}{2}}

\newcommand{\Drelb}[1]{\draw[#1](2,-\v)--++(-1,\v)--++(1,\v);}
\newcommand{\Dlelb}[1]{\draw[#1](0,-\v)--++(1,\v)--++(-1,\v);}
\newcommand{\Drdia}[1]{\draw[#1](2,-\v)--++(-2,+\v*2);}
\newcommand{\Dldia}[1]{\draw[#1](0,-\v)--++(2,+\v*2);}

\newcommand{\drawrightshear}[5]{
    \begin{scope}
        \clip(0,0)--++(2,0)--++(1,\v)--++(-2,0)--cycle;
        #5
    \end{scope}
    \draw[gray](0,0)--node[below,black]{\smfont$#3$}++(2,0)--node[right,black]{\smfont$#4$}++(1,\v)--node[above,black]{\smfont$#2$}++(-2,0)--node[left,black]{\smfont$#1$}cycle;
}
\newcommand{\Rver}[1]{\draw[#1](1/2,-\v/2)--++(2,+\v*2);}
\newcommand{\Rhor}[1]{\draw[#1](0,+\v/2)--++(3,0);}
\newcommand{\Rhord}[1]{\double{0}{+\v/2}{#1}{3}}

\newcommand{\Rrelb}[1]{\draw[#1](7/2,+\v/2)--++(-2,0)--++(-1,-\v);}
\newcommand{\Rlelb}[1]{\draw[#1](-1/2,+\v/2)--++(2,0)--++(1,\v);}
\newcommand{\Rrdia}[1]{\draw[#1](3,0)--++(-2,+\v*2);}
\newcommand{\Rldia}[1]{\draw[#1](0,\v)--++(2,-\v*2);}

\newcommand{\drawleftshear}[5]{
    \begin{scope}
        \clip(1,0)--++(2,0)--++(-1,\v)--++(-2,0)--cycle;
        #5
    \end{scope}
    \draw[gray](1,0)--node[below,black]{\smfont$#4$}++(2,0)--node[right,black]{\smfont$#2$}++(-1,\v)--node[above,black]{\smfont$#1$}++(-2,0)--node[left,black]{\smfont$#3$}cycle;
}
\newcommand{\Lver}[1]{\draw[#1](5/2,-\v/2)--++(-2,+\v*2);}
\newcommand{\Lhor}[1]{\draw[#1](0,+\v/2)--++(3,0);}
\newcommand{\Lrelb}[1]{\draw[#1](1/2,+\v*3/2)--++(1,-\v)--++(2,0);}
\newcommand{\Llelb}[1]{\draw[#1](5/2,-\v/2)--++(-1,\v)--++(-2,0);}
\newcommand{\Lrdia}[1]{\draw[#1](3/2,-\v/2)--++(2,+\v*2);}
\newcommand{\Lldia}[1]{\draw[#1](0,0)--++(2,+\v*2);}
\newcommand{\threetile}[4]{
    \draw[gray](0,0)--node[below,black]{\smfont$#2$}++(2,0)--node[right,black]{\smfont$#3$}++(1,\v)--++(-2,0)--node[left,black]{\smfont$#1$}cycle;
    \draw(-1,\v)node[above]{#4};
}
\newcommand{\clipx}[1]{
    \begin{scope}
        \clip(0,0)--++(2,0)--++(1,\v)--++(-2,0)--cycle;
        #1
    \end{scope}
}

\newcommand{\uptri}[5]{
    \begin{scope}
        \clip(0,0)--++(2,0)--++(-1,\v)--cycle;
        #5
    \end{scope}
    \draw[gray](0,0)--node[below,black]{\smfont$#2$}++(2,0)--++(-1,\v)--node[left,black]{\smfont$#1$}cycle;
    \draw(5/4,+\v*3/4)node[right]{\smfont$#3$};
    \draw(7/4,+\v/4)node[right]{\smfont$\color{purple2}\boldsymbol{#4}$};
}
\newcommand{\downtri}[5]{
    \begin{scope}
        \clip(0,\v)--++(2,0)--++(-1,-\v)--cycle;
        #5
    \end{scope}
    \draw[gray](0,\v)--node[above,black]{\smfont$#1$}++(2,0)--node[right,black]{\smfont$#2$}++(-1,-\v)--cycle;
    \draw(1/4,+\v*3/4)node[left]{\smfont$#3$};
    \draw(3/4,+\v/4)node[left]{\smfont$\color{purple2}\boldsymbol{#4}$};
}
\newcommand{\prismflip}[1]{
    \begin{tikzpicture}[x=#1,y=#1]
        \filldraw[gray fill](0,0)--++(6,-\v*6)--++(8,0)--++(12,+\v*12)--++(-6,+\v*6)--++(-8,0)--cycle;
        \foreach \i in {0,...,11} \draw[gray](\i,+\v*\i)--++(8,0)--++(6,-\v*6);
        \foreach \i in {1,...,4} \draw[gray]({2*\i+6},-\v*6)--++(-6,+\v*6)--++(12,+\v*12);
        \foreach \i in {1,...,5} \draw[gray](\i,-\v*\i)--++(8,0)--++(12,+\v*12);
        \draw[purple outline](6,0)--++(1,-\v)--++(2,0)--++(12,+\v*12)--++(-1,\v)--++(-2,0)--cycle;
        
        \begin{scope}[shift={(34,0)}]
            \filldraw[gray fill](0,0)--++(6,-\v*6)--++(8,0)--++(12,+\v*12)--++(-6,+\v*6)--++(-8,0)--cycle;
            \foreach \i in {1,...,11} \draw[gray](\i,+\v*\i)--++(6,0)--++(1,-\v)--++(2,0)--++(5,-\v*5);
            \foreach \i in {1,...,3} \draw[gray]({2*\i+6},-\v*6)--++(-6,+\v*6)--++(12,+\v*12);
            \foreach \i in {1,...,5} \draw[gray](\i,-\v*\i)--++(8,0)--++(12,+\v*12);
            \draw[gray](14,-\v*6)--++(-5,+\v*5);
            \draw[gray](0,0)--++(6,0);
            \draw[gray](7,-\v)--++(12,+\v*12)--++(-1,\v);
            \draw[gray](19,+\v*11)--++(2,0);
            \draw[purple outline](4,0)--++(1,-\v)--++(2,0)--++(12,+\v*12)--++(-1,\v)--++(-2,0)--cycle;
        \end{scope}
        
        \begin{scope}[shift={(68,0)}]
            \filldraw[gray fill](0,0)--++(6,-\v*6)--++(8,0)--++(12,+\v*12)--++(-6,+\v*6)--++(-8,0)--cycle;
            \foreach \i in {1,...,11} \draw[gray](\i,+\v*\i)--++(4,0)--++(1,-\v)--++(4,0)--++(5,-\v*5);
            \foreach \i in {1,...,2} \draw[gray]({2*\i+6},-\v*6)--++(-6,+\v*6)--++(12,+\v*12);
            \foreach \i in {1,...,5} \draw[gray](\i,-\v*\i)--++(8,0)--++(12,+\v*12);
            \draw[gray](12,-\v*6)--++(-5,+\v*5);
            \draw[gray](14,-\v*6)--++(-5,+\v*5);
            \draw[gray](0,0)--++(4,0);
            \draw[gray](5,-\v)--++(12,+\v*12)--++(-1,\v);
            \draw[gray](7,-\v)--++(12,+\v*12)--++(-1,\v);
            \draw[gray](17,+\v*11)--++(4,0);
            \draw[purple outline](2,0)--++(1,-\v)--++(2,0)--++(12,+\v*12)--++(-1,\v)--++(-2,0)--cycle;
        \end{scope}
        
        \begin{scope}[shift={(110,0)}]
            \filldraw[gray fill](0,0)--++(6,-\v*6)--++(8,0)--++(12,+\v*12)--++(-6,+\v*6)--++(-8,0)--cycle;
            \foreach \i in {1,...,12} \draw[gray](\i,+\v*\i)--++(6,-\v*6)--++(8,0);
            \foreach \i in {0,...,3} \draw[gray]({2*\i+6},-\v*6)--++(12,+\v*12)--++(-6,+\v*6);
            \foreach \i in {1,...,5} \draw[gray](\i,-\v*\i)--++(12,+\v*12)--++(8,0);
        \end{scope}
        
        \draw(30,+\v*3)node{$=$};
        \draw(64,+\v*3)node{$=$};
        \draw(102,+\v*3)node{$=\cdots=$};
    \end{tikzpicture}
}

\newcommand{\extracolumnlemma}[1]{
    \begin{tikzpicture}[x=#1,y=#1]
        \filldraw[light gray](0,0)--++(1,-\v)--++(2,0)--++(4,+\v*4)--++(-1,\v)--++(-2,0)--cycle;
        \draw[gray](0,0)--++(1,-\v)--++(2,0)--++(2,+\v*2)--++(-1,\v)--++(-2,0)--cycle;
        \draw[gray](3,+\v*3)--++(2,0)--++(1,-\v)--++(1,+\v)--++(-1,\v)--++(-2,0)--cycle;
        \draw[gray](2,0)--++(2,+\v*2);
        \draw[gray](5,+\v*3)--++(1,+\v);
        \draw[gray](0,0)--++(2,0)--++(1,-\v);
        \draw[gray](1,+\v)--++(2,0)--++(1,-\v);
        \draw[gray](2,+\v*2)--++(2,0)--++(1,-\v);
        \draw[gray, dashed](2,+\v*2)--++(1,\v);
        \draw[gray, dashed](4,+\v*2)--++(1,\v);
        \draw[gray, dashed](5,\v)--++(1,\v);
        
        \draw(1/2,+\v*1/2)node[left]{\smfont$q_1$};
        \draw(3/2,+\v*3/2)node[left]{\smfont$q_2$};
        \draw(7/2,+\v*7/2)node[left]{\smfont$q_m$};
        \draw(7/2,-\v*1/2)node[right]{\smfont$t_1$};
        \draw(9/2,+\v*1/2)node[right]{\smfont$t_2$};
        \draw(13/2,+\v*5/2)node[right]{\smfont$t_m$};
        \draw(5,+\v*4)node[above]{\smfont$v$};
        \draw(2,-\v)node[below]{\smfont$s$};
        \draw(1/4,-\v/4)node[left]{\smfont$r$};
        \draw(3/4,-\v*3/4)node[left]{\smfont\color{purple3}$\boldsymbol{0}$};
        \draw(25/4,+\v*15/4)node[right]{\smfont$u$};
        \draw(27/4,+\v*13/4)node[right]{\smfont\color{purple3}$\boldsymbol{0}$};
        \begin{scope}[shift={(13,0)}]
            \filldraw[light gray](0,0)--++(1,-\v)--++(2,0)--++(4,+\v*4)--++(-1,\v)--++(-2,0)--cycle;
            \draw[gray](0,0)--++(1,-\v)--++(2,0)--++(2,+\v*2)--++(-2,0)--++(-1,\v)--cycle;
            \draw[gray](3,+\v*3)--++(1,-\v)--++(2,0)--++(1,+\v)--++(-1,\v)--++(-2,0)--cycle;
            \draw[gray](1,-\v)--++(2,+\v*2);
            \draw[gray](4,+\v*2)--++(1,+\v);
            \draw[gray](1,+\v)--++(1,-\v)--++(2,0);
            \draw[gray](4,+\v*4)--++(1,-\v)--++(2,0);
            \draw[gray, dashed](2,+\v*2)--++(1,\v);
            \draw[gray, dashed](3,+\v)--++(1,\v);
            \draw[gray, dashed](5,\v)--++(1,\v);
            
            \draw(1/2,+\v*1/2)node[left]{\smfont$q_1$};
            \draw(3/2,+\v*3/2)node[left]{\smfont$q_2$};
            \draw(7/2,+\v*7/2)node[left]{\smfont$q_m$};
            \draw(7/2,-\v*1/2)node[right]{\smfont$t_1$};
            \draw(9/2,+\v*1/2)node[right]{\smfont$t_2$};
            \draw(13/2,+\v*5/2)node[right]{\smfont$t_m$};
            \draw(5,+\v*4)node[above]{\smfont$v$};
            \draw(2,-\v)node[below]{\smfont$s$};
            \draw(1/4,-\v/4)node[left]{\smfont$r$};
            \draw(3/4,-\v*3/4)node[left]{\smfont\color{purple3}$\boldsymbol{a}$};
            \draw(25/4,+\v*15/4)node[right]{\smfont$u$};
            \draw(27/4,+\v*13/4)node[right]{\smfont\color{purple3}$\boldsymbol{0}$};
            \draw(-1,+\v*3/2)node[left]{$\displaystyle\sum_{{\color{purple3}\boldsymbol{a}}=0}^n$};
        \end{scope}
        
        \draw(9,+\v*3/2)node{$=$};
    \end{tikzpicture}
}

\newcommand{\allowedglue}[1]{
    \begin{tikzpicture}[x=#1,y=#1]
        \begin{scope}
            \clip(0,0)--++(1,-\v)--++(4,0)--++(1,\v)--++(-1,\v)--++(-1,-\v)--cycle;
            \draw[blue1 line](3,-\v*2)\ddia\ddia\ddia;
            \draw[blue2 line](5,-\v*2)\ddia\ddia\ddia;
            \draw[blue3 line](6,-\v)\ddia\ddia\ddia;
        \end{scope}
        \draw[gray](0,0)--node[left,black]{\smfont$0$}++(1,-\v)--node[below,black]{\smfont$i$}++(2,0)--node[below,black]{\smfont$j$}++(2,0)--node[right,black]{\smfont$k$}++(1,\v)--node[right,black]{\smfont$0$}++(-1,\v)--node[left,black]{\smfont$k$}++(-1,-\v)--node[above,black]{\smfont$j$}++(-2,0)--node[above,black]{\smfont$i$}cycle;
        \draw[gray](2,0)--++(1,-\v);
        \draw[gray](4,0)--++(1,-\v);
    \end{tikzpicture}
}
\def\trihspace{5}
\newcommand{\Atriangles}{
    \uptri{0}{0}{0}{a}{}
    \begin{scope}[shift={(\trihspace,0)}]
        \uptri{0}{+}{+}{a}{
            \draw[red line](0,-\v)\vver\vver;
        }
    \end{scope}
    \begin{scope}[shift={(2*\trihspace,0)}]
        \uptri{i}{+}{i^+}{0}{
            \draw[blue1 line](0,+\v/2)\rhhor\rhhor;
            \draw[red line](0,-\v)\vver\vver;
        }
    \end{scope}
    \begin{scope}[shift={(0,-5)}]
        \downtri{0}{0}{0}{0}{}
        \begin{scope}[shift={(\trihspace,0)}]
            \downtri{+}{0}{+}{0}{
                \draw[red line](-1/2,-\v/2)\vver\vver;
            }
        \end{scope}
        \begin{scope}[shift={(2*\trihspace,0)}]
            \downtri{+}{i}{i^+}{0}{
                \draw[blue1 line](0,+\v/2)\rhhor\rhhor;
                \draw[red line](-1/2,-\v/2)\vver\vver;
            }
        \end{scope}
    \end{scope}
    \draw[rounded corners,blue6,line width=\colourwidth/2](-3/2,-6) rectangle (27/2,{\vx+1});
    \draw(6,{\vx+1})node[above]{\BB{A}};
}
\newcommand{\Btriangles}{
    \uptri{i}{i}{0}{b}{
        \draw[blue1 line](2,-\v)\ddia\ddia;
    }
    \begin{scope}[shift={(\trihspace,0)}]
        \uptri{i^+}{i}{+}{b}{
            \draw[blue1 line](2,-\v)\ddia\ddia;
            \draw[red line](0,+\v/2)\rhhor\rhhor;
        }
    \end{scope}
    \begin{scope}[shift={(2*\trihspace,0)}]
        \uptri{g}{f}{fg}{0}{
            \draw[blue1 line](0,-\v)\vver\vver;
            \draw[blue3 line](0,+\v/2)\rhhor\rhhor;
        }
    \end{scope}
    \begin{scope}[shift={(3*\trihspace,0)}]
        \uptri{g^+}{f}{fg^+}{0}{
            \draw[blue1 line](0,-\v)\vver\vver;
            \double{0}{+\v/2}{blue3 line}{2}
        }
    \end{scope}
    \begin{scope}[shift={(0,-5)}]
        \downtri{i}{i}{0}{i}{
            \draw[blue1 line](2,0)\ddia\ddia;
        }
        \begin{scope}[shift={(\trihspace,0)}]
            \downtri{i}{i^+}{+}{i}{
                \draw[blue1 line](2,0)\ddia\ddia;
                \draw[red line](0,+\v/2)\rhhor\rhhor;
            }
        \end{scope}
        \begin{scope}[shift={(2*\trihspace,0)}]
            \downtri{f}{g}{fg}{0}{
                \draw[blue1 line](-1/2,-\v/2)\vver\vver;
                \draw[blue3 line](0,+\v/2)\rhhor\rhhor;
            }
        \end{scope}
        \begin{scope}[shift={(3*\trihspace,0)}]
            \downtri{f}{g^+}{fg^+}{0}{
                \draw[blue1 line](-1/2,-\v/2)\vver\vver;
                \double{0}{+\v/2}{blue3 line}{2};
            }
        \end{scope}
    \end{scope}
    \draw[rounded corners,blue6,line width=\colourwidth/2](-3/2,-6) rectangle (37/2,{\vx+1});
    \draw(17/2,{\vx+1})node[above]{\BB{B}};
}
\newcommand{\Ctriangles}{
    \uptri{0}{i}{i}{0}{
        \filldraw[blue weight](0,0)--++(2,0)--++(-1,\v)--cycle;
        \draw[blue1 line](0,-\v)\vver\vver;
    }
    \begin{scope}[shift={(0,-5)}]
        \downtri{i}{0}{i}{0}{
            \filldraw[blue weight](0,\v)--++(1,-\v)--++(1,\v)--cycle;
            \draw[blue1 line](-1/2,-\v/2)\vver\vver;
        }
    \end{scope}
    \draw[rounded corners,blue6,line width=\colourwidth/2](-3/2,-6) rectangle (7/2,{\vx+1});
    \draw(1,{\vx+1})node[above]{\BB{C}};
}
\newcommand{\Dtriangles}{
    \uptri{i}{i}{0}{a}{
        \draw[blue1 line](2,-\v)\ddia\ddia;
    }
    \begin{scope}[shift={(\trihspace,0)}]
        \uptri{i^+}{i}{+}{a}{
            \draw[blue1 line](2,-\v)\ddia\ddia;
            \draw[red line](0,+\v/2)\rhhor\rhhor;
        }
    \end{scope}
    \begin{scope}[shift={(2*\trihspace,0)}]
    \uptri{+}{+}{0}{a}{
        \filldraw[red weight](0,0)--++(2,0)--++(-1,\v)--cycle;
        \draw[red line](2,-\v)\ddia\ddia;
    }
    \end{scope}
    
    \begin{scope}[shift={(0,-5)}]
        \downtri{i}{i}{0}{0}{
            \draw[blue1 line](2,0)\ddia\ddia;
        }
        \begin{scope}[shift={(\trihspace,0)}]
            \downtri{i}{i^+}{+}{0}{
                \draw[blue1 line](2,0)\ddia\ddia;
                \draw[red line](0,+\v/2)\rhhor\rhhor;
            }
        \end{scope}
        \begin{scope}[shift={(2*\trihspace,0)}]
        \downtri{+}{+}{0}{0}{
            \filldraw[red weight](0,\v)--++(1,-\v)--++(1,\v)--cycle;
            \draw[red line](2,0)\ddia\ddia;
        }
        \end{scope}
    \end{scope}
    \draw[rounded corners,blue6,line width=\colourwidth/2](-3/2,-6) rectangle (27/2,{\vx+1});
    \draw(6,{\vx+1})node[above]{\BB{D}};
}

\newcommand{\cuttriangles}[1]{
    \begin{tikzpicture}[x=#1,y=#1]
        \begin{scope}[shift={(0,0)}] \Atriangles \end{scope}
        \begin{scope}[shift={(15,0)}] \Btriangles \end{scope}
        \begin{scope}[shift={(35,0)}] \Ctriangles \end{scope}
        \begin{scope}[shift={(0,-11)}] \Dtriangles \end{scope}
    \end{tikzpicture}
}
\newcommand{\singletile}[1]{
    \begin{tikzpicture}[x=#1,y=#1]
        \draw[gray fill](0,0)--node[below]{\smfont$c$}++(2,0)--node[right]{\smfont$b$}++(1,\v)--++(-2,0)--node[left]{\smfont$a$}cycle;
    \end{tikzpicture}
}

\def\tchspace{6}
\def\tcvspace{-5}
\newcommand{\tilecases}[1]{
    \begin{tikzpicture}[x=#1,y=#1]
        \threetile{0}{0}{0}{1.}
        \begin{scope}[shift={(\tchspace,0)}]
            \clipx{
                \draw[blue1 line](3,0)\ddia\dia;
            }
            \threetile{0}{0}{b}{2.}
        \end{scope}
        \begin{scope}[shift={(\tchspace*2,0)}]
            \clipx{
                \draw[blue1 line](3/2,-\v/2)\ddia\dia;
            }
            \threetile{c}{c}{0}{3.}
        \end{scope}
        \begin{scope}[shift={(\tchspace*3,0)}]
            \clipx{
                \blueweight{1}{1}
                \draw[blue1 line](1/2,-\v/2)\vver\vver;
            }
            \threetile{0}{c}{0}{4.}
        \end{scope}
        \begin{scope}[shift={(\tchspace*4,0)}]
            \clipx{
                \draw[blue1 line](3,0)\ddia\dia;
                \draw[blue3 line](3/2,-\v/2)\ddia\dia;
            }
            \threetile{c}{c}{b}{5.}
        \end{scope}
        \begin{scope}[shift={(\tchspace*5,0)}]
            \clipx{
                \draw[blue1 line](1/2,-\v/2)\vver\vver;
                \draw[blue3 line](0,+\v/2)\rhhor\rhhor;
            }
            \threetile{b}{c}{b}{6.}
        \end{scope}
        
        \begin{scope}[shift={(0,\tcvspace)}]
            \threetile{\color{red}\boldsymbol{a}}{\color{red}\boldsymbol{0}}{0}{7.}
            \begin{scope}[shift={(\tchspace,0)}]
                \threetile{\color{red}\boldsymbol{a}}{\color{red}\boldsymbol{0}}{a}{8.}
            \end{scope}
            \begin{scope}[shift={(\tchspace*2,0)}]
                \threetile{\color{red}\boldsymbol{a}}{\color{red}\boldsymbol{0}}{b}{9.}
            \end{scope}
            \begin{scope}[shift={(\tchspace*3,0)}]
                \threetile{0}{\color{red}\boldsymbol{c}}{\color{red}\boldsymbol{c}}{10.}
            \end{scope}
            \begin{scope}[shift={(\tchspace*4,0)}]
                \threetile{a}{\color{red}\boldsymbol{c}}{\color{red}\boldsymbol{c}}{11.}
            \end{scope}
            \begin{scope}[shift={(\tchspace*5,0)}]
                \threetile{c}{\color{red}\boldsymbol{c}}{\color{red}\boldsymbol{c}}{12.}
            \end{scope}
        \end{scope}
        
        \begin{scope}[shift={(0,\tcvspace*2)}]
            \threetile{\color{red}\boldsymbol{a}}{\color{red}\boldsymbol{c}}{\color{red}\boldsymbol{0}}{13.}
            \begin{scope}[shift={(\tchspace,0)}]
                \threetile{\color{red}\boldsymbol{0}}{\color{red}\boldsymbol{c}}{\color{red}\boldsymbol{b}}{14.}
            \end{scope}
            \begin{scope}[shift={(\tchspace*2,0)}]
                \threetile{\color{red}\boldsymbol{a}}{\color{red}\boldsymbol{c}}{\color{red}\boldsymbol{b}}{15.}
            \end{scope}
        \end{scope}
    \end{tikzpicture}
}

\newcommand{\finalcaseb}[1]{
    \begin{tikzpicture}[x=#1,y=#1]
        \blueweight{1}{5}
        \clipcol{13}{
            \draw[blue2 line](13/2,+\v*3/2)\hhor\hor\dia\vver\vver\vver\vver\vver\vver\vver\vver\ddia;
            \draw[blue3 line](9/2,-\v/2)\hhor\hor\ddia\dia;
            \draw[blue3 line](11/2,+\v/2)\hhor\hor\ddia\dia;
            \draw[blue4 line](15/2,+\v*5/2)\hhor\hhor\hhor;
            \draw[blue4 line](17/2,+\v*7/2)\hhor\hhor\hhor;
            \draw[blue3 line](9,+\v*4)\ddia\dia\hhor\hor;
            \draw[blue3 line](10,+\v*5)\ddia\dia\hhor\hor;
            \draw[blue3 line](11,+\v*6)\ddia\dia\hhor\hor;
            \draw[blue1 line](12,+\v*7)\ddia\dia\hhor\hor;
            \draw[blue1 line](13,+\v*8)\ddia\dia\hhor\hor;
            \draw[blue1 line](14,+\v*9)\ddia\ddia\ddia;
            \draw[blue1 line](15,+\v*10)\ddia\ddia\ddia;
            \draw[blue1 line](16,+\v*11)\ddia\ddia\ddia;
            \draw[red line](1,-\v*2)\vver\vver\vver\vver\vver\vver\ver\rhhor;
            \draw[red line,dotted](21/2,+\v*11/2)\hhor\hhor\hhor;
            \draw[red line,dotted](23/2,+\v*13/2)\hhor\hhor\hhor;
            \draw[red line,dotted](25/2,+\v*15/2)\hhor\hhor\hhor;
            \draw[red line,dotted](27/2,+\v*17/2)\hhor\hhor\hhor;
            \draw[red line,dotted](29/2,+\v*19/2)\hhor\hhor\hhor;
            \draw[red line,dotted](31/2,+\v*21/2)\hhor\hhor\hhor;
            \draw[red line,dotted](33/2,+\v*23/2)\hhor\hhor\hhor;
            \draw[red line,dotted](35/2,+\v*25/2)\hhor\hhor\hhor;
        }
        \BLOCKC{0}{r}{s}{}{}
        \LDCHEVC{0}
        \LCHEVC{1}{q_{h-1}}{}{}{}
        \LCHEVC{2}{}{t_h}{}{}
        \LDCHEVC{3}
        \LCHEVC{4}{q_i}{}{}{}
        \LCHEVC{5}{}{}{}{}
        \LDCHEVC{6}
        \LCHEVC{7}{q_j}{}{}{}
        \LCHEVC{8}{}{t_{j+1}}{}{}
        \LDCHEVC{9}
        \LCHEVC{10}{q_k}{}{}{}
        \LCHEVC{11}{q_{k+1}}{}{}{}
        \LDCHEVC{12}
        \draw[color=white,fill=white](3/2,+\v/2) circle (\circsize);
        \draw[color=white,fill=white](9,+\v*6) circle (\circsize);
        \draw[purple outline](0,0)--++(2,+\v*2)--++(1,-\v)--++(-2,-\v*2)--cycle;
        \draw[purple outline](7,+\v*5)--++(3,+\v*3)--++(1,-\v)--++(-3,-\v*3)--cycle;
        \whitecirc{1/2}{-\v/2}
        \LWTE{4}
        \RWRD{13}
        \draw(14,+\v*13)node[above]{\smfont$v$};
        \draw(31/2,+\v*25/2)node[right]{\smfont$u$};
        \draw(3/2,+\v/2)node{\BB{A}};
        \draw(9,+\v*6)node{\BB{B}};
        
        \begin{scope}[shift={(20,0)}]
            \filldraw[red weight](0,0)--++(1,-\v)--++(2,0)--++(-1,\v)--cycle;
            \clipcol{13}{
                \draw[blue2 line](6,\v)\ddia\vver\vver\vver\vver\vver\ddia\vver\vver\ddia;
                \draw[blue3 line](4,-\v)\ddia\dia\hhor\hor;
                \draw[blue3 line](5,0)\ddia\dia\hhor\hor;
                \draw[blue3 line](19/2,+\v*9/2)\hhor\hor\ddia\dia;
                \draw[blue3 line](21/2,+\v*11/2)\hhor\hor\ddia\dia;
                \draw[blue3 line](23/2,+\v*13/2)\hhor\hor\ddia\dia;
                \draw[blue4 line](15/2,+\v*5/2)\hhor\hhor\hhor;
                \draw[blue4 line](17/2,+\v*7/2)\hhor\hhor\hhor;
                \draw[blue1 line](12,+\v*7)\ddia\dia\hhor\hor;
                \draw[blue1 line](13,+\v*8)\ddia\dia\hhor\hor;
                \draw[blue1 line](14,+\v*9)\ddia\ddia\ddia;
                \draw[blue1 line](15,+\v*10)\ddia\ddia\ddia;
                \draw[blue1 line](16,+\v*11)\ddia\ddia\ddia;
                \draw[red line](3,-\v*2)\ddia\ddia\vver\vver\vver\vver\ver\rhhor\rhhor;
                \draw[red line,dotted](21/2,+\v*11/2)\hhor\hhor\hhor;
                \draw[red line,dotted](23/2,+\v*13/2)\hhor\hhor\hhor;
                \draw[red line,dotted](25/2,+\v*15/2)\hhor\hhor\hhor;
                \draw[red line,dotted](27/2,+\v*17/2)\hhor\hhor\hhor;
                \draw[red line,dotted](29/2,+\v*19/2)\hhor\hhor\hhor;
                \draw[red line,dotted](31/2,+\v*21/2)\hhor\hhor\hhor;
                \draw[red line,dotted](33/2,+\v*23/2)\hhor\hhor\hhor;
                \draw[red line,dotted](35/2,+\v*25/2)\hhor\hhor\hhor;
                \clipdouble{6}{+\v*4}{blue3 line}
            }
            \BLOCKC{0}{r}{s}{}{}
            \LDCHEVC{0}
            \LCHEVC{1}{q_{h-1}}{}{}{}
            \LCHEVC{2}{}{t_h}{}{}
            \LDCHEVC{3}
            \LCHEVC{4}{q_i}{}{}{}
            \LCHEVC{5}{}{}{}{}
            \LDCHEVC{6}
            \LCHEVC{7}{q_j}{}{}{}
            \LCHEVC{8}{}{t_{j+1}}{}{}
            \LDCHEVC{9}
            \LCHEVC{10}{q_k}{}{}{}
            \LCHEVC{11}{q_{k+1}}{}{}{}
            \LDCHEVC{12}
            \draw[color=white,fill=white](7/2,+\v/2) circle (\circsize);
            \draw[color=white,fill=white](7,+\v*6) circle (\circsize);
            \draw[purple outline](2,0)--++(2,+\v*2)--++(1,-\v)--++(-2,-\v*2)--cycle;
            \draw[purple outline](5,+\v*5)--++(3,+\v*3)--++(1,-\v)--++(-3,-\v*3)--cycle;
            \whitecirc{1/2}{-\v/2}
            \LWTE{4}
            \RWRD{13}
            \draw(14,+\v*13)node[above]{\smfont$v$};
            \draw(31/2,+\v*25/2)node[right]{\smfont$u$};
            \draw(7/2,+\v/2)node{\BB{A}};
            \draw(7,+\v*6)node{\BB{B}};
        \end{scope}
        
        \draw(18,+\v*6)node{$+$};
        \draw(36,+\v*6)node{$=$};
        \draw(38,+\v*6)node{$0$};
    \end{tikzpicture}
}

\newcommand{\unweightedproof}[1]{
    \begin{tikzpicture}[x=#1,y=#1]
        \clipcol{9}{
            \draw[blue2 line](9,+\v*4)\ddia\ddia\vver\vver\vver\ver;
            \draw[blue1 line](3,-\v*2)\ddia\ddia\vver\vver\vver\ver;
            \draw[blue1 line](10,+\v*5)\ddia\dia\hhor\hor;
            \draw[blue1 line](11,+\v*6)\ddia\dia\hhor\hor;
            \draw[blue1 line](12,+\v*7)\ddia\dia\hhor\hor;
            \draw[blue2 line](4,-\v)\ddia\dia\hhor\hor;
            \draw[blue2 line](5,0)\ddia\dia\hhor\hor;
            \draw[blue2 line](6,\v)\ddia\dia\hhor\hor;
            \draw[red line,dotted](9/2,-\v/2)\hhor\hhor\hhor;
            \draw[red line,dotted](11/2,+\v/2)\hhor\hhor\hhor;
            \draw[red line,dotted](13/2,+\v*3/2)\hhor\hhor\hhor;
            \draw[red line,dotted](15/2,+\v*5/2)\hhor\hhor\hhor;
            \draw[red line,dotted](21/2,+\v*11/2)\hhor\hhor\hhor;
            \draw[red line,dotted](23/2,+\v*13/2)\hhor\hhor\hhor;
            \draw[red line,dotted](25/2,+\v*15/2)\hhor\hhor\hhor;
            \draw[red line,dotted](27/2,+\v*17/2)\hhor\hhor\hhor;
        }
        \LCUT{6}
        \BLOCKC{0}{r}{s}{}{}
        \LCHEVC{0}{q_1}{t_1}{}{}
        \LDCHEVC{1}
        \LCHEVC{2}{q_i}{t_i}{}{}
        \LCHEV{3}{}{}{}{}
        \LDCHEV{4}
        \LCHEVC{5}{q_{j-1}}{}{}{}
        \LCHEVC{6}{q_j}{t_j}{}{}
        \LDCHEVC{7}
        \LCHEVC{8}{q_m}{t_m}{v}{u}
        
        \draw(13/2,+\v*5/2)node[right]{\smfont$t_{i+1}$};
        \draw[color=white,fill=white](4,\v) circle (\circsize);
        \draw[color=white,fill=white](10,+\v*7) circle (\circsize);
        \draw[purple outline](2,0)--++(3,+\v*3)--++(1,-\v)--++(-3,-\v*3)--cycle;
        \draw[purple outline](8,+\v*6)--++(3,+\v*3)--++(1,-\v)--++(-3,-\v*3)--cycle;
        \draw(4,\v)node{\BB{A}};
        \draw(10,+\v*7)node{\BB{B}};
        \draw(7,+\v*4)node{\BB{C}};
        \LWRD{0}
        \RWRD{9}
        
        \begin{scope}[shift={(16,0)}]
            \clipcol{9}{
                \draw[blue2 line](15/2,+\v*9/2)\vver\vver\vver\ver\ddia\ddia;
                \draw[blue1 line](1,-\v*2)\vver\vver\vver\vver\ddia\dia;
                \draw[blue1 line](21/2,+\v*11/2)\hhor\hor\ddia\dia;
                \draw[blue1 line](23/2,+\v*13/2)\hhor\hor\ddia\dia;
                \draw[blue1 line](25/2,+\v*15/2)\hhor\hor\ddia\dia;
                \draw[blue2 line](9/2,-\v/2)\hhor\hor\ddia\dia;
                \draw[blue2 line](11/2,+\v/2)\hhor\hor\ddia\dia;
                \draw[blue2 line](13/2,+\v*3/2)\hhor\hor\ddia\dia;
                \draw[red line,dotted](9/2,-\v/2)\hhor\hhor\hhor;
                \draw[red line,dotted](11/2,+\v/2)\hhor\hhor\hhor;
                \draw[red line,dotted](13/2,+\v*3/2)\hhor\hhor\hhor;
                \draw[red line,dotted](15/2,+\v*5/2)\hhor\hhor\hhor;
                \draw[red line,dotted](21/2,+\v*11/2)\hhor\hhor\hhor;
                \draw[red line,dotted](23/2,+\v*13/2)\hhor\hhor\hhor;
                \draw[red line,dotted](25/2,+\v*15/2)\hhor\hhor\hhor;
                \draw[red line,dotted](27/2,+\v*17/2)\hhor\hhor\hhor;
            }
            \RCUT{4}
            \RCHEVC{0}{q_1}{t_1}{r}{s}
            \RDCHEVC{1}
            \RCHEVC{2}{q_i}{t_i}{}{}
            \RCHEVC{3}{}{t_{i+1}}{}{}
            \RDCHEV{4}
            \RCHEV{5}{}{}{}{}
            \RCHEVC{6}{q_j}{t_j}{}{}
            \RDCHEVC{7}
            \RCHEVC{8}{q_m}{t_m}{}{}
            \BLOCKC{9}{}{}{v}{u}
            
            \draw(11/2,+\v*11/2)node[left]{\smfont$q_{j-1}$};
            \draw[color=white,fill=white](2,\v) circle (\circsize);
            \draw[color=white,fill=white](8,+\v*7) circle (\circsize);
            \draw[purple outline](0,0)--++(3,+\v*3)--++(1,-\v)--++(-3,-\v*3)--cycle;
            \draw[purple outline](6,+\v*6)--++(3,+\v*3)--++(1,-\v)--++(-3,-\v*3)--cycle;
            \draw(2,\v)node{\BB{A}};
            \draw(8,+\v*7)node{\BB{B}};
            \draw(5,+\v*4)node{\BB{C}};
            \LWRD{0}
            \RWRD{9}
        \end{scope}
        
        \draw(14,+\v*4)node{$=$};
    \end{tikzpicture}
}

\newcommand{\zerohexa}[1]{
    \begin{tikzpicture}[x=#1,y=#1]
        \LCHEV{0}{q}{t}{v}{u}
        \BLOCK{0}{r}{s}{}{}
        \LWRD{0}
        \RWRD{1}
        
        \begin{scope}[shift={(7,0)}]
            \RCHEV{0}{q}{t}{r}{s}
            \BLOCK{1}{}{}{v}{u}
            \LWRD{0}
            \RWRD{1}
        \end{scope}
        
        \draw(11/2,0)node{$=$};
    \end{tikzpicture}
}

\newcommand{\zerohexb}[1]{
    \begin{tikzpicture}[x=#1,y=#1]
        \LCHEV{0}{q}{t}{v}{u}
        \BLOCK{0}{r}{s}{}{}
        \LWTE{0}
        \RWTE{0}
        
        \begin{scope}[shift={(7,0)}]
            \RCHEV{0}{q}{t}{r}{s}
            \BLOCK{1}{}{}{v}{u}
            \LWTE{0}
            \RWTE{0}
        \end{scope}
        
        \draw(11/2,0)node{$=$};
    \end{tikzpicture}
}

\newcommand{\skeletons}[1]{
    \begin{tikzpicture}[x=#1,y=#1]
        \begin{scope}[shift={(4,0)}]
            \drawrightshear{}{}{}{}{
                \Rver{blue1 line}
                \Rhor{blue3 line}
                \Rhor{red line, dotted}
            }
        \end{scope}
        
        \begin{scope}[shift={(12,0)}]
            \drawrightshear{}{}{}{}{
                \Rrdia{blue1 line}
                \Rldia{blue1 line}
                \Rhor{red line, dotted}
            }
        \end{scope}
        
        \draw(19/2,+\v/2)node{$\mapsto$};
        
        \begin{scope}[shift={(1/2,-\v*2)}]
            \begin{scope}[shift={(-4,0)}]
                \drawdiamond{}{}{}{}{
                    \Drelb{blue1 line}
                    \Dthor{blue3 line}
                    \Dthor{red line, dotted}
                    \Dbhor{red line, dotted}
                }
            \end{scope}
            
            \begin{scope}[shift={(0,0)}]
                \drawdiamond{}{}{}{}{
                    \Dlelb{blue1 line}
                    \Dthor{red line, dotted}
                    \Dbhor{blue3 line}
                    \Dbhor{red line, dotted}
                }
            \end{scope}
            
            \begin{scope}[shift={(4,0)}]
                \drawdiamond{}{}{}{}{
                    \Dldia{blue1 line}
                    \Dthor{blue2 line}
                    \Dthor{red line, dotted}
                    \Dbhor{blue3 line}
                    \Dbhor{red line, dotted}
                }
            \end{scope}
            
            \begin{scope}[shift={(12,0)}]
                \drawdiamond{}{}{}{}{
                    \Drdia{blue1 line}
                    \Dthor{red line, dotted}
                    \Dbhor{red line, dotted}
                }
            \end{scope}
            
            \draw(9,0)node{$\mapsto$};
        \end{scope}
    \end{tikzpicture}
}

\newcommand{\partskeletona}[1]{
    \begin{tikzpicture}[x=#1,y=#1]
        \filldraw[red weight](0,0)--++(1,-\v)--++(2,0)--++(-1,\v)--cycle;
        \clipcol{5}{
            \draw[blue1 line](4,-\v)\ddia\dia\hhor;
            \draw[blue1 line](5,0)\ddia\dia\hhor;
            \draw[blue1 line](6,+\v)\ddia\dia\hhor;
            \draw[blue1 line](7,+\v*2)\ddia\dia\hhor;
            \draw[blue1 line](8,+\v*3)\ddia\dia\hhor;
            \draw[red line](3,-\v*2)\ddia\ddia\vver\vver\vver\vver\vver\vver;
        }
        \BLOCKC{0}{r}{s}{}{}
        \LCHEVC{0}{q_1}{t_1}{}{}
        \LDCHEVC{1}
        \LDCHEVC{2}
        \LDCHEVC{3}
        \LCHEVC{4}{q_m}{t_m}{v}{u}
        \LWHT{0}
        \RWHT{5}
        
        \begin{scope}[shift={(12,0)}]
            \filldraw[red weight](5,+\v*5)--++(1,-\v)--++(2,0)--++(-1,\v)--cycle;
            \clipcol{5}{
                \draw[blue1 line](7/2,-\v/2)\hhor\ddia\ddia;
                \draw[blue1 line](9/2,+\v/2)\hhor\ddia\ddia;
                \draw[blue1 line](11/2,+\v*3/2)\hhor\ddia\ddia;
                \draw[blue1 line](13/2,+\v*5/2)\hhor\ddia\ddia;
                \draw[blue1 line](15/2,+\v*7/2)\hhor\ddia\ddia;
                \draw[red line](1,-\v*2)\vver\vver\vver\vver\vver\vver\ddia\ddia;
            }
            \RCHEVC{0}{q_1}{t_1}{r}{s}
            \RDCHEVC{1}
            \RDCHEVC{2}
            \RDCHEVC{3}
            \RCHEVC{4}{q_m}{t_m}{}{}
            \BLOCKC{5}{}{}{v}{u}
            \LWHT{0}
            \RWHT{5}
        \end{scope}
        
        \draw(10,+\v*2)node{$=$};
        \draw(47/2,+\v*13/2)node[right]{\phantom{\smfont$u$}};
        \draw(0,+\v*5)node[left]{(a)};
    \end{tikzpicture}
}

\newcommand{\partskeletonb}[1]{
    \begin{tikzpicture}[x=#1,y=#1]
        \filldraw[blue weight](2,+\v*2)--++(1,\v)--++(2,0)--++(-1,-\v)--cycle;
        \clipcol{5}{
            \draw[blue1 line](7/2,-\v/2)\hhor\ddia\ddia;
            \draw[blue1 line](9/2,+\v/2)\hhor\ddia\ddia;
            \draw[blue1 line](11/2,+\v*3/2)\hhor\dia\vver\ddia;
            \draw[blue1 line](7,+\v*2)\ddia\ddia\ddia;
            \draw[blue1 line](8,+\v*3)\ddia\ddia\ddia;
            \draw[red line](1,-\v*2)\vver\vver\vver\vver\ver\rhhor;
            \draw[red line,dotted](3,+\v*7/2)\rhhor\rhhor\rhhor;
            \draw[red line,dotted](4,+\v*9/2)\rhhor\rhhor\rhhor;
        }
        \BLOCKC{0}{r}{s}{}{}
        \LCHEVC{0}{q_1}{t_1}{}{}
        \LDCHEVC{1}
        \LCHEVC{2}{q_i}{}{}{}
        \LCHEVC{3}{}{}{}{}
        \LDCHEVC{4}
        \draw(6,+\v*5)node[above]{\smfont$v$};
        \draw(15/2,+\v*9/2)node[right]{\smfont$u$};
        \LWTE{2}
        \LWHT{0}
        \RWRD{5}
        
        \begin{scope}[shift={(12,0)}]
            \filldraw[red weight](0,0)--++(1,-\v)--++(2,0)--++(-1,\v)--cycle;
            \clipcol{5}{
                \draw[blue1 line](4,-\v)\ddia\dia\hhor;
                \draw[blue1 line](5,0)\ddia\dia\hhor;
                \draw[blue1 line](6,+\v)\ddia\ddia\ddia;
                \draw[blue1 line](7,+\v*2)\ddia\ddia\ddia;
                \draw[blue1 line](8,+\v*3)\ddia\ddia\ddia;
                \draw[red line](3,-\v*2)\ddia\ddia\vver\vver\ver\rhhor\rhhor;
                \draw[red line,dotted](3,+\v*7/2)\rhhor\rhhor\rhhor;
                \draw[red line,dotted](4,+\v*9/2)\rhhor\rhhor\rhhor;
            }
            \BLOCKC{0}{r}{s}{}{}
            \LCHEVC{0}{q_1}{t_1}{}{}
            \LDCHEVC{1}
            \LCHEVC{2}{q_i}{}{}{}
            \LCHEVC{3}{}{}{}{}
            \LDCHEVC{4}
            \draw(6,+\v*5)node[above]{\smfont$v$};
            \draw(15/2,+\v*9/2)node[right]{\smfont$u$};
            \LWTE{2}
            \LWHT{0}
            \RWRD{5}
        \end{scope}
        
        \draw(10,+\v*2)node{$+$};
        \draw(22,+\v*2)node{$=$};
        \draw(24,+\v*2)node{$0$};
        \draw(47/2,+\v*13/2)node[right]{\phantom{\smfont$u$}};
        \draw(0,+\v*5)node[left]{(b)};
    \end{tikzpicture}
}

\newcommand{\partskeletonc}[1]{
    \begin{tikzpicture}[x=#1,y=#1]
        \begin{scope}[shift={(4,0)}]
        \filldraw[blue weight](3,\v)--++(1,\v)--++(2,0)--++(-1,-\v)--cycle;
        \clipcol{5}{
            \draw[blue1 line](2,-\v)\ddia\ddia\ddia;
            \draw[blue1 line](4,-\v)\ddia\ddia\ddia;
            \draw[blue1 line](5,0)\ddia\vver\dia\hhor\hhor;
            \draw[blue1 line](7,+\v*2)\ddia\dia\hhor;
            \draw[blue1 line](8,+\v*3)\ddia\dia\hhor;
            \draw[red line,dotted](0,-\v/2)\rhhor\rhhor\rhhor;
            \draw[red line,dotted](0,+\v/2)\rhhor\rhhor\rhhor;
            \draw[red line](1/2,+\v*3/2)\rhhor\vver\vver\vver\vver;
        }
        \RDCHEVC{0}
        \RCHEVC{1}{}{}{}{}
        \RCHEVC{2}{}{t_i}{}{}
        \RDCHEVC{3}
        \RCHEVC{4}{q_m}{t_m}{}{}
        \BLOCKC{5}{}{}{v}{u}
        \draw(1/2,-\v/2)node[left]{\smfont$r$};
        \draw(2,-\v)node[below]{\smfont$s$};
        \RWTE{2}
        \LWRD{0}
        \RWHT{5}
        
        \begin{scope}[shift={(12,0)}]
            \filldraw[red weight](5,+\v*5)--++(1,-\v)--++(2,0)--++(-1,\v)--cycle;
            \clipcol{5}{
                \draw[blue1 line](2,-\v)\ddia\ddia\ddia;
                \draw[blue1 line](4,-\v)\ddia\ddia\ddia;
                \draw[blue1 line](5,0)\ddia\ddia\ddia;
                \draw[blue1 line](13/2,+\v*5/2)\hhor\ddia\ddia;
                \draw[blue1 line](15/2,+\v*7/2)\hhor\ddia\ddia;
                \draw[red line,dotted](0,-\v/2)\rhhor\rhhor\rhhor;
                \draw[red line,dotted](0,+\v/2)\rhhor\rhhor\rhhor;
                \draw[red line](1/2,+\v*3/2)\rhhor\rhhor\vver\vver\ver\ddia\ddia;
            }
            \RDCHEVC{0}
            \RCHEVC{1}{}{}{}{}
            \RCHEVC{2}{}{t_i}{}{}
            \RDCHEVC{3}
            \RCHEVC{4}{q_m}{t_m}{}{}
            \BLOCKC{5}{}{}{v}{u}
            \draw(1/2,-\v/2)node[left]{\smfont$r$};
            \draw(2,-\v)node[below]{\smfont$s$};
            \RWTE{2}
            \LWRD{0}
            \RWHT{5}
        \end{scope}
        \end{scope}
        
        \draw(14,+\v*2)node{$+$};
        \draw(2,+\v*2)node{$=$};
        \draw(0,+\v*2)node{$0$};
        \draw(47/2,+\v*13/2)node[right]{\phantom{\smfont$u$}};
        \draw(0,+\v*5)node[left]{(c)};
        \draw(0,-\v*7/2)node{};
    \end{tikzpicture}
}

\newcommand{\partskeletond}[1]{
    \begin{tikzpicture}[x=#1,y=#1]
        \filldraw[blue weight](4,+\v*4)--++(1,\v)--++(2,0)--++(-1,-\v)--cycle;
        \clipcol{7}{
            \draw[blue1 line](2,-\v)\ddia\ddia\ddia;
            \draw[blue1 line](4,-\v)\ddia\ddia\ddia;
            \draw[blue1 line](5,0)\ddia\ddia\ddia;
            \draw[blue1 line](13/2,+\v*5/2)\hhor\ddia\ddia;
            \draw[blue1 line](15/2,+\v*7/2)\hhor\dia\vver\ddia;
            \draw[blue1 line](9,+\v*4)\ddia\ddia\ddia;
            \draw[blue1 line](10,+\v*5)\ddia\ddia\ddia;
            \draw[red line,dotted](0,-\v/2)\rhhor\rhhor\rhhor;
            \draw[red line,dotted](0,+\v/2)\rhhor\rhhor\rhhor;
            \draw[red line,dotted](5,+\v*11/2)\rhhor\rhhor\rhhor;
            \draw[red line,dotted](6,+\v*13/2)\rhhor\rhhor\rhhor;
            \draw[red line](1/2,+\v*3/2)\rhhor\rhhor\vver\vver\vver\rhhor;
        }
        \BLOCKC{0}{r}{s}{}{}
        \LDCHEVC{0}
        \LCHEVC{1}{}{}{}{}
        \LCHEVC{2}{}{t_i}{}{}
        \LDCHEVC{3}
        \LCHEVC{4}{q_j}{}{}{}
        \LCHEVC{5}{}{}{}{}
        \LDCHEVC{6}
        \draw(8,+\v*7)node[above]{\smfont$v$};
        \draw(19/2,+\v*13/2)node[right]{\smfont$u$};
        \LWTE{4}
        \RWTE{2}
        \LWRD{0}
        \RWRD{7}
        
        \begin{scope}[shift={(14,0)}]
            \filldraw[blue weight](3,\v)--++(1,\v)--++(2,0)--++(-1,-\v)--cycle;
            \clipcol{7}{
                \draw[blue1 line](2,-\v)\ddia\ddia\ddia;
                \draw[blue1 line](4,-\v)\ddia\ddia\ddia;
                \draw[blue1 line](5,0)\ddia\vver\dia\hhor\hhor;
                \draw[blue1 line](7,+\v*2)\ddia\dia\hhor;
                \draw[blue1 line](8,+\v*3)\ddia\ddia\ddia;
                \draw[blue1 line](9,+\v*4)\ddia\ddia\ddia;
                \draw[blue1 line](10,+\v*5)\ddia\ddia\ddia;
                \draw[red line,dotted](0,-\v/2)\rhhor\rhhor\rhhor;
                \draw[red line,dotted](0,+\v/2)\rhhor\rhhor\rhhor;
                \draw[red line,dotted](5,+\v*11/2)\rhhor\rhhor\rhhor;
                \draw[red line,dotted](6,+\v*13/2)\rhhor\rhhor\rhhor;
                \draw[red line](1/2,+\v*3/2)\rhhor\vver\vver\vver\rhhor\rhhor;
            }
            \RDCHEVC{0}
            \RCHEVC{1}{}{}{}{}
            \RCHEVC{2}{}{t_i}{}{}
            \RDCHEVC{3}
            \RCHEVC{4}{q_j}{}{}{}
            \RDCHEVC{5}{}{}{}{}
            \RDCHEVC{6}
            \BLOCKC{7}{}{}{v}{u}
            \draw(1/2,-\v/2)node[left]{\smfont$r$};
            \draw(2,-\v)node[below]{\smfont$s$};
            \LWTE{4}
            \RWTE{2}
            \LWRD{0}
            \RWRD{7}
        \end{scope}
        
        \draw(12,+\v*3)node{$=$};
        \draw(0,+\v*7)node[left]{(d)};
    \end{tikzpicture}
}

\newcommand{\fullskeleton}[1]{
    \begin{tikzpicture}[x=#1,y=#1]
        \clipcol{5}{
            \draw[blue1 line](2,-\v)\ddia\ddia\ddia;
            \draw[blue1 line](4,-\v)\ddia\ddia\ddia;
            \draw[blue1 line](5,0)\ddia\ddia\ddia;
            \draw[blue1 line](6,+\v)\ddia\ddia\ddia;
            \draw[blue1 line](7,+\v*2)\ddia\ddia\ddia;
            \draw[blue1 line](8,+\v*3)\ddia\ddia\ddia;
            \draw[red line,dotted](0,-\v/2)\rhhor\rhhor\rhhor;
            \draw[red line,dotted](0,+\v/2)\rhhor\rhhor\rhhor;
            \draw[red line,dotted](1,+\v*3/2)\rhhor\rhhor\rhhor;
            \draw[red line,dotted](2,+\v*5/2)\rhhor\rhhor\rhhor;
            \draw[red line,dotted](3,+\v*7/2)\rhhor\rhhor\rhhor;
            \draw[red line,dotted](4,+\v*9/2)\rhhor\rhhor\rhhor;
        }
        \BLOCKC{0}{r}{s}{}{}
        \LCHEVC{0}{q_1}{t_1}{}{}
        \LDCHEVC{1}
        \LDCHEVC{2}
        \LDCHEVC{3}
        \LCHEVC{4}{q_m}{t_m}{v}{u}
        \LWRD{0}
        \RWRD{5}
        
        \begin{scope}[shift={(12,0)}]
            
            \clipcol{5}{
                \draw[blue1 line](2,-\v)\ddia\ddia\ddia;
                \draw[blue1 line](4,-\v)\ddia\ddia\ddia;
                \draw[blue1 line](5,0)\ddia\ddia\ddia;
                \draw[blue1 line](6,+\v)\ddia\ddia\ddia;
                \draw[blue1 line](7,+\v*2)\ddia\ddia\ddia;
                \draw[blue1 line](8,+\v*3)\ddia\ddia\ddia;
                \draw[red line,dotted](0,-\v/2)\rhhor\rhhor\rhhor;
                \draw[red line,dotted](0,+\v/2)\rhhor\rhhor\rhhor;
                \draw[red line,dotted](1,+\v*3/2)\rhhor\rhhor\rhhor;
                \draw[red line,dotted](2,+\v*5/2)\rhhor\rhhor\rhhor;
                \draw[red line,dotted](3,+\v*7/2)\rhhor\rhhor\rhhor;
                \draw[red line,dotted](4,+\v*9/2)\rhhor\rhhor\rhhor;
            }
            \RCHEVC{0}{q_1}{t_1}{r}{s}
            \RDCHEVC{1}
            \RDCHEVC{2}
            \RDCHEVC{3}
            \RCHEVC{4}{q_m}{t_m}{}{}
            \BLOCKC{5}{}{}{v}{u}
            \LWRD{0}
            \RWRD{5}
        \end{scope}
        
        \draw(10,+\v*2)node{$=$};
    \end{tikzpicture}
}

\newcommand{\horizontalcases}[1]{
    \begin{tikzpicture}[x=#1,y=#1]
        \clipcol{8}{
            \draw[blue1 line](2,+\v*5/2)\rhhor\rhhor\rhhor;
            \draw[blue1 line](3,+\v*7/2)\rhhor\rhhor\rhhor;
            \draw[blue1 line](4,+\v*9/2)\rhhor\rhhor\rhhor;
            \draw[blue1 line](10,+\v*5)\ddia\dia\hhor\hor;
            \draw[blue1 line](11,+\v*6)\ddia\dia\hhor\hor;
            \draw[blue1 line](9/2,-\v/2)\hhor\hor\ddia\dia;
            \draw[blue1 line](11/2,+\v/2)\hhor\hor\ddia\dia;
            \draw[red line](1/2,+\v*3/2)\rhhor\vver\vver\vver\vver\vver\vver\vver;
            \draw[red line](3/2,-\v*3/2)\vver\vver\vver\vver\vver\vver\vver\rhhor;
        }
        \BLOCKC{0}{r}{s}{}{}
        \LDCHEVC{0}
        \LCHEVC{1}{}{}{}{}
        \LCHEVC{2}{q_h}{t_h}{}{}
        \LDCHEVC{3}
        \LCHEVC{4}{q_{i-1}}{}{}{}
        \LCHEVC{5}{q_i}{t_i}{}{}
        \LCHEVC{6}{}{}{}{}
        \LDCHEVC{7}
        \LWTE{5}
        \RWTE{2}
        \LWHT{0}
        \RWHT{8}
        \draw(9,+\v*8)node[above]{\smfont$v$};
        \draw(21/2,+\v*15/2)node[right]{\smfont$u$};
        
        \begin{scope}[shift={(15,0)}]
            \clipcol{8}{
                \draw[blue1 line](2,+\v*5/2)\rhhor\rhhor\rhhor;
                \draw[blue1 line](3,+\v*7/2)\rhhor\rhhor\rhhor;
                \draw[blue1 line](4,+\v*9/2)\rhhor\rhhor\rhhor;
                \draw[blue1 line](10,+\v*5)\ddia\dia\hhor\hor;
                \draw[blue1 line](11,+\v*6)\ddia\dia\hhor\hor;
                \draw[blue1 line](9/2,-\v/2)\hhor\hor\ddia\dia;
                \draw[blue1 line](11/2,+\v/2)\hhor\hor\ddia\dia;
                \draw[red line](1/2,+\v*3/2)\rhhor\vver\vver\vver\vver\vver\vver\vver;
                \draw[red line](3/2,-\v*3/2)\vver\vver\vver\vver\vver\vver\vver\rhhor;
            }
            \RDCHEVC{0}
            \RCHEVC{1}{}{}{}{}
            \RCHEVC{2}{q_h}{t_h}{}{}
            \RDCHEVC{3}
            \RCHEVC{4}{q_{i-1}}{}{}{}
            \RCHEVC{5}{q_i}{t_i}{}{}
            \RCHEVC{6}{}{}{}{}
            \RDCHEVC{7}
            \BLOCKC{8}{}{}{v}{u}
            \LWTE{5}
            \RWTE{2}
            \LWHT{0}
            \RWHT{8}
            \draw(2,-\v)node[below]{\smfont$s$};
            \draw(1/2,-\v/2)node[left]{\smfont$r$};
        \end{scope}
        
        \draw(13,+\v*7/2)node{$=$};
    \end{tikzpicture}
}

\newcommand{\numberofblues}[1]{
    \begin{tikzpicture}[x=#1,y=#1]
        \BLOCK{0}{r}{s}{}{}
        \LCHEV{0}{q_1}{t_1}{}{}
        \LDCHEV{1}
        \LCHEV{2}{q_{i-1}}{}{}{}
        \LCHEV{3}{q_i}{t_i}{}{}
        \LDCHEV{4}
        \LCHEV{5}{q_m}{t_m}{v}{u}
        \LWTE{3}
        \LBRD{0}
        \LBRD{1}
        \LBRD{2}
        \RBRD{3}
        \RBRD{4}
        \RBRD{5}
        \LWRD{0}
        \RWRD{6}
        \draw[decorate,decoration={brace,amplitude=10pt}] (-2.5,+\v/2) --++ (3,+\v*3) node[rotate=60,above=10pt,midway]{\smfont$i-1$};
        \draw[decorate,decoration={brace,amplitude=10pt,mirror}] (8,+\v*3/2) --++ (3,+\v*3) node[rotate=60,below=10pt,midway]{\smfont$m-i+1$};
    \end{tikzpicture}
}

\def\fdspace{3}
\newcommand{\forceddiamonds}[1]{
    \begin{tikzpicture}[x=#1,y=#1]
        \draw[draw=gray](0,0)--++(1,-\v)--++(1,\v)--++(-1,\v)--cycle;
        
        \begin{scope}[shift={(\fdspace,0)}]
            \begin{scope}
                \clip(0,0)--++(1,-\v)--++(1,\v)--++(-1,\v)--cycle;
                \draw[red line](0,-\v)\vver\vver;
            \end{scope}
            \draw[draw=gray](0,0)--++(1,-\v)--++(1,\v)--++(-1,\v)--cycle;
        \end{scope}
        
        \begin{scope}[shift={(2*\fdspace,0)}]
            \begin{scope}
                \clip(0,0)--++(1,-\v)--++(1,\v)--++(-1,\v)--cycle;
                \draw[blue1 line](2,-\v/2)\hhor;
                \draw[red line](0,-\v)\vver\vver;
            \end{scope}
            \draw[draw=gray](0,0)--++(1,-\v)--++(1,\v)--++(-1,\v)--cycle;
        \end{scope}
        
        \begin{scope}[shift={(3*\fdspace,0)}]
            \begin{scope}
                \clip(0,0)--++(1,-\v)--++(1,\v)--++(-1,\v)--cycle;
                \draw[blue1 line](2,+\v/2)\hhor;
                \draw[red line](0,-\v)\vver\vver;
            \end{scope}
            \draw[draw=gray](0,0)--++(1,-\v)--++(1,\v)--++(-1,\v)--cycle;
        \end{scope}
    \end{tikzpicture}
}

\newcommand{\fliplemmaa}[1]{
    \begin{tikzpicture}[x=#1,y=#1]
        \BLOCK{0}{r}{s}{}{}
        \LCHEV{0}{q_1}{t_1}{}{}
        \LDCHEV{1}
        \LCHEV{2}{q_i}{t_i}{}{}
        \LDCHEV{3}
        \LCHEV{4}{q_j}{t_j}{}{}
        \LDCHEV{5}
        \LCHEV{6}{q_m}{t_m}{v}{u}
        \LWTE{2}
        \RWTE{4}
        \LWRD{0}
        \RWRD{7}
        \begin{scope}[shift={(14,0)}]
            \RCHEV{0}{q_1}{t_1}{r}{s}
            \RDCHEV{1}
            \RCHEV{2}{q_i}{t_i}{}{}
            \RDCHEV{3}
            \RCHEV{4}{q_j}{t_j}{}{}
            \RDCHEV{5}
            \RCHEV{6}{q_m}{t_m}{}{}
            \BLOCK{7}{}{}{v}{u}
            \LWTE{2}
            \RWTE{4}
            \LWRD{0}
            \RWRD{7}
        \end{scope}
        
        \draw(12,+\v*3)node{$=$};
    \end{tikzpicture}
}

\newcommand{\fliplemmab}[1]{
    \begin{tikzpicture}[x=#1,y=#1]
        \BLOCK{0}{r}{s}{}{}
        \LCHEV{0}{q_1}{t_1}{}{}
        \LDCHEV{1}
        \LCHEV{2}{q_i}{t_i}{}{}
        \LDCHEV{3}
        \LCHEV{4}{q_j}{t_j}{}{}
        \LDCHEV{5}
        \LCHEV{6}{q_m}{t_m}{v}{u}
        \LWTE{2}
        \RWTE{4}
        \LWRD{0}
        \RWRD{4}
        \RWRD{7}
        \begin{scope}[shift={(12,0)}]
            \RCHEV{0}{q_1}{t_1}{r}{s}
            \RDCHEV{1}
            \RCHEV{2}{q_i}{t_i}{}{}
            \RDCHEV{3}
            \BLOCK{4}{}{}{}{}
            \LCHEV{4}{q_j}{t_j}{}{}
            \LDCHEV{5}
            \LCHEV{6}{q_m}{t_m}{v}{u}
            \LWTE{2}
            \RWTE{4}
            \LWRD{0}
            \LWRD{3}
            \RWRD{4}
            \RWRD{7}
        \end{scope}
        \begin{scope}[shift={(24,0)}]
            \RCHEV{0}{q_1}{t_1}{r}{s}
            \RDCHEV{1}
            \RCHEV{2}{q_i}{t_i}{}{}
            \BLOCK{3}{}{}{}{}
            \LDCHEV{3}
            \LCHEV{4}{q_j}{t_j}{}{}
            \LDCHEV{5}
            \LCHEV{6}{q_m}{t_m}{v}{u}
            \LWTE{2}
            \RWTE{4}
            \LWRD{0}
            \LWRD{3}
            \RWRD{4}
            \RWRD{7}
        \end{scope}
        \begin{scope}[shift={(36,0)}]
            \RCHEV{0}{q_1}{t_1}{r}{s}
            \RDCHEV{1}
            \RCHEV{2}{q_i}{t_i}{}{}
            \RDCHEV{3}
            \RCHEV{4}{q_j}{t_j}{}{}
            \RDCHEV{5}
            \RCHEV{6}{q_m}{t_m}{}{}
            \BLOCK{7}{}{}{v}{u}
            \LWTE{2}
            \RWTE{4}
            \LWRD{0}
            \LWRD{3}
            \RWRD{7}
        \end{scope}
        
        \draw(11,+\v*3)node{$=$};
        \draw(23,+\v*3)node{$=$};
        \draw(35,+\v*3)node{$=$};
    \end{tikzpicture}
}

\newcommand{\fliplemmac}[1]{
    \begin{tikzpicture}[x=#1,y=#1]
        \BLOCK{0}{r}{s}{}{}
        \LCHEV{0}{q_1}{t_1}{}{}
        \LDCHEV{1}
        \LCHEV{2}{q_i}{t_i}{}{}
        \LDCHEV{3}
        \LCHEV{4}{q_m}{t_m}{v}{u}
        \LWTE{2}
        \RWTE{2}
        \LWRD{0}
        \RWRD{2}
        \RWRD{5}
        \begin{scope}[shift={(12,0)}]
            \RCHEV{0}{q_1}{t_1}{r}{s}
            \RDCHEV{1}
            \BLOCK{2}{}{}{}{}
            \LCHEV{2}{q_i}{t_i}{}{}
            \LDCHEV{3}
            \LCHEV{4}{q_m}{t_m}{v}{u}
            \LWTE{2}
            \RWTE{2}
            \LWRD{0}
            \RWRD{2}
            \RWRD{5}
        \end{scope}
        \begin{scope}[shift={(24,0)}]
            \RCHEV{0}{q_1}{t_1}{r}{s}
            \RDCHEV{1}
            \RCHEV{2}{q_i}{t_i}{}{}
            \BLOCK{3}{}{}{}{}
            \LDCHEV{3}
            \LCHEV{4}{q_m}{t_m}{v}{u}
            \LWTE{2}
            \RWTE{2}
            \LWRD{0}
            \LWRD{3}
            \RWRD{5}
        \end{scope}
        \begin{scope}[shift={(36,0)}]
            \RCHEV{0}{q_1}{t_1}{r}{s}
            \RDCHEV{1}
            \RCHEV{2}{q_i}{t_i}{}{}
            \RDCHEV{3}
            \RCHEV{4}{q_m}{t_m}{}{}
            \BLOCK{5}{}{}{v}{u}
            \LWTE{2}
            \RWTE{2}
            \LWRD{0}
            \LWRD{3}
            \RWRD{5}
        \end{scope}
        
        \draw(10,+\v*2)node{$=$};
        \draw(22,+\v*2)node{$=$};
        \draw(34,+\v*2)node{$=$};
    \end{tikzpicture}
}

\newcommand{\intdiagram}[1]{
    \foreach [count=\i] \ali in #1 {
        \foreach [count=\j] \entry in \ali {
            \ifthenelse{\j = 1}{\draw (\i*2-1,1) node{$\mathit{\entry}$};}{\draw (\i*2-1,\j*2-1) node{$\entry$};}
            \draw (\i*2-2,\j*2-2)--++(2,0)--++(0,2)--++(-2,0)--cycle;
        }
    }
}

\newcommand{\bijectionex}[1]{
    \begin{tikzpicture}[x=#1,y=#1]
        \blueweight{1}{2}
        \blueweight{2}{2}
        \blueweight{2}{3}
        \blueweight{3}{2}
        \blueweight{3}{3}
        \draw[blue2 line](1,0)\dia;
        \draw[blue3 line](3,0)\vver\vver\vver\vver\ddia\dia;
        \draw[blue4 line](5,0)\vver\vver\vver\dia\hhor\ddia;
        \draw[blue1 line](7,0)\dia\hhor\hhor\dia\vver\dia;
        
        \draw[gray](0,0)--++(8,0)--++(6,+\v*6)--++(-8,0)--cycle;
        \foreach \i in {1,...,3} \draw[gray](\i*2,0)--++(6,+\v*6);
        \foreach \i in {1,...,5} \draw[gray](\i,+\v*\i)--++(8,0);
        
        \colourcirc{1/2}{+\v*1/2}{blue2}
        \colourcirc{5/2}{+\v*5/2}{blue1}
        \colourcirc{9/2}{+\v*9/2}{blue4}
        \colourcirc{11/2}{+\v*11/2}{blue3}
        \colourcirc{1}{0}{blue2}
        \colourcirc{3}{0}{blue3}
        \colourcirc{5}{0}{blue4}
        \colourcirc{7}{0}{blue1}
        
        \draw(1/2,+\v*1/2)node[left]{\smfont$2$};
        \draw(5/2,+\v*5/2)node[left]{\smfont$1$};
        \draw(9/2,+\v*9/2)node[left]{\smfont$4$};
        \draw(11/2,+\v*11/2)node[left]{\smfont$3$};
        \draw(1,0)node[below]{\smfont$2$};
        \draw(3,0)node[below]{\smfont$3$};
        \draw(5,0)node[below]{\smfont$4$};
        \draw(7,0)node[below]{\smfont$1$};
        
        \begin{scope}[shift={(12,0)}]
            \blueweight{1}{4}
            \blueweight{2}{1}
            \blueweight{2}{4}
            \blueweight{3}{1}
            \blueweight{4}{1}
            \draw[blue2 line](1,0)\dia;
            \draw[blue3 line](3,0)\vver\vver\vver\vver\ddia\dia;
            \draw[blue4 line](5,0)\vver\dia\hhor\dia\vver\vver\dia;
            \draw[blue1 line](7,0)\vver\ddia\dia\hhor\hhor;
            
            \draw[gray](0,0)--++(8,0)--++(6,+\v*6)--++(-8,0)--cycle;
            \foreach \i in {1,...,3} \draw[gray](\i*2,0)--++(6,+\v*6);
            \foreach \i in {1,...,5} \draw[gray](\i,+\v*\i)--++(8,0);
            
            \colourcirc{1/2}{+\v*1/2}{blue2}
            \colourcirc{5/2}{+\v*5/2}{blue1}
            \colourcirc{9/2}{+\v*9/2}{blue4}
            \colourcirc{11/2}{+\v*11/2}{blue3}
            \colourcirc{1}{0}{blue2}
            \colourcirc{3}{0}{blue3}
            \colourcirc{5}{0}{blue4}
            \colourcirc{7}{0}{blue1}
            
            \draw(1/2,+\v*1/2)node[left]{\smfont$2$};
            \draw(5/2,+\v*5/2)node[left]{\smfont$1$};
            \draw(9/2,+\v*9/2)node[left]{\smfont$4$};
            \draw(11/2,+\v*11/2)node[left]{\smfont$3$};
            \draw(1,0)node[below]{\smfont$2$};
            \draw(3,0)node[below]{\smfont$3$};
            \draw(5,0)node[below]{\smfont$4$};
            \draw(7,0)node[below]{\smfont$1$};
        \end{scope}
        
        \begin{scope}[shift={(24,0)}]
            \blueweight{1}{4}
            \blueweight{2}{1}
            \blueweight{3}{1}
            \blueweight{3}{3}
            \blueweight{4}{1}
            \draw[blue2 line](1,0)\dia;
            \draw[blue3 line](3,0)\vver\ddia\vver\vver\vver\dia;
            \draw[blue4 line](5,0)\vver\vver\vver\ddia\dia\hhor;
            \draw[blue1 line](7,0)\vver\dia\hhor\ddia\hhor;
            
            \draw[gray](0,0)--++(8,0)--++(6,+\v*6)--++(-8,0)--cycle;
            \foreach \i in {1,...,3} \draw[gray](\i*2,0)--++(6,+\v*6);
            \foreach \i in {1,...,5} \draw[gray](\i,+\v*\i)--++(8,0);
            
            \colourcirc{1/2}{+\v*1/2}{blue2}
            \colourcirc{5/2}{+\v*5/2}{blue1}
            \colourcirc{9/2}{+\v*9/2}{blue4}
            \colourcirc{11/2}{+\v*11/2}{blue3}
            \colourcirc{1}{0}{blue2}
            \colourcirc{3}{0}{blue3}
            \colourcirc{5}{0}{blue4}
            \colourcirc{7}{0}{blue1}
            
            \draw(1/2,+\v*1/2)node[left]{\smfont$2$};
            \draw(5/2,+\v*5/2)node[left]{\smfont$1$};
            \draw(9/2,+\v*9/2)node[left]{\smfont$4$};
            \draw(11/2,+\v*11/2)node[left]{\smfont$3$};
            \draw(1,0)node[below]{\smfont$2$};
            \draw(3,0)node[below]{\smfont$3$};
            \draw(5,0)node[below]{\smfont$4$};
            \draw(7,0)node[below]{\smfont$1$};
        \end{scope}
        
        \begin{scope}[shift={(36,0)}]
            \blueweight{2}{1}
            \blueweight{2}{3}
            \blueweight{3}{1}
            \blueweight{3}{3}
            \blueweight{4}{1}
            \draw[blue2 line](1,0)\dia;
            \draw[blue3 line](3,0)\vver\vver\vver\vver\ddia\dia;
            \draw[blue4 line](5,0)\vver\vver\vver\dia\hhor\ddia;
            \draw[blue1 line](7,0)\vver\dia\hhor\hhor\ddia;
            
            \draw[gray](0,0)--++(8,0)--++(6,+\v*6)--++(-8,0)--cycle;
            \foreach \i in {1,...,3} \draw[gray](\i*2,0)--++(6,+\v*6);
            \foreach \i in {1,...,5} \draw[gray](\i,+\v*\i)--++(8,0);
            
            \colourcirc{1/2}{+\v*1/2}{blue2}
            \colourcirc{5/2}{+\v*5/2}{blue1}
            \colourcirc{9/2}{+\v*9/2}{blue4}
            \colourcirc{11/2}{+\v*11/2}{blue3}
            \colourcirc{1}{0}{blue2}
            \colourcirc{3}{0}{blue3}
            \colourcirc{5}{0}{blue4}
            \colourcirc{7}{0}{blue1}
            
            \draw(1/2,+\v*1/2)node[left]{\smfont$2$};
            \draw(5/2,+\v*5/2)node[left]{\smfont$1$};
            \draw(9/2,+\v*9/2)node[left]{\smfont$4$};
            \draw(11/2,+\v*11/2)node[left]{\smfont$3$};
            \draw(1,0)node[below]{\smfont$2$};
            \draw(3,0)node[below]{\smfont$3$};
            \draw(5,0)node[below]{\smfont$4$};
            \draw(7,0)node[below]{\smfont$1$};
        \end{scope}
        
        \begin{scope}[shift={(0,-15)}]
            \draw[stealth-stealth](4,8)--++(0,4);
            \intdiagram{{{4,1},{1},{2,2,2},{3,3,3}}}
            \begin{scope}[shift={(12,0)}]
            \draw[stealth-stealth](4,8)--++(0,4);
            \intdiagram{{{4,4},{1},{2,2,2},{3,3,1}}}
            \end{scope}
            \begin{scope}[shift={(24,0)}]
            \draw[stealth-stealth](4,8)--++(0,4);
            \intdiagram{{{4,4},{1},{2,2,1},{3,3,3}}}
            \end{scope}
            \begin{scope}[shift={(36,0)}]
                \draw[stealth-stealth](4,8)--++(0,4);
                \intdiagram{{{4,4},{1},{2,2,2},{3,3,3}}}
            \end{scope}
        \end{scope}
    \end{tikzpicture}
}

\newcommand{\firstcolumnex}[1]{
    \begin{tikzpicture}[x=#1,y=#1]
        \blueweight{1}{2}
        \blueweight{1}{3}
        \blueweight{1}{6}
        
        \draw[blue2 line](1,0)\vver\vver\vver\vver\dia;
        \draw[blue3 line](5/2,+\v*1/2)\hhor;
        \draw[blue4 line](11/2,+\v*7/2)\hhor;
        \draw[blue1 line](13/2,+\v*9/2)\dia\vver\dia;
        
        \draw[gray](0,0)--++(2,0)--++(7,+\v*7)--++(-2,0)--cycle;
        \foreach \i in {1,...,6} \draw[gray](\i,+\v*\i)--++(2,0);
        
        \colourcirc{1/2}{+\v*1/2}{blue3}
        \colourcirc{7/2}{+\v*7/2}{blue4}
        \colourcirc{9/2}{+\v*9/2}{blue2}
        \colourcirc{13/2}{+\v*13/2}{blue1}
        \draw(1/2,+\v*1/2)node[left]{\smfont$3$};
        \draw(7/2,+\v*7/2)node[left]{\smfont$4$};
        \draw(9/2,+\v*9/2)node[left]{\smfont$2$};
        \draw(13/2,+\v*13/2)node[left]{\smfont$1$};
        
        \colourcirc{5/2}{+\v*1/2}{blue3}
        \colourcirc{11/2}{+\v*7/2}{blue4}
        \colourcirc{13/2}{+\v*9/2}{blue1}
        \draw(5/2,+\v*1/2)node[right]{\smfont$3$};
        \draw(11/2,+\v*7/2)node[right]{\smfont$4$};
        \draw(13/2,+\v*9/2)node[right]{\smfont$1$};
        
        \colourcirc{1}{0}{blue2}
        \draw(1,0)node[below]{\smfont$2$};
        
        \perparrow{0}{0}{7}{\vx*7}{-2}{left}{$\al^*$}
        \perparrow{2}{0}{9}{\vx*7}{2}{right}{$\be^{(2)}$}
        
        \begin{scope}[shift={(15,0)}]
            \blueweight{1}{2}
            \blueweight{1}{3}
            
            \draw[blue2 line](1,0)\vver\vver\vver\vver\dia;
            \draw[blue3 line](5/2,+\v*1/2)\hhor;
            \draw[blue4 line](11/2,+\v*7/2)\hhor;
            \draw[blue1 line](15/2,+\v*11/2)\ddia;
            
            \draw[gray](0,0)--++(2,0)--++(7,+\v*7)--++(-2,0)--cycle;
            \foreach \i in {1,...,6} \draw[gray](\i,+\v*\i)--++(2,0);
            
            \colourcirc{1/2}{+\v*1/2}{blue3}
            \colourcirc{7/2}{+\v*7/2}{blue4}
            \colourcirc{9/2}{+\v*9/2}{blue2}
            \colourcirc{13/2}{+\v*13/2}{blue1}
            \draw(1/2,+\v*1/2)node[left]{\smfont$3$};
            \draw(7/2,+\v*7/2)node[left]{\smfont$4$};
            \draw(9/2,+\v*9/2)node[left]{\smfont$2$};
            \draw(13/2,+\v*13/2)node[left]{\smfont$1$};
            
            \colourcirc{5/2}{+\v*1/2}{blue3}
            \colourcirc{11/2}{+\v*7/2}{blue4}
            \colourcirc{15/2}{+\v*11/2}{blue1}
            \draw(5/2,+\v*1/2)node[right]{\smfont$3$};
            \draw(11/2,+\v*7/2)node[right]{\smfont$4$};
            \draw(15/2,+\v*11/2)node[right]{\smfont$1$};
            
            \colourcirc{1}{0}{blue2}
            \draw(1,0)node[below]{\smfont$2$};
            
            \perparrow{0}{0}{7}{\vx*7}{-2}{left}{$\al^*$}
            \perparrow{2}{0}{9}{\vx*7}{2}{right}{$\de^{(2)}$}
        \end{scope}
    \end{tikzpicture}
}
\newcommand{\robust}[1]{
    \begin{tikzpicture}[x=#1,y=#1]
        \begin{scope}
            \clip(0,0)--++(2,0)--++(2,+\v*2)--++(-2,0)--cycle;
            \draw[blue2 line](1/2,-\v/2)\vver\vver\vver;
            \draw[blue1 line](0,+\v/2)--++(3,0);
            \draw[blue3 line](1,+\v*3/2)--++(3,0);
        \end{scope}
        \filldraw[light gray](2,0)--++(2,0)--++(2,+\v*2)--++(-2,0)--cycle;
        \draw[gray](0,0)--++(2,0)--++(2,+\v*2)--++(-2,0)--cycle;
        \draw[gray](1,\v)--++(2,0);
        \draw[gray,dashed](2,0)--++(2,0);
        \draw[gray,dashed](3,\v)--++(2,0);
        \draw[gray,dashed](4,+\v*2)--++(2,0);
        \draw(1,0)node[below]{\smfont$h$};
        \draw(1/2,+\v/2)node[left]{\smfont$i$};
        \draw(3/2,+\v*3/2)node[left]{\smfont$j_k$};
        \draw(3,+\v*2)node[above]{\smfont$h$};
        \draw(7/2,+\v*3/2)node[right]{\smfont$j_k$};
        \draw(5/2,+\v/2)node[right]{\smfont$i$};
        \draw(-5/2,+\v*2)node[right]{1. (a)};
        \begin{scope}[shift={(10,0)}]
            \begin{scope}
                \clip(0,0)--++(2,0)--++(2,+\v*2)--++(-2,0)--cycle;
                \draw[blue2 line](1/2,-\v/2)\vver\ver\ddia\dia;
                \draw[blue1 line](0,+\v/2)--++(3,0);
                \draw[blue3 line](4,\v)\ddia\ddia;
            \end{scope}
            \filldraw[light gray](2,0)--++(2,0)--++(2,+\v*2)--++(-2,0)--cycle;
            \draw[gray](0,0)--++(2,0)--++(2,+\v*2)--++(-2,0)--cycle;
            \draw[gray](1,\v)--++(2,0);
            \draw[gray,dashed](2,0)--++(2,0);
            \draw[gray,dashed](3,\v)--++(2,0);
            \draw[gray,dashed](4,+\v*2)--++(2,0);
            \draw(1,0)node[below]{\smfont$j_{k+1}$};
            \draw(1/2,+\v/2)node[left]{\smfont$i$};
            \draw(3/2,+\v*3/2)node[left]{\smfont$j_{k+1}$};
            \draw(3,+\v*2)node[above]{\smfont$j_k$};
            \draw(7/2,+\v*3/2)node[right]{\smfont$j_k$};
            \draw(5/2,+\v/2)node[right]{\smfont$i$};
            \draw(-5/2,+\v*2)node[right]{(b)};
        \end{scope}
        \begin{scope}[shift={(20,0)}]
            \begin{scope}
                \clip(0,0)--++(2,0)--++(2,+\v*2)--++(-2,0)--cycle;
                \filldraw[light gray](0,0)--++(2,0)--++(-1,\v)--cycle;
                \draw[gray, dotted](2,0)--++(-1,\v);
                \draw[blue1 line](3,0)\ddia\ddia;
                \draw[blue3 line](4,\v)\ddia\ddia;
            \end{scope}
            \filldraw[light gray](2,0)--++(2,0)--++(2,+\v*2)--++(-2,0)--cycle;
            \draw[gray](0,0)--++(2,0)--++(2,+\v*2)--++(-2,0)--cycle;
            \draw[gray](1,\v)--++(2,0);
            \draw[gray,dashed](2,0)--++(2,0);
            \draw[gray,dashed](3,\v)--++(2,0);
            \draw[gray,dashed](4,+\v*2)--++(2,0);
            \draw(3/2,+\v*3/2)node[left]{\smfont$i$};
            \draw(3,+\v*2)node[above]{\smfont$j_k$};
            \draw(7/2,+\v*3/2)node[right]{\smfont$j_k$};
            \draw(5/2,+\v/2)node[right]{\smfont$i$};
            \draw(-5/2,+\v*2)node[right]{(c)};
            \filldraw[red block](1,\v)--++(2,0)--++(1,\v)--++(-2,0)--cycle;
        \end{scope}
        \begin{scope}[shift={(30,0)}]
            \begin{scope}
                \clip(0,0)--++(2,0)--++(2,+\v*2)--++(-2,0)--cycle;
                \filldraw[light gray](0,0)--++(2,0)--++(-1,\v)--cycle;
                \draw[gray, dotted](2,0)--++(-1,\v);
                \draw[blue1 line](3,0)\ddia\vver;
                \draw[blue3 line](1,+\v*3/2)--++(3,0);
            \end{scope}
            \filldraw[light gray](2,0)--++(2,0)--++(2,+\v*2)--++(-2,0)--cycle;
            \draw[gray](0,0)--++(2,0)--++(2,+\v*2)--++(-2,0)--cycle;
            \draw[gray](1,\v)--++(2,0);
            \draw[gray,dashed](2,0)--++(2,0);
            \draw[gray,dashed](3,\v)--++(2,0);
            \draw[gray,dashed](4,+\v*2)--++(2,0);
            \draw(3/2,+\v*3/2)node[left]{\smfont$j_k$};
            \draw(3,+\v*2)node[above]{\smfont$i$};
            \draw(7/2,+\v*3/2)node[right]{\smfont$j_k$};
            \draw(5/2,+\v/2)node[right]{\smfont$i$};
            \draw(-5/2,+\v*2)node[right]{(d)};
        \end{scope}
        \begin{scope}[shift={(0,-7)}]
            \begin{scope}
                \clip(0,0)--++(4,0)--++(2,+\v*2)--++(-4,0)--cycle;
                \draw[blue2 line](5/2,-\v/2)\vver\vver\vver;
                \draw[blue1 line](11/2,+\v/2)\hhor\hor\dia\vver\ver;
                \draw[blue3 line](13/2,+\v*3/2)\hhor\hhor\hhor;
            \end{scope}
            \filldraw[light gray](0,0)--++(2,0)--++(2,+\v*2)--++(-2,0)--cycle;
            \draw[gray](2,0)--++(2,0)--++(2,+\v*2)--++(-2,0)--cycle;
            \draw[gray](3,\v)--++(2,0);
            \draw[gray,dashed](0,0)--++(2,0);
            \draw[gray,dashed](1,\v)--++(2,0);
            \draw[gray,dashed](2,+\v*2)--++(2,0);
            \draw(3,0)node[below]{\smfont$h$};
            \draw(5,+\v*2)node[above]{\smfont$h$};
            \draw(7/2,+\v*3/2)node[left]{\smfont$j$};
            \draw(5/2,+\v*1/2)node[left]{\smfont$i_k$};
            \draw(9/2,+\v/2)node[right]{\smfont$i_k$};
            \draw(11/2,+\v*3/2)node[right]{\smfont$j$};
            \draw(-5/2,+\v*2)node[right]{2. (a)};
            \begin{scope}[shift={(10,0)}]
                \begin{scope}
                    \clip(0,0)--++(4,0)--++(2,+\v*2)--++(-4,0)--cycle;
                    \draw[blue2 line](5,0)\ddia\vver\ver;
                    \draw[blue1 line](7/2,-\v/2)\ddia\dia\vver\ver;
                    \draw[blue3 line](13/2,+\v*3/2)\hhor\hhor\hhor;
                \end{scope}
                \filldraw[light gray](0,0)--++(2,0)--++(2,+\v*2)--++(-2,0)--cycle;
                \draw[gray](2,0)--++(2,0)--++(2,+\v*2)--++(-2,0)--cycle;
                \draw[gray](3,\v)--++(2,0);
                \draw[gray,dashed](0,0)--++(2,0);
                \draw[gray,dashed](1,\v)--++(2,0);
                \draw[gray,dashed](2,+\v*2)--++(2,0);
                \draw(3,0)node[below]{\smfont$i_k$};
                \draw(5,+\v*2)node[above]{\smfont$i_{k+1}$};
                \draw(7/2,+\v*3/2)node[left]{\smfont$j$};
                \draw(5/2,+\v*1/2)node[left]{\smfont$i_k$};
                \draw(9/2,+\v/2)node[right]{\smfont$i_{k+1}$};
                \draw(11/2,+\v*3/2)node[right]{\smfont$j$};
                \draw(-5/2,+\v*2)node[right]{(b)};
            \end{scope}
            \begin{scope}[shift={(20,0)}]
                \begin{scope}
                    \clip(0,0)--++(4,0)--++(2,+\v*2)--++(-4,0)--cycle;
                    \filldraw[light gray](4,+\v*2)--++(2,0)--++(-1,-\v)--cycle;
                    \draw[gray, dotted](4,+\v*2)--++(1,-\v);
                    \draw[blue1 line](7/2,-\v/2)\ddia\dia\vver\ver;
                    \draw[blue3 line](5,0)\ddia\dia\hhor\hor;
                \end{scope}
                \filldraw[light gray](0,0)--++(2,0)--++(2,+\v*2)--++(-2,0)--cycle;
                \draw[gray](2,0)--++(2,0)--++(2,+\v*2)--++(-2,0)--cycle;
                \draw[gray](3,\v)--++(2,0);
                \draw[gray,dashed](0,0)--++(2,0);
                \draw[gray,dashed](1,\v)--++(2,0);
                \draw[gray,dashed](2,+\v*2)--++(2,0);
                \draw(3,0)node[below]{\smfont$i_k$};
                \draw(7/2,+\v*3/2)node[left]{\smfont$j$};
                \draw(5/2,+\v*1/2)node[left]{\smfont$i_k$};
                \draw(9/2,+\v/2)node[right]{\smfont$j$};
                \filldraw[red block](2,0)--++(2,0)--++(1,\v)--++(-2,0)--cycle;
                \draw(-5/2,+\v*2)node[right]{(c)};
            \end{scope}
            \begin{scope}[shift={(30,0)}]
                \begin{scope}
                    \clip(0,0)--++(4,0)--++(2,+\v*2)--++(-4,0)--cycle;
                    \filldraw[light gray](4,+\v*2)--++(2,0)--++(-1,-\v)--cycle;
                    \draw[gray, dotted](4,+\v*2)--++(1,-\v);
                    \draw[blue1 line](11/2,+\v/2)\hhor\hor\dia\vver\ver;
                    \draw[blue3 line](5/2,-\v/2)\vver\ver\dia\hhor\hor;
                    \draw[blue1 line](11/2,+\v/2)\hhor\hor;
                \end{scope}
                \filldraw[light gray](0,0)--++(2,0)--++(2,+\v*2)--++(-2,0)--cycle;
                \draw[gray](2,0)--++(2,0)--++(2,+\v*2)--++(-2,0)--cycle;
                \draw[gray](3,\v)--++(2,0);
                \draw[gray,dashed](0,0)--++(2,0);
                \draw[gray,dashed](1,\v)--++(2,0);
                \draw[gray,dashed](2,+\v*2)--++(2,0);
                \draw(3,0)node[below]{\smfont$j$};
                \draw(7/2,+\v*3/2)node[left]{\smfont$j$};
                \draw(5/2,+\v*1/2)node[left]{\smfont$i_k$};
                \draw(9/2,+\v/2)node[right]{\smfont$i_k$};
                \draw(-5/2,+\v*2)node[right]{(d)};
            \end{scope}
        \end{scope}
    \end{tikzpicture}
}

\newcommand{\diagram}[2]{
    \begin{tikzpicture}[x=#1,y=#1,baseline=1cm]
        \foreach [count=\i] \ali in #2 {
            \foreach [count=\j] \entry in \ali {
                \ifthenelse{\j = 1}{\draw (\i+1/2,3/2) node{$\mathit{\entry}$};}{\draw (\i+0.5,\j+0.5) node{$\entry$};}
                \draw (\i,\j)--++(1,0)--++(0,1)--++(-1,0)--cycle;
            }
        }
    \end{tikzpicture}
}

\newcommand{\diagramnotikz}[1]{
    \foreach [count=\i] \ali in #1 {
        \foreach [count=\j] \entry in \ali {
            \ifthenelse{\j = 1}{\draw (\i+1/2,3/2) node{$\mathit{\entry}$};}{\draw (\i+0.5,\j+0.5) node{$\entry$};}
            \draw (\i,\j)--++(1,0)--++(0,1)--++(-1,0)--cycle;
        }
    }
}

\newcommand{\deleteboxes}[1]{
    \begin{tikzpicture}[x=#1,y=#1]
        \diagramnotikz{{{6,6,6,6,1},{2,2,2,1},{4,3},{1,1,1},{3},{5,5,4,3}}}
        \draw(4,0)node{$F$};
        \begin{scope}[shift={(11,0)}]
            \diagramnotikz{{{6,6,6,6},{2,2,2},{4,3},{1},{3},{5,5,4,3}}}
            \draw(4,0)node{$F'$};
        \end{scope}
        
        \draw[-stealth](8,3)-++(3,0);
    \end{tikzpicture}
}

\newcommand{\atomexample}[1]{
    \begin{tikzpicture}[x=#1,y=#1]
        \blueweight{1}{2}
        \blueweight{2}{2}
        \blueweight{2}{3}
        \blueweight{3}{2}
        \blueweight{3}{3}
        \draw[blue2 line](1,0)\dia;
        \draw[blue3 line](3,0)\vver\vver\vver\vver\ddia\dia;
        \draw[blue4 line](5,0)\vver\vver\vver\dia\hhor\ddia;
        \draw[blue1 line](7,0)\dia\hhor\hhor\dia\vver\dia;
        
        \draw[gray](0,0)--++(8,0)--++(6,+\v*6)--++(-8,0)--cycle;
        \foreach \i in {1,...,3} \draw[gray](\i*2,0)--++(6,+\v*6);
        \foreach \i in {1,...,5} \draw[gray](\i,+\v*\i)--++(8,0);
        
        \colourcirc{1/2}{+\v*1/2}{blue2}
        \colourcirc{5/2}{+\v*5/2}{blue1}
        \colourcirc{9/2}{+\v*9/2}{blue4}
        \colourcirc{11/2}{+\v*11/2}{blue3}
        \colourcirc{1}{0}{blue2}
        \colourcirc{3}{0}{blue3}
        \colourcirc{5}{0}{blue4}
        \colourcirc{7}{0}{blue1}
        
        \draw(1/2,+\v*1/2)node[left]{\smfont$2$};
        \draw(5/2,+\v*5/2)node[left]{\smfont$1$};
        \draw(9/2,+\v*9/2)node[left]{\smfont$4$};
        \draw(11/2,+\v*11/2)node[left]{\smfont$3$};
        \draw(1,0)node[below]{\smfont$2$};
        \draw(3,0)node[below]{\smfont$3$};
        \draw(5,0)node[below]{\smfont$4$};
        \draw(7,0)node[below]{\smfont$1$};
        
        \perparrow{0}{0}{8}{0}{2}{below}{$\si^{-1}$}
        
        \draw(4,-5)node{$x_1x_2^2x_3^2$};
        
        \perparrow{0}{0}{6}{\vx*6}{-2}{left}{$\al^*$}
        \begin{scope}[shift={(12,0)}]
            \blueweight{1}{4}
            \blueweight{2}{1}
            \blueweight{2}{4}
            \blueweight{3}{1}
            \blueweight{4}{1}
            \draw[blue2 line](1,0)\dia;
            \draw[blue3 line](3,0)\vver\vver\vver\vver\ddia\dia;
            \draw[blue4 line](5,0)\vver\dia\hhor\dia\vver\vver\dia;
            \draw[blue1 line](7,0)\vver\ddia\dia\hhor\hhor;
            
            \draw[gray](0,0)--++(8,0)--++(6,+\v*6)--++(-8,0)--cycle;
            \foreach \i in {1,...,3} \draw[gray](\i*2,0)--++(6,+\v*6);
            \foreach \i in {1,...,5} \draw[gray](\i,+\v*\i)--++(8,0);
            
            \colourcirc{1/2}{+\v*1/2}{blue2}
            \colourcirc{5/2}{+\v*5/2}{blue1}
            \colourcirc{9/2}{+\v*9/2}{blue4}
            \colourcirc{11/2}{+\v*11/2}{blue3}
            \colourcirc{1}{0}{blue2}
            \colourcirc{3}{0}{blue3}
            \colourcirc{5}{0}{blue4}
            \colourcirc{7}{0}{blue1}
            
            \draw(1/2,+\v*1/2)node[left]{\smfont$2$};
            \draw(5/2,+\v*5/2)node[left]{\smfont$1$};
            \draw(9/2,+\v*9/2)node[left]{\smfont$4$};
            \draw(11/2,+\v*11/2)node[left]{\smfont$3$};
            \draw(1,0)node[below]{\smfont$2$};
            \draw(3,0)node[below]{\smfont$3$};
            \draw(5,0)node[below]{\smfont$4$};
            \draw(7,0)node[below]{\smfont$1$};
            
            \draw(4,-5)node{$x_1x_2^2x_3x_4$};
        \end{scope}
        
        \begin{scope}[shift={(24,0)}]
            \blueweight{1}{4}
            \blueweight{2}{1}
            \blueweight{3}{1}
            \blueweight{3}{3}
            \blueweight{4}{1}
            \draw[blue2 line](1,0)\dia;
            \draw[blue3 line](3,0)\vver\ddia\vver\vver\vver\dia;
            \draw[blue4 line](5,0)\vver\vver\vver\ddia\dia\hhor;
            \draw[blue1 line](7,0)\vver\dia\hhor\ddia\hhor;
            
            \draw[gray](0,0)--++(8,0)--++(6,+\v*6)--++(-8,0)--cycle;
            \foreach \i in {1,...,3} \draw[gray](\i*2,0)--++(6,+\v*6);
            \foreach \i in {1,...,5} \draw[gray](\i,+\v*\i)--++(8,0);
            
            \colourcirc{1/2}{+\v*1/2}{blue2}
            \colourcirc{5/2}{+\v*5/2}{blue1}
            \colourcirc{9/2}{+\v*9/2}{blue4}
            \colourcirc{11/2}{+\v*11/2}{blue3}
            \colourcirc{1}{0}{blue2}
            \colourcirc{3}{0}{blue3}
            \colourcirc{5}{0}{blue4}
            \colourcirc{7}{0}{blue1}
            
            \draw(1/2,+\v*1/2)node[left]{\smfont$2$};
            \draw(5/2,+\v*5/2)node[left]{\smfont$1$};
            \draw(9/2,+\v*9/2)node[left]{\smfont$4$};
            \draw(11/2,+\v*11/2)node[left]{\smfont$3$};
            \draw(1,0)node[below]{\smfont$2$};
            \draw(3,0)node[below]{\smfont$3$};
            \draw(5,0)node[below]{\smfont$4$};
            \draw(7,0)node[below]{\smfont$1$};
            
            \draw(4,-5)node{$x_1x_2x_3^2x_4$};
        \end{scope}
        
        \begin{scope}[shift={(36,0)}]
            \blueweight{2}{1}
            \blueweight{2}{3}
            \blueweight{3}{1}
            \blueweight{3}{3}
            \blueweight{4}{1}
            \draw[blue2 line](1,0)\dia;
            \draw[blue3 line](3,0)\vver\vver\vver\vver\ddia\dia;
            \draw[blue4 line](5,0)\vver\vver\vver\dia\hhor\ddia;
            \draw[blue1 line](7,0)\vver\dia\hhor\hhor\ddia;
            
            \draw[gray](0,0)--++(8,0)--++(6,+\v*6)--++(-8,0)--cycle;
            \foreach \i in {1,...,3} \draw[gray](\i*2,0)--++(6,+\v*6);
            \foreach \i in {1,...,5} \draw[gray](\i,+\v*\i)--++(8,0);
            
            \colourcirc{1/2}{+\v*1/2}{blue2}
            \colourcirc{5/2}{+\v*5/2}{blue1}
            \colourcirc{9/2}{+\v*9/2}{blue4}
            \colourcirc{11/2}{+\v*11/2}{blue3}
            \colourcirc{1}{0}{blue2}
            \colourcirc{3}{0}{blue3}
            \colourcirc{5}{0}{blue4}
            \colourcirc{7}{0}{blue1}
            
            \draw(1/2,+\v*1/2)node[left]{\smfont$2$};
            \draw(5/2,+\v*5/2)node[left]{\smfont$1$};
            \draw(9/2,+\v*9/2)node[left]{\smfont$4$};
            \draw(11/2,+\v*11/2)node[left]{\smfont$3$};
            \draw(1,0)node[below]{\smfont$2$};
            \draw(3,0)node[below]{\smfont$3$};
            \draw(5,0)node[below]{\smfont$4$};
            \draw(7,0)node[below]{\smfont$1$};
            
            \draw(4,-5)node{$x_2^2x_3^2x_4$};
        \end{scope}
    \end{tikzpicture}
}

\newcommand{\schurexample}[1]{
    \begin{tikzpicture}[x=#1,y=#1]
        \blueweight{1}{1}
        \blueweight{2}{1}
        \blueweight{3}{1}
        \blueweight{1}{3}
        \blueweight{2}{3}
        \draw[blue1 line](1,0)\vver\dia;
        \draw[blue1 line](3,0)\vver\ddia\vver\dia;
        \draw[blue1 line](5,0)\vver\ddia\vver\ddia\dia;
        
        \draw[gray](0,0)--++(6,0)--++(5,+\v*5)--++(-6,0)--cycle;
        \foreach \i in {1,...,2} \draw[gray](\i*2,0)--++(5,+\v*5);
        \foreach \i in {1,...,4} \draw[gray](\i,+\v*\i)--++(6,0);
        \draw(3/2,+\v*3/2)node[left]{\smfont$1$};
        \draw(7/2,+\v*7/2)node[left]{\smfont$1$};
        \draw(9/2,+\v*9/2)node[left]{\smfont$1$};
        \draw(1,0)node[below]{\smfont$1$};
        \draw(3,0)node[below]{\smfont$1$};
        \draw(5,0)node[below]{\smfont$1$};
        \colourcirc{3/2}{+\v*3/2}{blue1}
        \colourcirc{7/2}{+\v*7/2}{blue1}
        \colourcirc{9/2}{+\v*9/2}{blue1}
        \colourcirc{1}{0}{blue1}
        \colourcirc{3}{0}{blue1}
        \colourcirc{5}{0}{blue1}

        \perparrow{0}{0}{6}{0}{2}{below}{$\bacirc^3$}
        \draw(3,-5)node{$x_1^2x_2^2x_3$};
        
        \perparrow{0}{0}{5}{\vx*5}{-2}{left}{$\la^-$}
        \begin{scope}[shift={(12,0)}]
            \blueweight{1}{1}
            \blueweight{2}{1}
            \blueweight{3}{1}
            \blueweight{1}{3}
            \blueweight{3}{2}
            \draw[blue1 line](1,0)\vver\dia;
            \draw[blue1 line](3,0)\vver\ddia\vver\dia;
            \draw[blue1 line](5,0)\vver\vver\ddia\ddia\dia;
            
            \draw[gray](0,0)--++(6,0)--++(5,+\v*5)--++(-6,0)--cycle;
            \foreach \i in {1,...,2} \draw[gray](\i*2,0)--++(5,+\v*5);
            \foreach \i in {1,...,4} \draw[gray](\i,+\v*\i)--++(6,0);
            \draw(3/2,+\v*3/2)node[left]{\smfont$1$};
            \draw(7/2,+\v*7/2)node[left]{\smfont$1$};
            \draw(9/2,+\v*9/2)node[left]{\smfont$1$};
            \draw(1,0)node[below]{\smfont$1$};
            \draw(3,0)node[below]{\smfont$1$};
            \draw(5,0)node[below]{\smfont$1$};
            \colourcirc{3/2}{+\v*3/2}{blue1}
            \colourcirc{7/2}{+\v*7/2}{blue1}
            \colourcirc{9/2}{+\v*9/2}{blue1}
            \colourcirc{1}{0}{blue1}
            \colourcirc{3}{0}{blue1}
            \colourcirc{5}{0}{blue1}
            
            \draw(3,-5)node{$x_1^2x_2x_3^2$};
        \end{scope}
        \begin{scope}[shift={(24,0)}]
            \blueweight{1}{1}
            \blueweight{2}{1}
            \blueweight{3}{1}
            \blueweight{2}{2}
            \blueweight{3}{2}
            \draw[blue1 line](1,0)\vver\dia;
            \draw[blue1 line](3,0)\vver\vver\ddia\dia;
            \draw[blue1 line](5,0)\vver\vver\ddia\ddia\dia;
            
            \draw[gray](0,0)--++(6,0)--++(5,+\v*5)--++(-6,0)--cycle;
            \foreach \i in {1,...,2} \draw[gray](\i*2,0)--++(5,+\v*5);
            \foreach \i in {1,...,4} \draw[gray](\i,+\v*\i)--++(6,0);
            \draw(3/2,+\v*3/2)node[left]{\smfont$1$};
            \draw(7/2,+\v*7/2)node[left]{\smfont$1$};
            \draw(9/2,+\v*9/2)node[left]{\smfont$1$};
            \draw(1,0)node[below]{\smfont$1$};
            \draw(3,0)node[below]{\smfont$1$};
            \draw(5,0)node[below]{\smfont$1$};
            \colourcirc{3/2}{+\v*3/2}{blue1}
            \colourcirc{7/2}{+\v*7/2}{blue1}
            \colourcirc{9/2}{+\v*9/2}{blue1}
            \colourcirc{1}{0}{blue1}
            \colourcirc{3}{0}{blue1}
            \colourcirc{5}{0}{blue1}
            
            \draw(3,-5)node{$x_1x_2^2x_3^2$};
        \end{scope}
    \end{tikzpicture}
}

\newcommand{\demazuretiles}[1]{
    \begin{tikzpicture}[x=#1,y=#1]
        \drawrightshear{0}{0}{0}{0}{}
        
        \begin{scope}[shift={(5,0)}]
            \drawrightshear{0}{i}{0}{i}{
                \Rrdia{blue1 line}
            }
        \end{scope}
        
        \begin{scope}[shift={(10,0)}]
            \drawrightshear{i}{0}{i}{0}{
                \Rldia{blue1 line}
            }
        \end{scope}
        
        \begin{scope}[shift={(15,0)}]
            \drawrightshear{j}{i}{j}{i}{
                \Rrdia{blue1 line}
                \Rldia{blue3 line}
            }
        \end{scope}
        
        \begin{scope}[shift={(20,0)}]
            \drawrightshear{g}{f}{f}{g}{
                \Rver{blue1 line}
                \Rhor{blue3 line}
            }
        \end{scope}
        
        \begin{scope}[shift={(25,0)}]
            \drawrightshear{0}{i}{i}{0}{
                \blueweight{1}{1}
                \Rver{blue1 line}
            }
            \draw(1.5,-2)node{weight: $x_c$};
        \end{scope}
    \end{tikzpicture}
}

\newcommand{\composition}[1]{
    \begin{tikzpicture}[x=#1,y=#1]
        \filldraw[draw=gray,fill=light gray](0,0)--++(0,4)--++(2,0)--++(0,2)--++(4,0)--++(0,4)--++(-6,0)--cycle;
        \colourcirc{0}{1}{blue3}
        \colourcirc{0}{3}{blue1}
        \colourcirc{2}{5}{blue4}
        \colourcirc{6}{7}{blue5}
        \colourcirc{6}{9}{blue2}
        \draw(0,1)node[right]{\smfont$3$};
        \draw(0,3)node[right]{\smfont$1$};
        \draw(2,5)node[right]{\smfont$4$};
        \draw(6,7)node[right]{\smfont$5$};
        \draw(6,9)node[right]{\smfont$2$};
        \draw(1,4)node[above]{\smfont$0$};
        \draw(3,6)node[above]{\smfont$0$};
        \draw(5,6)node[above]{\smfont$0$};
        
        \draw[gray,dotted](0,6)--++(2,0);
        \draw[gray,dotted](0,8)--++(6,0);
        \draw[gray,dotted](2,10)--++(0,-4);
        \draw[gray,dotted](4,10)--++(0,-4);
        
        \draw(3,10)node[above]{$\al^*$};
    
        \begin{scope}[shift={(10,0)}]
            \filldraw[draw=gray,fill=light gray](0,0)--++(0,2)--++(2,0)--++(0,2)--++(2,0)--++(0,2)--++(4,0)--++(0,4)--++(-8,0)--cycle;
            \colourcirc{0}{1}{blue1}
            \colourcirc{2}{3}{blue1}
            \colourcirc{4}{5}{blue1}
            \colourcirc{8}{7}{blue1}
            \colourcirc{8}{9}{blue1}
            \draw(0,1)node[right]{\smfont$1$};
            \draw(2,3)node[right]{\smfont$1$};
            \draw(4,5)node[right]{\smfont$1$};
            \draw(8,7)node[right]{\smfont$1$};
            \draw(8,9)node[right]{\smfont$1$};
            \draw(1,2)node[above]{\smfont$0$};
            \draw(3,4)node[above]{\smfont$0$};
            \draw(5,6)node[above]{\smfont$0$};
            \draw(7,6)node[above]{\smfont$0$};
            
            \draw[gray,dotted](0,4)--++(2,0);
            \draw[gray,dotted](0,6)--++(4,0);
            \draw[gray,dotted](0,8)--++(8,0);
            \draw[gray,dotted](2,10)--++(0,-6);
            \draw[gray,dotted](4,10)--++(0,-4);
            \draw[gray,dotted](6,10)--++(0,-4);
            
            \draw(4,10)node[above]{$\la^-$};
        \end{scope}
    
        \begin{scope}[shift={(22,0)}]
            \filldraw[draw=gray,fill=light gray](0,0)--++(0,2)--++(2,0)--++(0,2)--++(2,0)--++(0,2)--++(4,0)--++(0,4)--++(-8,0)--cycle;
            \colourcirc{1}{2}{red}
            \colourcirc{3}{4}{red}
            \colourcirc{5}{6}{red}
            \colourcirc{7}{6}{red}
            \draw(0,1)node[right]{\smfont$0$};
            \draw(2,3)node[right]{\smfont$0$};
            \draw(4,5)node[right]{\smfont$0$};
            \draw(8,7)node[right]{\smfont$0$};
            \draw(8,9)node[right]{\smfont$0$};
            \draw(1,2)node[above]{\smfont$+$};
            \draw(3,4)node[above]{\smfont$+$};
            \draw(5,6)node[above]{\smfont$+$};
            \draw(7,6)node[above]{\smfont$+$};
            
            \draw[gray,dotted](0,4)--++(2,0);
            \draw[gray,dotted](0,6)--++(4,0);
            \draw[gray,dotted](0,8)--++(8,0);
            \draw[gray,dotted](2,10)--++(0,-6);
            \draw[gray,dotted](4,10)--++(0,-4);
            \draw[gray,dotted](6,10)--++(0,-4);
            
            \draw(4,10)node[above]{$\la^+$};
        \end{scope}
    \end{tikzpicture}
}

\newcommand{\asmodels}[1]{
    \begin{tikzpicture}[x=#1,y=#1]
        \filldraw[gray fill](0,0)--++(6,0)--++(4,+\v*4)--++(-6,0)--cycle;
        \perparrow{0}{0}{4}{\vx*4}{-1}{left}{$\al^*$}
        \perparrow{0}{0}{6}{0}{1}{below}{$\si^{-1}$}
        \draw(-1,+\v*2)node[left]{$\A_\al^\si(x):=$};
        
        \begin{scope}[shift={(18,0)}]
            \filldraw[gray fill](0,0)--++(6,0)--++(4,+\v*4)--++(-6,0)--cycle;
            \perparrow{0}{0}{4}{\vx*4}{-1}{left}{$\la^-$}
            \perparrow{0}{0}{6}{0}{1}{below}{$\bacirc^n$}
            \draw(-1,+\v*2)node[left]{$s_\la(x):=$};
        \end{scope}
    \end{tikzpicture}
}

\def\diamondspace{5}
\newcommand{\diamonds}[1]{
    \begin{tikzpicture}[x=#1,y=#1]
        \drawdiamond{i}{0}{0}{i}{
            \Drdia{blue1 line}
        }
    
        \begin{scope}[shift={(\diamondspace,0)}]
            \drawdiamond{i}{0}{+}{i^+}{
                \Drdia{blue1 line}
                \Dbhor{red line}
            }
        \end{scope}
    
        \begin{scope}[shift={(2*\diamondspace,0)}]
            \drawdiamond{i^+}{+}{0}{i}{
                \Drdia{blue1 line}
                \Dthor{red line}
            }
        \end{scope}
        
        \begin{scope}[shift={(3*\diamondspace,0)}]
            \drawdiamond{i^+}{+}{+}{i^+}{
                \Drdia{blue1 line}
                \Dthor{red line}
                \Dbhor{red line}
            }
        \end{scope}
    
        \begin{scope}[shift={(4*\diamondspace,0)}]
            \drawdiamond{g}{fg}{0}{f}{
                \Drelb{blue1 line}
                \Dthor{blue3 line}
            }
        \end{scope}
    
        \begin{scope}[shift={(5*\diamondspace,0)}]
            \drawdiamond{f}{0}{fg}{g}{
                \Dlelb{blue1 line}
                \Dbhor{blue3 line}
            }
        \end{scope}
        
        \begin{scope}[shift={(6*\diamondspace,0)}]
            \drawdiamond{g}{fg}{fh}{h}{
                \Dldia{blue1 line}
                \Dbhor{blue2 line}
                \Dthor{blue3 line}
            }
        \end{scope}
        
        \begin{scope}[shift={(0,-9/2)}]
            \drawdiamond{0}{+}{+}{0}{
                \Dldia{red line}
            }
            
            \begin{scope}[shift={(\diamondspace,0)}]
                \drawdiamond{0}{+}{i^+}{i}{
                    \Dbhor{blue1 line}
                    \Dldia{red line}
                }
            \end{scope}
            
            \begin{scope}[shift={(2*\diamondspace,0)}]
                \drawdiamond{i}{i^+}{+}{0}{
                    \Dthor{blue1 line}
                    \Dldia{red line}
                }
            \end{scope}
            
            \begin{scope}[shift={(3*\diamondspace,0)}]
                \drawdiamond{i}{i^+}{k^+}{k}{
                    \Dthor{blue1 line}
                    \Dbhor{blue3 line}
                    \Dldia{red line}
                }
            \end{scope}
            
            \begin{scope}[shift={(4*\diamondspace,0)}]
                \drawdiamond{g}{fg}{+}{f^+}{
                    \Drelb{blue1 line}
                    \Dthor{blue3 line}
                    \Dbhor{red line}
                }
            \end{scope}
            
            \begin{scope}[shift={(5*\diamondspace,0)}]
                \drawdiamond{f^+}{+}{fg}{g}{
                    \Dlelb{blue1 line}
                    \Dbhor{blue3 line}
                    \Dthor{red line}
                }
            \end{scope}
            
            \begin{scope}[shift={(6*\diamondspace,0)}]
                \drawdiamond{0}{0}{0}{0}{}
            \end{scope}
        \end{scope}
    \end{tikzpicture}
}

\newcommand{\banned}[1]{
    \begin{tikzpicture}[x=#1,y=#1]
        \begin{scope}
            \clip (0,0)--++(2,+\v*2)--++(1,-\v)--++(-2,-\v*2)--cycle;
            \draw[blue1 line](2,-\v)\ddia\ddia;
            \draw[blue3 line](3,0)\ddia\ddia;
        \end{scope}
        \filldraw[light gray](1,\v)--++(2,0)--++(-1,\v)--cycle;
        \filldraw[light gray](0,0)--++(2,0)--++(-1,-\v)--cycle;
        \draw[gray](0,0)--++(2,+\v*2)--++(1,-\v)--++(-2,-\v*2)--cycle;
        \draw[gray](2,0)--++(-1,\v);
        \draw[gray, dotted](0,0)--++(2,0);
        \draw[gray, dotted](1,\v)--++(2,0);
        \draw(1/2,+\v/2)node[left]{\smfont$i$};
        \draw(5/2,+\v/2)node[right]{\smfont$j$};
        
        \begin{scope}[shift={(6,0)}]
            \begin{scope}
                \clip (0,0)--++(2,+\v*2)--++(1,-\v)--++(-2,-\v*2)--cycle;
                \draw[blue1 line](2,-\v)\ddia\ddia;
                \draw[blue3 line](3,0)\ddia\ddia;
                \draw[red line](3,+\v/2)\hhor\hhor;
            \end{scope}
            \filldraw[light gray](1,\v)--++(2,0)--++(-1,\v)--cycle;
            \filldraw[light gray](0,0)--++(2,0)--++(-1,-\v)--cycle;
            \draw[gray](0,0)--++(2,+\v*2)--++(1,-\v)--++(-2,-\v*2)--cycle;
            \draw[gray](2,0)--++(-1,\v);
            \draw[gray, dotted](0,0)--++(2,0);
            \draw[gray, dotted](1,\v)--++(2,0);
            \draw(1/2,+\v/2)node[left]{\smfont$i^+$};
            \draw(5/2,+\v/2)node[right]{\smfont$j^+$};
        \end{scope}
    \end{tikzpicture}
}

\newcommand{\producttheorem}[1]{
    \begin{tikzpicture}[x=#1,y=#1]
        \filldraw[gray fill](0,0)--++(4,-\v*4)--++(4,+\v*4)--++(-4,+\v*4)--cycle;
        \draw(4,0)node{$a_{\al\la}^\be(\si)$};
        \perparrow{0}{0}{4}{\vx*4}{-1}{left}{$\al^*$}
        \perparrow{4}{-\vx*4}{8}{0}{1}{right}{$\be^*$}
        \perparrow{4}{-\vx*4}{0}{0}{-1}{left}{$\la^+$}
        \perparrow{8}{0}{4}{\vx*4}{1}{right}{$\rcirc^k\wcirc^n$}
    \end{tikzpicture}
}

\newcommand{\atomdiamondex}[1]{
    \begin{tikzpicture}[x=#1,y=#1]
        \draw[blue1 line](21/2,-\v*11/2)\ddia\ddia\ddia\rel\lel\ddia\ddia\ddia\ddia;
        \draw[blue2 line](31/2,-\v*1/2)\hhor\ddia\ddia\ddia\ddia\ddia\ddia\ddia;
        \draw[blue3 line](27/2,-\v*5/2)\hhor\ddia\hhor\hhor\ddia\ddia\ddia\ddia;
        \draw[blue4 line](19/2,-\v*13/2)\ddia\ddia\ddia\ddia\ddia\hhor\ddia\ddia;
        \draw[red line](1/2,-\v*1/2)\rhhor\vver\rhhor\rhhor\vver\vver\rhhor\vver;
        \draw[red line](5/2,-\v*5/2)\vver\vver\rhhor\rhhor\vver\vver\rhhor\vver;
        \draw[red line](7/2,-\v*7/2)\vver\rhhor\rhhor\vver\vver\vver\rhhor\vver;
        \draw[red line](11/2,-\v*11/2)\vver\rhhor\rhhor\vver\vver\vver\vver\vver;
        \draw[gray](0,0)--++(8,-\v*8)--++(8,+\v*8)--++(-8,+\v*8)--cycle;
        \foreach \i in {1,...,7}{
            \draw[gray](\i,-\v*\i)--++(8,+\v*8);
            \draw[gray](\i,+\v*\i)--++(8,-\v*8);
        }
        \colourcirc{7/2}{+\v*7/2}{blue1}
        \colourcirc{13/2}{+\v*13/2}{blue2}
        \colourcirc{5/2}{+\v*5/2}{blue3}
        \colourcirc{1/2}{+\v*1/2}{blue4}
        \colourcirc{21/2}{-\v*11/2}{blue1}
        \colourcirc{31/2}{-\v*1/2}{blue2}
        \colourcirc{27/2}{-\v*5/2}{blue3}
        \colourcirc{19/2}{-\v*13/2}{blue4}
        \colourcirc{1/2}{-\v*1/2}{red}
        \colourcirc{5/2}{-\v*5/2}{red}
        \colourcirc{7/2}{-\v*7/2}{red}
        \colourcirc{11/2}{-\v*11/2}{red}
        \colourcirc{31/2}{+\v*1/2}{red}
        \colourcirc{29/2}{+\v*3/2}{red}
        \colourcirc{27/2}{+\v*5/2}{red}
        \colourcirc{25/2}{+\v*7/2}{red}
        \draw(7/2,+\v*7/2)node[left]{\smfont$1$};
        \draw(13/2,+\v*13/2)node[left]{\smfont$2$};
        \draw(5/2,+\v*5/2)node[left]{\smfont$3$};
        \draw(1/2,+\v*1/2)node[left]{\smfont$4$};
        \draw(21/2,-\v*11/2)node[right]{\smfont$1$};
        \draw(31/2,-\v*1/2)node[right]{\smfont$2$};
        \draw(27/2,-\v*5/2)node[right]{\smfont$3$};
        \draw(19/2,-\v*13/2)node[right]{\smfont$4$};
        \draw(1/2,-\v*1/2)node[left]{\smfont$+$};
        \draw(5/2,-\v*5/2)node[left]{\smfont$+$};
        \draw(7/2,-\v*7/2)node[left]{\smfont$+$};
        \draw(11/2,-\v*11/2)node[left]{\smfont$+$};
        \draw(31/2,+\v*1/2)node[right]{\smfont$+$};
        \draw(29/2,+\v*3/2)node[right]{\smfont$+$};
        \draw(27/2,+\v*5/2)node[right]{\smfont$+$};
        \draw(25/2,+\v*7/2)node[right]{\smfont$+$};
        
        \perparrow{0}{0}{8}{\vx*8}{-2}{left}{$\al^*$}
        \perparrow{8}{-\vx*8}{0}{0}{-2}{left}{$\la^+$}
        \perparrow{16}{0}{8}{\vx*8}{2}{right}{$\rcirc^4\wcirc^4$}
        \perparrow{8}{-\vx*8}{16}{0}{2}{right}{$\be^*$}
        
        \begin{scope}[shift={(25,0)}]
            \draw[blue1 line](21/2,-\v*11/2)\ddia\ddia\ddia\ddia\rel\lel\ddia\ddia\ddia;
            \draw[blue2 line](31/2,-\v*1/2)\hhor\ddia\ddia\ddia\ddia\ddia\ddia\ddia;
            \draw[blue3 line](27/2,-\v*5/2)\ddia\ddia\hhor\hhor\hhor\ddia\ddia\ddia;
            \draw[blue4 line](19/2,-\v*13/2)\ddia\ddia\hhor\ddia\ddia\ddia\ddia\ddia;
            \draw[red line](1/2,-\v*1/2)\rhhor\vver\rhhor\rhhor\vver\vver\rhhor\vver;
            \draw[red line](5/2,-\v*5/2)\rhhor\vver\rhhor\vver\vver\vver\rhhor\vver;
            \draw[red line](7/2,-\v*7/2)\rhhor\vver\rhhor\vver\vver\vver\rhhor\vver;
            \draw[red line](11/2,-\v*11/2)\vver\vver\rhhor\vver\vver\rhhor\vver\vver;
            \draw[gray](0,0)--++(8,-\v*8)--++(8,+\v*8)--++(-8,+\v*8)--cycle;
            \foreach \i in {1,...,7}{
                \draw[gray](\i,-\v*\i)--++(8,+\v*8);
                \draw[gray](\i,+\v*\i)--++(8,-\v*8);
            }
            \colourcirc{7/2}{+\v*7/2}{blue1}
            \colourcirc{13/2}{+\v*13/2}{blue2}
            \colourcirc{5/2}{+\v*5/2}{blue3}
            \colourcirc{1/2}{+\v*1/2}{blue4}
            \colourcirc{21/2}{-\v*11/2}{blue1}
            \colourcirc{31/2}{-\v*1/2}{blue2}
            \colourcirc{27/2}{-\v*5/2}{blue3}
            \colourcirc{19/2}{-\v*13/2}{blue4}
            \colourcirc{1/2}{-\v*1/2}{red}
            \colourcirc{5/2}{-\v*5/2}{red}
            \colourcirc{7/2}{-\v*7/2}{red}
            \colourcirc{11/2}{-\v*11/2}{red}
            \colourcirc{31/2}{+\v*1/2}{red}
            \colourcirc{29/2}{+\v*3/2}{red}
            \colourcirc{27/2}{+\v*5/2}{red}
            \colourcirc{25/2}{+\v*7/2}{red}
            \draw(7/2,+\v*7/2)node[left]{\smfont$1$};
            \draw(13/2,+\v*13/2)node[left]{\smfont$2$};
            \draw(5/2,+\v*5/2)node[left]{\smfont$3$};
            \draw(1/2,+\v*1/2)node[left]{\smfont$4$};
            \draw(21/2,-\v*11/2)node[right]{\smfont$1$};
            \draw(31/2,-\v*1/2)node[right]{\smfont$2$};
            \draw(27/2,-\v*5/2)node[right]{\smfont$3$};
            \draw(19/2,-\v*13/2)node[right]{\smfont$4$};
            \draw(1/2,-\v*1/2)node[left]{\smfont$+$};
            \draw(5/2,-\v*5/2)node[left]{\smfont$+$};
            \draw(7/2,-\v*7/2)node[left]{\smfont$+$};
            \draw(11/2,-\v*11/2)node[left]{\smfont$+$};
            \draw(31/2,+\v*1/2)node[right]{\smfont$+$};
            \draw(29/2,+\v*3/2)node[right]{\smfont$+$};
            \draw(27/2,+\v*5/2)node[right]{\smfont$+$};
            \draw(25/2,+\v*7/2)node[right]{\smfont$+$};
        \end{scope}
    \end{tikzpicture}
}

\newcommand{\columnlemma}[1]{
    \begin{tikzpicture}[x=#1,y=#1]
        \filldraw[light gray](0,0)--++(1,-\v)--++(2,0)--++(4,+\v*4)--++(-1,\v)--++(-2,0)--cycle;
        \draw[gray](0,0)--++(1,-\v)--++(2,0)--++(2,+\v*2)--++(-1,\v)--++(-2,0)--cycle;
        \draw[gray](3,+\v*3)--++(2,0)--++(1,-\v)--++(1,+\v)--++(-1,\v)--++(-2,0)--cycle;
        \draw[gray](2,0)--++(2,+\v*2);
        \draw[gray](5,+\v*3)--++(1,+\v);
        \draw[gray](0,0)--++(2,0)--++(1,-\v);
        \draw[gray](1,+\v)--++(2,0)--++(1,-\v);
        \draw[gray](2,+\v*2)--++(2,0)--++(1,-\v);
        \draw[gray, dashed](2,+\v*2)--++(1,\v);
        \draw[gray, dashed](4,+\v*2)--++(1,\v);
        \draw[gray, dashed](5,\v)--++(1,\v);
        
        \draw(1/2,+\v*1/2)node[left]{\smfont$q_1$};
        \draw(3/2,+\v*3/2)node[left]{\smfont$q_2$};
        \draw(7/2,+\v*7/2)node[left]{\smfont$q_m$};
        \draw(7/2,-\v*1/2)node[right]{\smfont$t_1$};
        \draw(9/2,+\v*1/2)node[right]{\smfont$t_2$};
        \draw(13/2,+\v*5/2)node[right]{\smfont$t_m$};
        \draw(5,+\v*4)node[above]{\smfont$v$};
        \draw(2,-\v)node[below]{\smfont$s$};
        \draw(1/2,-\v/2)node[left]{\smfont$r$};
        \draw(13/2,+\v*7/2)node[right]{\smfont$u$};
        \begin{scope}[shift={(11,0)}]
            \filldraw[light gray](0,0)--++(1,-\v)--++(2,0)--++(4,+\v*4)--++(-1,\v)--++(-2,0)--cycle;
            \draw[gray](0,0)--++(1,-\v)--++(2,0)--++(2,+\v*2)--++(-2,0)--++(-1,\v)--cycle;
            \draw[gray](3,+\v*3)--++(1,-\v)--++(2,0)--++(1,+\v)--++(-1,\v)--++(-2,0)--cycle;
            \draw[gray](1,-\v)--++(2,+\v*2);
            \draw[gray](4,+\v*2)--++(1,+\v);
            \draw[gray](1,+\v)--++(1,-\v)--++(2,0);
            \draw[gray](4,+\v*4)--++(1,-\v)--++(2,0);
            \draw[gray, dashed](2,+\v*2)--++(1,\v);
            \draw[gray, dashed](3,+\v)--++(1,\v);
            \draw[gray, dashed](5,\v)--++(1,\v);
            
            \draw(1/2,+\v*1/2)node[left]{\smfont$q_1$};
            \draw(3/2,+\v*3/2)node[left]{\smfont$q_2$};
            \draw(7/2,+\v*7/2)node[left]{\smfont$q_m$};
            \draw(7/2,-\v*1/2)node[right]{\smfont$t_1$};
            \draw(9/2,+\v*1/2)node[right]{\smfont$t_2$};
            \draw(13/2,+\v*5/2)node[right]{\smfont$t_m$};
            \draw(5,+\v*4)node[above]{\smfont$v$};
            \draw(2,-\v)node[below]{\smfont$s$};
            \draw(1/2,-\v/2)node[left]{\smfont$r$};
            \draw(13/2,+\v*7/2)node[right]{\smfont$u$};
        \end{scope}
    
        \draw(9,+\v*3/2)node{$=$};
    \end{tikzpicture}
}

\newcommand{\columnlemmarestated}[1]{
    \begin{tikzpicture}[x=#1,y=#1]
        \filldraw[light gray](0,0)--++(1,-\v)--++(2,0)--++(4,+\v*4)--++(-1,\v)--++(-2,0)--cycle;
        \draw[gray](0,0)--++(1,-\v)--++(2,0)--++(2,+\v*2)--++(-1,\v)--++(-2,0)--cycle;
        \draw[gray](3,+\v*3)--++(2,0)--++(1,-\v)--++(1,+\v)--++(-1,\v)--++(-2,0)--cycle;
        \draw[gray](2,0)--++(2,+\v*2);
        \draw[gray](5,+\v*3)--++(1,+\v);
        \draw[gray](0,0)--++(2,0)--++(1,-\v);
        \draw[gray](1,+\v)--++(2,0)--++(1,-\v);
        \draw[gray](2,+\v*2)--++(2,0)--++(1,-\v);
        \draw[gray, dashed](2,+\v*2)--++(1,\v);
        \draw[gray, dashed](4,+\v*2)--++(1,\v);
        \draw[gray, dashed](5,\v)--++(1,\v);
        
        \draw(1/2,+\v*1/2)node[left]{\smfont$q_1$};
        \draw(3/2,+\v*3/2)node[left]{\smfont$q_2$};
        \draw(7/2,+\v*7/2)node[left]{\smfont$q_m$};
        \draw(7/2,-\v*1/2)node[right]{\smfont$t_1$};
        \draw(9/2,+\v*1/2)node[right]{\smfont$t_2$};
        \draw(13/2,+\v*5/2)node[right]{\smfont$t_m$};
        \draw(5,+\v*4)node[above]{\smfont$v$};
        \draw(2,-\v)node[below]{\smfont$s$};
        \draw(1/2,-\v/2)node[left]{\smfont$r$};
        \draw(13/2,+\v*7/2)node[right]{\smfont$u$};
        
        \whiteredcirc{1/2}{-\v/2}
        \whiteredcirc{13/2}{+\v*7/2}
        
        \begin{scope}[shift={(11,0)}]
            \filldraw[light gray](0,0)--++(1,-\v)--++(2,0)--++(4,+\v*4)--++(-1,\v)--++(-2,0)--cycle;
            \draw[gray](0,0)--++(1,-\v)--++(2,0)--++(2,+\v*2)--++(-2,0)--++(-1,\v)--cycle;
            \draw[gray](3,+\v*3)--++(1,-\v)--++(2,0)--++(1,+\v)--++(-1,\v)--++(-2,0)--cycle;
            \draw[gray](1,-\v)--++(2,+\v*2);
            \draw[gray](4,+\v*2)--++(1,+\v);
            \draw[gray](1,+\v)--++(1,-\v)--++(2,0);
            \draw[gray](4,+\v*4)--++(1,-\v)--++(2,0);
            \draw[gray, dashed](2,+\v*2)--++(1,\v);
            \draw[gray, dashed](3,+\v)--++(1,\v);
            \draw[gray, dashed](5,\v)--++(1,\v);
            
            \draw(1/2,+\v*1/2)node[left]{\smfont$q_1$};
            \draw(3/2,+\v*3/2)node[left]{\smfont$q_2$};
            \draw(7/2,+\v*7/2)node[left]{\smfont$q_m$};
            \draw(7/2,-\v*1/2)node[right]{\smfont$t_1$};
            \draw(9/2,+\v*1/2)node[right]{\smfont$t_2$};
            \draw(13/2,+\v*5/2)node[right]{\smfont$t_m$};
            \draw(5,+\v*4)node[above]{\smfont$v$};
            \draw(2,-\v)node[below]{\smfont$s$};
            \draw(1/2,-\v/2)node[left]{\smfont$r$};
            \draw(13/2,+\v*7/2)node[right]{\smfont$u$};
            
            \whiteredcirc{1/2}{-\v/2}
            \whiteredcirc{13/2}{+\v*7/2}
        \end{scope}
        
        \draw(9,+\v*3/2)node{$=$};
    \end{tikzpicture}
}

\newcommand{\prisms}[1]{
    \begin{tikzpicture}[x=#1,y=#1]
        \filldraw[gray fill](0,0)--++(3,-\v*3)--++(4,0)--++(6,+\v*6)--++(-3,+\v*3)--++(-4,0)--cycle;
        \draw[gray](0,0)--++(4,0)--++(3,-\v*3);
        \draw[gray](3,+\v*3)--++(4,0)--++(3,-\v*3);
        \draw[gray](4,0)--++(6,+\v*6);
        \perparrow{0}{0}{3}{\vx*3}{-1}{left}{$\al^*$}
        \perparrow{3}{-\vx*3}{0}{0}{-1}{left}{$\la^+$}
        \perparrow{3}{-\vx*3}{7}{-\vx*3}{1}{below}{\smfont$\si^{-1}$}
        \perparrow{6}{\vx*6}{10}{\vx*6}{-1}{above}{$\bacirc^n$}
        \perparrow{13}{\vx*3}{10}{\vx*6}{1}{right}{$\rcirc^k\wcirc^n$}
        \perparrow{10}{0}{13}{\vx*3}{1}{right}{$\wcirc^k\bacirc^n$}
        
        \begin{scope}[shift={(21,0)}]
            \filldraw[gray fill](0,0)--++(3,-\v*3)--++(4,0)--++(6,+\v*6)--++(-3,+\v*3)--++(-4,0)--cycle;
            \draw[gray](3,+\v*3)--++(3,-\v*3)--++(4,0);
            \draw[gray](6,+\v*6)--++(3,-\v*3)--++(4,0);
            \draw[gray](3,-\v*3)--++(6,+\v*6);
            \perparrow{0}{0}{3}{\vx*3}{-1}{left}{$\al^*$}
            \perparrow{3}{-\vx*3}{0}{0}{-1}{left}{$\la^+$}
            \perparrow{3}{-\vx*3}{7}{-\vx*3}{1}{below}{\smfont$\si^{-1}$}
            \perparrow{6}{\vx*6}{10}{\vx*6}{-1}{above}{$\bacirc^n$}
            \perparrow{13}{\vx*3}{10}{\vx*6}{1}{right}{$\rcirc^k\wcirc^n$}
            \perparrow{10}{0}{13}{\vx*3}{1}{right}{$\wcirc^k\bacirc^n$}
        \end{scope}
    
        \draw(17,+\v*3/2)node{$=$};
    \end{tikzpicture}
}

\def\allshearspace{5}
\newcommand{\xtiles}[1]{
    \begin{tikzpicture}[x=#1,y=#1]
        \drawrightshear{0}{0}{0}{0}{}
        
        \begin{scope}[shift={(\allshearspace,0)}]
            \drawrightshear{0}{i}{0}{i}{
                \Rrdia{blue1 line}
            }
        \end{scope}
        
        \begin{scope}[shift={(2*\allshearspace,0)}]
            \drawrightshear{i}{0}{i}{0}{
                \Rldia{blue1 line}
            }
        \end{scope}
        
        \begin{scope}[shift={(3*\allshearspace,0)}]
            \drawrightshear{j}{i}{j}{i}{
                \Rrdia{blue1 line}
                \Rldia{blue3 line}
            }
        \end{scope}
        
        \begin{scope}[shift={(4*\allshearspace,0)}]
            \drawrightshear{g}{f}{f}{g}{
                \Rver{blue1 line}
                \Rhor{blue3 line}
            }
        \end{scope}
        
        \begin{scope}[shift={(5*\allshearspace,0)}]
            \drawrightshear{i}{+}{+}{i}{
                \Rhor{blue1 line}
                \Rver{red line}
            }
        \end{scope}
        
        \begin{scope}[shift={(0,-9/2)}]
            \drawrightshear{0}{+}{+}{0}{
                \Rver{red line}
            }
            
            \begin{scope}[shift={(\allshearspace,0)}]
                \drawrightshear{0}{i}{+}{i^+}{
                    \Rrdia{blue1 line}
                    \Rrelb{red line}
                }
            \end{scope}
            
            \begin{scope}[shift={(2*\allshearspace,0)}]
                \drawrightshear{i^+}{+}{i}{0}{
                    \Rldia{blue1 line}
                    \Rlelb{red line}
                }
            \end{scope}
            
            \begin{scope}[shift={(3*\allshearspace,0)}]
                \drawrightshear{j^+}{i}{j}{i^+}{
                    \Rrdia{blue1 line}
                    \Rldia{blue3 line}
                    \Rhor{red line}
                }
            \end{scope}
            
            \begin{scope}[shift={(4*\allshearspace,0)}]
                \drawrightshear{g^+}{f}{f}{g^+}{
                    \Rver{blue1 line}
                    \Rhord{blue3 line}
                }
            \end{scope}
        
            \begin{scope}[shift={(5*\allshearspace,0)}]
                \drawrightshear{0}{i}{i}{0}{
                    \blueweight{1}{1}
                    \Rver{blue1 line}
                }
                \draw(1.5,-2)node{weight: $x_c$};
            \end{scope}
        \end{scope}
        
        \draw(0,-\v*9/2)node{};
    \end{tikzpicture}
}

\newcommand{\ytiles}[1]{
    \begin{tikzpicture}[x=#1,y=#1]
        \drawleftshear{i}{0}{0}{i}{
            \Lver{blue1 line}
        }
        
        \begin{scope}[shift={(\allshearspace,0)}]
            \drawleftshear{i}{i^+}{0}{+}{
                \Lrelb{blue1 line}
                \Lrdia{red line}
            }
        \end{scope}
        
        \begin{scope}[shift={(2*\allshearspace,0)}]
            \drawleftshear{+}{0}{i^+}{i}{
                \Llelb{blue1 line}
                \Lldia{red line}
            }
        \end{scope}
        
        \begin{scope}[shift={(3*\allshearspace,0)}]
            \drawleftshear{+}{i^+}{i^+}{+}{
                \Lhor{blue1 line}
                \Lrdia{red line}
                \Lldia{red line}
            }
        \end{scope}
        
        \begin{scope}[shift={(4*\allshearspace,0)}]
            \drawleftshear{i}{+}{+}{i}{
                \Lver{blue1 line}
                \Lhor{red line}
            }
        \end{scope}
        
        \begin{scope}[shift={(0,-9/2)}]
            \drawleftshear{0}{0}{0}{0}{}
            
            \begin{scope}[shift={(\allshearspace,0)}]
                \drawleftshear{0}{+}{0}{+}{
                    \Lrdia{red line}
                }
            \end{scope}
            
            \begin{scope}[shift={(2*\allshearspace,0)}]
                \drawleftshear{+}{0}{+}{0}{
                    \Lldia{red line}
                }
            \end{scope}
            
            \begin{scope}[shift={(3*\allshearspace,0)}]
                \drawleftshear{+}{+}{+}{+}{
                    \Lrdia{red line}
                    \Lldia{red line}
                }
            \end{scope}
            
            \begin{scope}[shift={(4*\allshearspace,0)}]
                \drawleftshear{+}{0}{0}{+}{
                    \fill[red weight](1,0)--++(2,0)--++(-1,\v)--++(-2,0)--cycle;
                    \Lver{red line}
                }
                \draw(1.5,-2)node{weight: $-x_c$};
            \end{scope}
        \end{scope}
        \draw(0,-\v*5)node{};
    \end{tikzpicture}
}

\def\alldiamondspace{5}
\def\diamondrowspace{9/2}
\newcommand{\ztiles}[1]{
    \begin{tikzpicture}[x=#1,y=#1]
        \begin{scope}[shift={(-\alldiamondspace/2,0)}]
            \drawdiamond{0}{0}{0}{0}{}
            
            \begin{scope}[shift={(\alldiamondspace,0)}]
                \drawdiamond{0}{+}{+}{0}{
                    \Dldia{red line}
                }
            \end{scope}
            
            \begin{scope}[shift={(2*\alldiamondspace,0)}]
                \drawdiamond{0}{+}{i^+}{i}{
                    \Dbhor{blue1 line}
                    \Dldia{red line}
                }
            \end{scope}
            
            \begin{scope}[shift={(3*\alldiamondspace,0)}]
                \drawdiamond{i}{i^+}{+}{0}{
                    \Dthor{blue1 line}
                    \Dldia{red line}
                }
            \end{scope}
            
            \begin{scope}[shift={(4*\alldiamondspace,0)}]
                \drawdiamond{i}{i^+}{k^+}{k}{
                    \Dthor{blue1 line}
                    \Dbhor{blue3 line}
                    \Dldia{red line}
                }
            \end{scope}
        \end{scope}
        
        \begin{scope}[shift={(0,-\diamondrowspace)}]
            \drawdiamond{i}{0}{0}{i}{
                \Drdia{blue1 line}
            }
            
            \begin{scope}[shift={(\alldiamondspace,0)}]
                \drawdiamond{g}{fg}{0}{f}{
                    \Drelb{blue1 line}
                    \Dthor{blue3 line}
                }
            \end{scope}
            
            \begin{scope}[shift={(2*\alldiamondspace,0)}]
                \drawdiamond{f}{0}{fg}{g}{
                    \Dlelb{blue1 line}
                    \Dbhor{blue3 line}
                }
            \end{scope}
            
            \begin{scope}[shift={(3*\alldiamondspace,0)}]
                \drawdiamond{g}{fg}{fh}{h}{
                    \Dldia{blue1 line}
                    \Dthor{blue2 line}
                    \Dbhor{blue3 line}
                }
            \end{scope}
        \end{scope}
        
        \begin{scope}[shift={(0,-2*\diamondrowspace)}]
            \drawdiamond{i}{0}{+}{i^+}{
                \Drdia{blue1 line}
                \Dbhor{red line}
            }
            
            \begin{scope}[shift={(\alldiamondspace,0)}]
                \drawdiamond{g}{fg}{+}{f^+}{
                    \Drelb{blue1 line}
                    \Dthor{blue3 line}
                    \Dbhor{red line}
                }
            \end{scope}
            
            \begin{scope}[shift={(2*\alldiamondspace,0)}]
                \drawdiamond{f}{0}{fg^+}{g^+}{
                    \Dlelb{blue1 line}
                    \Dbhord{blue3 line}
                }
            \end{scope}
            
            \begin{scope}[shift={(3*\alldiamondspace,0)}]
                \drawdiamond{g}{fg}{fh^+}{h^+}{
                    \Dldia{blue1 line}
                    \Dthor{blue2 line}
                    \Dbhord{blue3 line}
                }
            \end{scope}
        \end{scope}
        
        \begin{scope}[shift={(0,-3*\diamondrowspace)}]
            \drawdiamond{i^+}{+}{0}{i}{
                \Drdia{blue1 line}
                \Dthor{red line}
            }
            
            \begin{scope}[shift={(\alldiamondspace,0)}]
                \drawdiamond{g^+}{fg^+}{0}{f}{
                    \Drelb{blue1 line}
                    \Dthord{blue3 line}
                }
            \end{scope}
            
            \begin{scope}[shift={(2*\alldiamondspace,0)}]
                \drawdiamond{f^+}{+}{fg}{g}{
                    \Dlelb{blue1 line}
                    \Dbhor{blue3 line}
                    \Dthor{red line}
                }
            \end{scope}
            
            \begin{scope}[shift={(3*\alldiamondspace,0)}]
                \drawdiamond{g^+}{fg^+}{fh}{h}{
                    \Dldia{blue1 line}
                    \Dthord{blue2 line}
                    \Dbhor{blue3 line}
                }
            \end{scope}
        \end{scope}
        
        \begin{scope}[shift={(0,-4*\diamondrowspace)}]
            \drawdiamond{i^+}{+}{+}{i^+}{
                \Drdia{blue1 line}
                \Dthor{red line}
                \Dbhor{red line}
            }
            
            \begin{scope}[shift={(\alldiamondspace,0)}]
                \drawdiamond{g^+}{fg^+}{+}{f^+}{
                    \Drelb{blue1 line}
                    \Dthord{blue3 line}
                    \Dbhor{red line}
                }
            \end{scope}
            
            \begin{scope}[shift={(2*\alldiamondspace,0)}]
                \drawdiamond{f^+}{+}{fg^+}{g^+}{
                    \Dlelb{blue1 line}
                    \Dbhord{blue3 line}
                    \Dthor{red line}
                }
            \end{scope}
            
            \begin{scope}[shift={(3*\alldiamondspace,0)}]
                \drawdiamond{g^+}{fg^+}{fh^+}{h^+}{
                    \Dldia{blue1 line}
                    \Dthord{blue2 line}
                    \Dbhord{blue3 line}
                }
            \end{scope}
        \end{scope}
    \end{tikzpicture}
}

\newcommand{\columnexamplea}[1]{
    \begin{tikzpicture}[x=#1,y=#1]
        \blueweight{1}{4}
        \blueweight{1}{6}
        \clipcol{7}{
            \draw[blue2 line](15/2,+\v*5/2)\hhor\hor\dia\vver\vver\vver\ddia;
            \draw[blue1 line](17/2,+\v*5/2)\ddia\ddia\hhor\hhor;
            \draw[blue3 line](3,-\v*2)\ddia\ddia\vver\vver\ddia;
            \draw[blue4 line](9/2,-\v*3/2)\ddia\ddia\hhor\hhor;
            \draw[blue5 line](11/2,+\v*3/2)\hhor\hhor\hhor;
            \draw[red line](7/2,+\v*9/2)\rhhor\rhhor\vver\vver\vver;
            \draw[red line](-1/2,+\v/2)\rhhor\rhhor\vver\vver\vver\rhhor;
            
            \clipdouble{4}{+\v*4}{blue1 line}
            \clipdouble{0}{0}{blue4 line}
        }
        \BLOCKC{0}{}{3}{}{}
        \LCHEVC{0}{4^+}{4}{}{}
        \LCHEVC{1}{5}{}{}{}
        \LCHEVC{2}{3}{5}{}{}
        \LCHEVC{3}{}{2}{}{}
        \LCHEVC{4}{1^+}{1^+}{}{}
        \LCHEVC{5}{}{}{}{}
        \LCHEVC{6}{2}{}{}{+}
        
        \colourcirc{7/2}{-\v/2}{blue4}
        \colourcirc{11/2}{+\v*3/2}{blue5}
        \colourcirc{13/2}{+\v*5/2}{blue2}
        \colourcirc{2}{-\v}{blue3}
        \colourcirc{3/2}{+\v*3/2}{blue5}
        \colourcirc{5/2}{+\v*5/2}{blue3}
        \colourcirc{13/2}{+\v*13/2}{blue2}
        \redcirc{19/2}{+\v*13/2}
        \halfcirc{9/2}{+\v*9/2}{blue1}
        \halfcirc{15/2}{+\v*7/2}{blue1}
        \halfcirc{1/2}{+\v/2}{blue4}
        \draw(2,-5)node{$x_c^2$};
        
        \begin{scope}[shift={(14,0)}]
            \blueweight{2}{1}
            \blueweight{2}{5}
            \clipcol{7}{
                \draw[blue2 line](15/2,+\v*3/2)\ddia\dia\vver\vver\ddia\ddia;
                \draw[blue1 line](17/2,+\v*7/2)\hhor\hor\ddia\dia;
                \draw[blue3 line](1,-\v*2)\vver\vver\vver\vver\dia\hhor\hor;
                \draw[blue4 line](9/2,-\v/2)\hhor\hor\ddia\dia;
                \draw[blue5 line](11/2,+\v*3/2)\hhor\hhor\hhor;
                \draw[red line](7/2,+\v*9/2)\rhhor\vver\rhhor\vver\vver;
                \draw[red line](-1/2,+\v/2)\rhhor\vver\vver\vver\rhhor\rhhor;
                
                \clipdouble{5}{+\v*3}{blue1 line}
            }
            \RCHEVC{0}{4^+}{4}{}{3}
            \RCHEVC{1}{5}{}{}{}
            \RCHEVC{2}{3}{5}{}{}
            \RCHEVC{3}{}{2}{}{}
            \RCHEVC{4}{1^+}{1^+}{}{}
            \RCHEVC{5}{}{}{}{}
            \RCHEVC{6}{2}{}{}{}
            \BLOCKC{7}{}{}{}{+}
            
            \colourcirc{7/2}{-\v/2}{blue4}
            \colourcirc{11/2}{+\v*3/2}{blue5}
            \colourcirc{13/2}{+\v*5/2}{blue2}
            \colourcirc{2}{-\v}{blue3}
            \colourcirc{3/2}{+\v*3/2}{blue5}
            \colourcirc{5/2}{+\v*5/2}{blue3}
            \colourcirc{13/2}{+\v*13/2}{blue2}
            \redcirc{19/2}{+\v*13/2}
            \halfcirc{9/2}{+\v*9/2}{blue1}
            \halfcirc{15/2}{+\v*7/2}{blue1}
            \halfcirc{1/2}{+\v/2}{blue4}
            \draw(2,-5)node{$x_c^2$};
        \end{scope}
        \draw(12,+\v*3)node{$=$};
    \end{tikzpicture}
}

\newcommand{\columnexampleb}[1]{
    \begin{tikzpicture}[x=#1,y=#1]
        \blueweight{1}{2}
        \blueweight{1}{7}
        \clipcol{7}{
            \draw[blue3 line](11/2,+\v/2)\hhor\hor\dia\vver\vver\vver\vver\ddia;
            \draw[blue5 line](6,\v)\ddia\vver\vver\dia\hhor\hor;
            \draw[blue1 line](15/2,+\v*5/2)\hhor\hhor\hhor;
            \draw[blue2 line](9,+\v*4)\ddia\ddia\vver\ver;
            \draw[red line](1,-\v*2)\vver\vver\vver\ver\rhhor;
            
            \clipdouble{3}{+\v*3}{blue4 line}
            \clipdouble{5}{+\v*3}{blue4 line}
        }
        \BLOCKC{0}{}{+}{}{}
        \LCHEVC{0}{}{}{}{}
        \LCHEVC{1}{}{3}{}{}
        \LCHEVC{2}{1}{5^+}{}{}
        \LCHEVC{3}{4^+}{1}{}{}
        \LCHEVC{4}{5}{4^+}{}{}
        \LCHEVC{5}{3}{2}{}{}
        \LCHEVC{6}{}{}{2}{}
        
        \colourcirc{9/2}{+\v/2}{blue3}
        \colourcirc{13/2}{+\v*5/2}{blue1}
        \colourcirc{17/2}{+\v*9/2}{blue2}
        \colourcirc{8}{+\v*7}{blue2}
        \colourcirc{5/2}{+\v*5/2}{blue1}
        \colourcirc{9/2}{+\v*9/2}{blue5}
        \colourcirc{11/2}{+\v*11/2}{blue3}
        \redcirc{2}{-\v}
        \halfcirc{7/2}{+\v*7/2}{blue4}
        \halfcirc{15/2}{+\v*7/2}{blue4}
        \halfcirc{11/2}{+\v*3/2}{blue5}
        \draw(2,-5)node{$x_c^2$};
        
        \begin{scope}[shift={(14,0)}]
            \filldraw[red weight](0,0)--++(2,0)--++(1,-\v)--++(-2,0)--cycle;
            \blueweight{1}{7}
            \clipcol{7}{
                \draw[blue3 line](5,0)\ddia\vver\vver\vver\ddia\ddia;
                \draw[blue5 line](13/2,+\v*3/2)\hhor\hor\dia\vver\vver\ddia;
                \draw[blue1 line](15/2,+\v*5/2)\hhor\hhor\hhor;
                \draw[blue2 line](9,+\v*4)\ddia\ddia\vver\ver;
                \draw[red line](3,-\v*2)\ddia\ddia\vver\ver\rhhor\rhhor;
                
                \clipdouble{3}{+\v*3}{blue4 line}
                \clipdouble{5}{+\v*3}{blue4 line}
                \clipdouble{3}{+\v}{blue5 line}
            }
            \BLOCKC{0}{}{+}{}{}
            \LCHEVC{0}{}{}{}{}
            \LCHEVC{1}{}{3}{}{}
            \LCHEVC{2}{1}{5^+}{}{}
            \LCHEVC{3}{4^+}{1}{}{}
            \LCHEVC{4}{5}{4^+}{}{}
            \LCHEVC{5}{3}{2}{}{}
            \LCHEVC{6}{}{}{2}{}
            
            \colourcirc{9/2}{+\v/2}{blue3}
            \colourcirc{13/2}{+\v*5/2}{blue1}
            \colourcirc{17/2}{+\v*9/2}{blue2}
            \colourcirc{8}{+\v*7}{blue2}
            \colourcirc{5/2}{+\v*5/2}{blue1}
            \colourcirc{9/2}{+\v*9/2}{blue5}
            \colourcirc{11/2}{+\v*11/2}{blue3}
            \redcirc{2}{-\v}
            \halfcirc{7/2}{+\v*7/2}{blue4}
            \halfcirc{15/2}{+\v*7/2}{blue4}
            \halfcirc{11/2}{+\v*3/2}{blue5}
            \draw(2,-5)node{$-x_c^2$};
        \end{scope}
        \draw(12,+\v*3)node{$+$};
        \draw(26,+\v*3)node{$=$};
        \draw(28,+\v*3)node{$0$};
        \draw(-4,0)node{};
    \end{tikzpicture}
}

\newcommand{\prismexample}[1]{
    \begin{tikzpicture}[x=#1,y=#1]
        \blueweight{1}{1}
        \blueweight{1}{3}
        \blueweight{2}{1}
        \blueweight{3}{1}
        \blueweight{3}{2}
        \blueweight{1}{14}
        \blueweight{2}{14}
        \blueweight{3}{11}
        \blueweight{3}{14}
        
        \draw[blue1 line](53/2,+\v*13/2)\hhor--++(-6,+\v*6)\dia\vver;
        \draw[blue1 line](51/2,+\v*11/2)\hhor--++(-6,+\v*6)\ddia\dia\vver;
        \draw[blue1 line](49/2,+\v*9/2)\hhor\ddia\ddia\hhor\ddia\ddia\ddia\dia\vver\ddia\ddia\vver;
        \draw[blue3 line](8,-\v*7)--++(-7,+\v*7)\vver\dia;
        \draw[blue1 line](10,-\v*7)--++(-7,+\v*7)\vver\ddia\vver\vver\dia;
        \draw[blue2 line](12,-\v*7)--++(-7,+\v*7)\vver\vver\dia\ddia\hhor;
        \draw[red line](1/2,-\v/2)--++(6,0)--++(7,+\v*7)\vver\vver\rhhor\vver\rhhor\rhhor\vver;
        \draw[red line](3/2,-\v*3/2)--++(6,0)--++(7,+\v*7)\vver\vver\rhhor\vver\rhhor\rhhor\vver;
        \draw[red line](7/2,-\v*7/2)--++(6,0)--++(7,+\v*7)\vver\vver\vver\vver\rhhor\rhhor\vver;
        \draw[red line](13/2,-\v*13/2)--++(6,0)--++(14,+\v*14);
        
        \draw[gray](0,0)--++(7,-\v*7)--++(6,0)--++(14,+\v*14)--++(-7,+\v*7)--++(-6,0)--cycle;
        \draw[gray](0,0)--++(6,0)--++(7,-\v*7);
        \draw[gray](7,+\v*7)--++(6,0)--++(7,-\v*7);
        \draw[gray](6,0)--++(14,+\v*14);
        
        \colourcirc{15}{+\v*14}{blue1}
        \colourcirc{17}{+\v*14}{blue1}
        \colourcirc{19}{+\v*14}{blue1}
        \colourcirc{53/2}{+\v*13/2}{blue1}
        \colourcirc{51/2}{+\v*11/2}{blue1}
        \colourcirc{49/2}{+\v*9/2}{blue1}
        \colourcirc{9/2}{+\v*9/2}{blue1}
        \colourcirc{3/2}{+\v*3/2}{blue3}
        \colourcirc{7/2}{+\v*7/2}{blue2}
        \colourcirc{8}{-\v*7}{blue3}
        \colourcirc{10}{-\v*7}{blue1}
        \colourcirc{12}{-\v*7}{blue2}
        \colourcirc{1/2}{-\v/2}{red}
        \colourcirc{3/2}{-\v*3/2}{red}
        \colourcirc{7/2}{-\v*7/2}{red}
        \colourcirc{13/2}{-\v*13/2}{red}
        \colourcirc{53/2}{+\v*15/2}{red}
        \colourcirc{51/2}{+\v*17/2}{red}
        \colourcirc{49/2}{+\v*19/2}{red}
        \colourcirc{47/2}{+\v*21/2}{red}
        
        \draw(15,+\v*14)node[above]{\smfont$1$};
        \draw(17,+\v*14)node[above]{\smfont$1$};
        \draw(19,+\v*14)node[above]{\smfont$1$};
        \draw(53/2,+\v*13/2)node[right]{\smfont$1$};
        \draw(51/2,+\v*11/2)node[right]{\smfont$1$};
        \draw(49/2,+\v*9/2)node[right]{\smfont$1$};
        \draw(9/2,+\v*9/2)node[left]{\smfont$1$};
        \draw(3/2,+\v*3/2)node[left]{\smfont$3$};
        \draw(7/2,+\v*7/2)node[left]{\smfont$2$};
        \draw(8,-\v*7)node[below]{\smfont$3$};
        \draw(10,-\v*7)node[below]{\smfont$1$};
        \draw(12,-\v*7)node[below]{\smfont$2$};
        \draw(1/2,-\v/2)node[left]{\smfont$+$};
        \draw(3/2,-\v*3/2)node[left]{\smfont$+$};
        \draw(7/2,-\v*7/2)node[left]{\smfont$+$};
        \draw(13/2,-\v*13/2)node[left]{\smfont$+$};
        \draw(53/2,+\v*15/2)node[right]{\smfont$+$};
        \draw(51/2,+\v*17/2)node[right]{\smfont$+$};
        \draw(49/2,+\v*19/2)node[right]{\smfont$+$};
        \draw(47/2,+\v*21/2)node[right]{\smfont$+$};
        
        \draw[blue6](15/2,-\v*11/2)node{\BB{A}};
        \draw[blue6](17,\v)node{\BB{B}};
        \draw[blue6](19,+\v*3)node{\BB{C}};
        \draw[blue6](9,+\v*6)node{\BB{D}};
        \draw[blue6](11,+\v*8)node{\BB{E}};
        
        \perparrow{0}{0}{7}{\vx*7}{-2}{left}{$\al^*$}
        \perparrow{7}{-\vx*7}{0}{0}{-2}{left}{$\la^+$}
        \perparrow{7}{-\vx*7}{13}{-\vx*7}{2}{below}{$\si^{-1}$}
        \perparrow{14}{\vx*14}{20}{\vx*14}{-2}{above}{$\bacirc^3$}
        \perparrow{27}{\vx*7}{20}{\vx*14}{2}{right}{$\rcirc^4\wcirc^3$}
        \perparrow{20}{0}{27}{\vx*7}{2}{right}{$\wcirc^4\bacirc^3$}
        
        \begin{scope}[shift={(35,0)}]
            \blueweight{8}{-6}
            \blueweight{8}{-5}
            \blueweight{8}{-3}
            \blueweight{9}{-6}
            \blueweight{9}{-5}
            \blueweight{10}{-6}
            \blueweight{10}{-5}
            \blueweight{10}{-4}
            \blueweight{10}{-3}
            
            \draw[blue1 line](53/2,+\v*13/2)--++(-15/2,+\v*15/2);
            \draw[blue1 line](51/2,+\v*11/2)--++(-17/2,+\v*17/2);
            \draw[blue1 line](49/2,+\v*9/2)--++(-19/2,+\v*19/2);
            \draw[blue3 line](8,-\v*7)\vver\vver\dia\hhor--++(-6,+\v*6);
            \draw[blue1 line](10,-\v*7)\vver\vver\ddia\vver\dia\hhor\ddia\rel\lel--++(-4,+\v*4);
            \draw[blue2 line](12,-\v*7)\vver\vver\vver\vver\dia\ddia\ddia\hhor\hhor\hhor--++(-4,+\v*4);
            \draw[red line](1/2,-\v/2)\vver\rhhor\vver\rhhor\rhhor\vver\vver--++(7,+\v*7)--++(6,0);
            \draw[red line](3/2,-\v*3/2)\vver\rhhor\vver\rhhor\rhhor\vver\vver--++(7,+\v*7)--++(6,0);
            \draw[red line](7/2,-\v*7/2)\vver\rhhor\vver\rhhor\vver\vver\vver--++(7,+\v*7)--++(6,0);
            \draw[red line](13/2,-\v*13/2)--++(14,+\v*14)--++(6,0);
            
            \draw[gray](0,0)--++(7,-\v*7)--++(6,0)--++(14,+\v*14)--++(-7,+\v*7)--++(-6,0)--cycle;
            \draw[gray](7,+\v*7)--++(7,-\v*7)--++(6,0);
            \draw[gray](14,+\v*14)--++(7,-\v*7)--++(6,0);
            \draw[gray](7,-\v*7)--++(14,+\v*14);
            
            \colourcirc{15}{+\v*14}{blue1}
            \colourcirc{17}{+\v*14}{blue1}
            \colourcirc{19}{+\v*14}{blue1}
            \colourcirc{53/2}{+\v*13/2}{blue1}
            \colourcirc{51/2}{+\v*11/2}{blue1}
            \colourcirc{49/2}{+\v*9/2}{blue1}
            \colourcirc{9/2}{+\v*9/2}{blue1}
            \colourcirc{3/2}{+\v*3/2}{blue3}
            \colourcirc{7/2}{+\v*7/2}{blue2}
            \colourcirc{8}{-\v*7}{blue3}
            \colourcirc{10}{-\v*7}{blue1}
            \colourcirc{12}{-\v*7}{blue2}
            \colourcirc{1/2}{-\v/2}{red}
            \colourcirc{3/2}{-\v*3/2}{red}
            \colourcirc{7/2}{-\v*7/2}{red}
            \colourcirc{13/2}{-\v*13/2}{red}
            \colourcirc{53/2}{+\v*15/2}{red}
            \colourcirc{51/2}{+\v*17/2}{red}
            \colourcirc{49/2}{+\v*19/2}{red}
            \colourcirc{47/2}{+\v*21/2}{red}
            
            \draw(15,+\v*14)node[above]{\smfont$1$};
            \draw(17,+\v*14)node[above]{\smfont$1$};
            \draw(19,+\v*14)node[above]{\smfont$1$};
            \draw(53/2,+\v*13/2)node[right]{\smfont$1$};
            \draw(51/2,+\v*11/2)node[right]{\smfont$1$};
            \draw(49/2,+\v*9/2)node[right]{\smfont$1$};
            \draw(9/2,+\v*9/2)node[left]{\smfont$1$};
            \draw(3/2,+\v*3/2)node[left]{\smfont$3$};
            \draw(7/2,+\v*7/2)node[left]{\smfont$2$};
            \draw(8,-\v*7)node[below]{\smfont$3$};
            \draw(10,-\v*7)node[below]{\smfont$1$};
            \draw(12,-\v*7)node[below]{\smfont$2$};
            \draw(1/2,-\v/2)node[left]{\smfont$+$};
            \draw(3/2,-\v*3/2)node[left]{\smfont$+$};
            \draw(7/2,-\v*7/2)node[left]{\smfont$+$};
            \draw(13/2,-\v*13/2)node[left]{\smfont$+$};
            \draw(53/2,+\v*15/2)node[right]{\smfont$+$};
            \draw(51/2,+\v*17/2)node[right]{\smfont$+$};
            \draw(49/2,+\v*19/2)node[right]{\smfont$+$};
            \draw(47/2,+\v*21/2)node[right]{\smfont$+$};
            
            \draw[blue6](15/2,+\v*9/2)node{\BB{K}};
            \draw[blue6](19/2,+\v*13/2)node{\BB{I}};
            \draw[blue6](17,-\v)node{\BB{J}};
            \draw[blue6](19,+\v)node{\BB{G}};
            \draw[blue6](19,+\v*13)node{\BB{H}};
        \end{scope}
        
        \draw(31,+\v*7/2)node{$=$};
    \end{tikzpicture}
}

\newcommand{\LHSeq}[1]{
    \begin{tikzpicture}[x=#1,y=#1]
        \filldraw[gray fill](0,0)--++(3,-\v*3)--++(4,0)--++(6,+\v*6)--++(-3,+\v*3)--++(-4,0)--cycle;
        \draw[gray](0,0)--++(4,0)--++(3,-\v*3);
        \draw[gray](3,+\v*3)--++(4,0)--++(3,-\v*3);
        \draw[gray](4,0)--++(6,+\v*6);
        \perparrow{0}{0}{3}{\vx*3}{-1}{left}{$\al^*$}
        \perparrow{3}{-\vx*3}{0}{0}{-1}{left}{$\la^+$}
        \perparrow{3}{-\vx*3}{7}{-\vx*3}{1}{below}{\smfont$\si^{-1}$}
        \perparrow{6}{\vx*6}{10}{\vx*6}{-1}{above}{$\bacirc^n$}
        \perparrow{13}{\vx*3}{10}{\vx*6}{1}{right}{$\rcirc^k\wcirc^n$}
        \perparrow{10}{0}{13}{\vx*3}{1}{right}{$\wcirc^k\bacirc^n$}
        
        \draw(7/2,-\v*3/2)node{$1$};
        \draw(7,0)node{$1$};
        \draw(10,+\v*3)node{$1$};
        
        \draw(79/4,+\v*3/2)node{$=$};
        
        \begin{scope}[shift={(25,-\v*3/2)}]
            \filldraw[gray fill](0,0)--++(4,0)--++(6,+\v*6)--++(-4,0)--cycle;
            \draw[gray](3,+\v*3)--++(4,0);
            \perparrow{0}{0}{3}{\vx*3}{-1}{left}{$\al^*$}
            \perparrow{0}{0}{4}{0}{1}{below}{\smfont$\si^{-1}$}
            \perparrow{6}{\vx*6}{10}{\vx*6}{-1}{above}{$\bacirc^n$}
            \perparrow{10}{\vx*6}{7}{\vx*3}{-1}{right}{$\la^-$}
            
            \draw(13/2,+\v*9/2)node{$s_\la$};
            \draw(7/2,+\v*3/2)node{$\A_\al^\si$};
        \end{scope}
    
        \draw(38,+\v*3/2)node{$=$};
        \draw(40,+\v*3/2)node[right]{$\A_\al^\si(x)s_\la(x)$};
    \end{tikzpicture}
}

\newcommand{\RHSeq}[1]{
    \begin{tikzpicture}[x=#1,y=#1]
        \filldraw[gray fill](0,0)--++(3,-\v*3)--++(4,0)--++(6,+\v*6)--++(-3,+\v*3)--++(-4,0)--cycle;
        \draw[gray](3,+\v*3)--++(3,-\v*3)--++(4,0);
        \draw[gray](6,+\v*6)--++(3,-\v*3)--++(4,0);
        \draw[gray](3,-\v*3)--++(6,+\v*6);
        \perparrow{0}{0}{3}{\vx*3}{-1}{left}{$\al^*$}
        \perparrow{3}{-\vx*3}{0}{0}{-1}{left}{$\la^+$}
        \perparrow{3}{-\vx*3}{7}{-\vx*3}{1}{below}{\smfont$\si^{-1}$}
        \perparrow{6}{\vx*6}{10}{\vx*6}{-1}{above}{$\bacirc^n$}
        \perparrow{13}{\vx*3}{10}{\vx*6}{1}{right}{$\rcirc^k\wcirc^n$}
        \perparrow{10}{0}{13}{\vx*3}{1}{right}{$\wcirc^k\bacirc^n$}
        
        \draw(6,+\v*3)node{$1$};
        \draw(19/2,+\v*9/2)node{$1$};
        \draw(19/2,+\v*3/2)node{$1$};
        
        \draw(79/4,+\v*3/2)node{$=$};
        \begin{scope}[shift={(24,+\v*3/2)}]
            \draw(0,0)node{$\displaystyle\sum_\be$};
        \begin{scope}[shift={(3,0)}]
            \filldraw[gray fill](0,0)--++(3,-\v*3)--++(4,0)--++(3,+\v*3)--++(-4,0)--++(-3,+\v*3)--cycle;
            \draw[gray](3,-\v*3)--++(3,+\v*3);
            \perparrow{0}{0}{3}{\vx*3}{-1}{left}{$\al^*$}
            \perparrow{3}{-\vx*3}{0}{0}{-1}{left}{$\la^+$}
            \perparrow{3}{-\vx*3}{7}{-\vx*3}{1}{below}{\smfont$\si^{-1}$}
            \perparrow{6}{0}{3}{\vx*3}{1}{right}{$\rcirc^k\wcirc^n$}
            
            \draw(13/2,-\v*3/2)node{$\A_\be^\si$};
        \end{scope}
        \end{scope}
    
        \begin{scope}[shift={(45,+\v*3/2)}]
            \draw(0,0)node{$\displaystyle\sum_\be$};
            \begin{scope}[shift={(3,0)}]
                \filldraw[gray fill](0,0)--++(3,-\v*3)--++(3,+\v*3)--++(-3,+\v*3)--cycle;
                \perparrow{0}{0}{3}{\vx*3}{-1}{left}{$\al^*$}
                \perparrow{3}{-\vx*3}{0}{0}{-1}{left}{$\la^+$}
                \perparrow{6}{0}{3}{\vx*3}{1}{right}{$\rcirc^k\wcirc^n$}
                \perparrow{3}{-\vx*3}{6}{0}{1}{right}{$\be^*$}
                
                \draw(3,0)node{$a_{\al\la}^\be(\si)$};
                \draw(7,0)node[right]{$\A_\be^\si(x)$};
            \end{scope}
        \end{scope}
        
        \draw(41,+\v*3/2)node{$=$};
    \end{tikzpicture}
}

\newcommand{\triples}[1]{
    \begin{tikzpicture}[x=#1,y=#1]
        \draw (0,0)--++(1,0)--++(0,2)--++(-1,0)--cycle;
        \draw (0,1)--++(1,0);
        \draw (3,1)--++(1,0)--++(0,1)--++(-1,0)--cycle;
        
        \draw (0,4) node{\textbf{Type A:}};
        \draw (0.5,1.5) node{$a$};
        \draw (3.5,1.5) node{$b$};
        \draw (0.5,0.5) node{$c$};
        \draw (2,1.5) node{$\cdots$};
        \draw (0.5,2.5) node{$\ell$};
        \draw (3.5,2.5) node{$r$};
        
        \begin{scope}[shift={(10,0)}]
            \draw (0,4) node{\textbf{Type B:}};
            \draw (3,0)--++(1,0)--++(0,2)--++(-1,0)--cycle;
            \draw (3,1)--++(1,0);
            \draw (0,0)--++(1,0)--++(0,1)--++(-1,0)--cycle;
            
            \draw (3.5,1.5) node{$a$};
            \draw (0.5,0.5) node{$b$};
            \draw (3.5,0.5) node{$c$};
            \draw (2,0.5) node{$\cdots$};
            \draw (0.5,2.5) node{$\ell$};
            \draw (3.5,2.5) node{$r$};
        \end{scope}
    \end{tikzpicture}
}

\begin{document}
    \maketitle
    
    \begin{abstract}
        We present the first positive combinatorial rule for expanding the product of a permuted-basement Demazure atom and a Schur polynomial. Special cases of permuted-basement Demazure atoms include Demazure atoms and characters. These cases have known tableau formulas for their expansions when multiplied by a Schur polynomial, due to Haglund, Luoto, Mason and van Willigenburg. We find a vertex model formula, giving a new rule even in these special cases, extending a technique introduced by Zinn-Justin for calculating Littlewood--Richardson coefficients.
        
        We derive a coloured vertex model for permuted-basement Demazure atoms, inspired by Borodin and Wheeler's model for non-symmetric Macdonald polynomials. We make this model compatible with an uncoloured vertex model for Schur polynomials, putting them in a single framework. Unlike previous work on structure coefficients via vertex models, a remarkable feature of our construction is that it relies on a Yang--Baxter equation that only holds for certain boundary conditions. However, this restricted Yang--Baxter equation is sufficient to show our result.
    \end{abstract}
    
    \section{Introduction}
    The \itc{Schur polynomials} $s_\la(x)$, indexed by partitions $\la$, are an important $\Z$-basis for the ring of symmetric polynomials. The product of two Schur polynomials expands into a linear combination with positive integer structure coefficients, called \itc{Littlewood--Richardson coefficients}. These integers have many important applications, particularly in the representation theory of the general linear group and symmetric group, Schubert calculus and algebraic combinatorics. There are many combinatorial rules for computing these integers, the original being the eponymous Littlewood--Richardson rule \cite{LR34}.
    
    Knutson, Tao and Woodward developed an alternative formula, calculating Littlewood--Richardson coefficients as the number of certain tilings called \itc{puzzles} \cite{KT03,KTW03}. Puzzles were brought into the context of integrable systems by Zinn-Justin who reproved this rule, envisioning a particle model which places puzzles and Schur polynomials in one unified picture \cite{Z09}. We build on this technique, giving the first combinatorial formula for the product of a Schur polynomial with a \itc{permuted-basement Demazure atom} $\A_\al^\si(x)$, indexed by weak compositions $\al$ with permutation parameter $\si$.
    
    We derive a manifestly positive, combinatorial rule for the structure coefficients $a_{\al\la}^\be(\si)$ in the expansion
    \begin{align*}
        \A_\al^\si(x)s_\la(x)=\sum_\be a_{\al\la}^\be(\si)\A_\be^\si(x),
    \end{align*}
    where the summation is over weak compositions $\be$. Alexandersson and Sawhney \cite{AS19} show this expansion is positive and express interest in a combinatorial rule, which we now have.
    
    Several bases of the ring of multivariate polynomials are known to expand positively when multiplied by a Schur polynomial \cite{S20}. Two well-known bases that exhibit this phenomenon are \itc{Demazure atoms} and \itc{Demazure characters}. Demazure characters, also called key polynomials, were introduced by Demazure \cite{D74} and studied combinatorially by Lascoux and Sch{\"u}tzenberger \cite{LS90}. They are characters of representations of a Borel subgroup, known as Demazure modules, and they generalize Schur polynomials. Demazure characters are also a special case of Kohnert polynomials, introduced by Assaf and Searles \cite{AS22}. Haglund, Luoto, Mason and van Willigenburg \cite{HLMW11} give positive rules for the products of a Schur polynomial with a Demazure atom or character, expressed in terms of certain tableaux known as skyline fillings; see also Assaf \cite{A23}. There are many interesting open problems regarding expansions of Demazure polynomials, such as whether the product of two characters expands positively into atoms \cite{P16}.
    
    The polynomials $\A_\al^\si(x)$ are a specialization of \itc{permuted-basement non-symmetric Macdonald polynomials} $E_\al^\si(x;q,t)$, introduced by Ferreira \cite{F11} and expanded on by Alexandersson \cite{A19}; see also Guo and Ram \cite{GR22}. We follow the conventions of Alexandersson, setting $\A_\al^\si(x):=E_\al^\si(x;0,0)$; note that some conventions in the literature reverse the order of $\al$. The author's work in an FPSAC extended abstract for this paper \cite{M24} originally considered separate models for Demazure atoms and characters. Now, both are captured by one model depending on the parameter $\si$ where $\si=(1,\ldots,n)$ is the atom case and $\si=(n,\ldots,1)$ is the character case.
    
    \itc{Vertex models} are a tool in statistical mechanics used to study particle systems. A notable example is the six-vertex or square-ice model; see \cite{B82} for detailed discussion. Each configuration of vertices in a graph is assigned a weight contributing to its \itc{partition function}. Vertex models are often in a two-dimensional lattice where vertices may be represented with tiles, which is the convention we adopt in this paper.
    
    Vertex model approaches have shown increasing promise in Schubert calculus and algebraic combinatorics with many recent papers demonstrating their applicability, e.g. \cite{ABW23,BW16,BWZ15,BFHTW23,CGKM22,CFYZZ23,GW20,KZ20,KZ21,KZ23,MS13,WZ16}. Often, an important polynomial can be realized as the partition function of a vertex model, making analysis amenable to a standard set of techniques from integrable systems. Schur polynomials, Demazure atoms and non-symmetric Macdonald polynomials all have vertex model formulations. Our model for $\A_\al^\si(x)$ is inspired by Borodin and Wheeler's model for $E_\al^\si(x;q,t)$ \cite{BW22}, but we make significant modifications. Our model now bears more resemblance to the model for Demazure atoms given by Brubaker, Buciumas, Bump and Gustafsson \cite{BBBG21}; see \Cref{rem.BBBG}.
    
    Our proof is completely combinatorial. We establish equivalence by gluing vertex models together in two different ways, similar to \cite{WZ19,Z09}. Unlike these previous results, which rely on a \itc{Yang--Baxter equation} that holds for all possible boundary conditions, our equation in \Cref{lem.column} only holds for certain boundaries. Happily, these restrictions are satisfied when needed. Results in other types involve going beyond the Yang--Baxter equation in a more fundamental way. Buciumas and Scrimshaw \cite{BS22} consider vertex models for Demazure atoms and characters in type $B$ and $C$ where one of the required Yang--Baxter equations does not hold, but a different relation allows them to characterize the partition functions. Zhong \cite{Z22} finds a similar phenomenon in stochastic type $C$ vertex models.
    
    It is tempting to define a notion of ``double'' Demazure atoms in two sets of variables, analogous to the double Schur and Grothendieck polynomials in \cite{WZ19,Z09}. These polynomials correspond to equivariant cohomology classes. However, we do not find a rule when adding this extra set of variables. In particular, \Cref{lem.column} does not hold when introducing ``equivariant'' tiles.
    
    Our results suggest further applications, such as extensions to the Grothendieck models in \cite{WZ19}. \itc{Grothendieck polynomials}, \itc{Lascoux atoms} and \itc{Lascoux polynomials} are the $K$-theoretic analogues of Schur polynomials, Demazure atoms and Demazure characters respectively. Orelowitz and Yu \cite{OY23} have recently given a tableau formula for the expansion of a Lascoux polynomial and a stable Grothendieck. Buciumas, Scrimshaw and Weber \cite{BSK20} also have Vertex models for Lascoux polynomials.
    
    A benefit of this approach is that we may develop vertex models independently and then fit them into this framework, allowing one to test rules assuming a Yang--Baxter such as \Cref{lem.column} holds. As an offshoot, our notion of \itc{$\si$-extendable} in \Cref{sec.skyline} yields a poset on weak compositions for each $\si$ which determines a branching formula for atoms. This can be the subject of future work.
    
    In \Cref{sec.models}, we introduce a vertex model for Schur polynomials and our model for atoms. \Cref{sec.skyline} gives the skyline formulation for atoms and introduces the $\si$-extendable relation, yielding a branching formula in \Cref{cor.expansion}. In \Cref{sec.bijection}, we show that our vertex models satisfy the same recurrence, making a bijection clear. In \Cref{sec.theorem}, we state our expansion rule, \Cref{thm.diamonds}, and then prove it in \Cref{sec.proof}. We reserve proof of our key technical result, the Column \Cref{lem.column}, for \Cref{sec.colproof}.
    
    \section{Vertex models for atoms and Schurs}\label{sec.models}
    We give vertex model formulations for permuted-basement Demazure atoms and Schur polynomials, using the convention of depicting ``vertices'' in the models with tiles, aligning with \cite{WZ19,Z09}.
    
    Throughout, $x=(x_1,\ldots,x_n)$ is a list of $n$ variables and we set $[n]:=\{1,\ldots,n\}$. A \itc{weak composition} $\al=(\al_1,\ldots,\al_n)$ is a sequence of non-negative integers. Each integer $\al_i$ is the \itc{part} of $\al$ at index $i$ and the \itc{length} of $\al$ is its number of parts, denoted $\len(\al)$. The largest part in $\al$ is denoted $\max(\al)$. A \itc{partition} $\la=(\la_1,\ldots,\la_n)$ is a weak composition with parts sorted in descending order. A \itc{permutation} $\si=(\si(1),\ldots,\si(n))$ is an ordering of the integers in $[n]$. Throughout, $\al$ and $\be$ are weak compositions, $\la$ is a partition, $\si$ is a permutation and all have length $n$.
    
    We label certain boundaries in our models with strings that re-encode $\al$ or $\la$. Let $\mu$ be the partition with the same parts as $\al$. Enclose the Young diagram of $\mu$ (in English convention) within the top left corner of a rectangle with a North-East lattice path as depicted in \Cref{ex.compositions}. Label East steps $0$. Each index $i\in[n]$ labels a North step in a row of width $\al_i$; if multiple rows have the same width, their labels are sorted in descending order from bottom to top. We obtain the string $\al^*$ by reading labels along the lattice path, starting from the bottom.
    
    We re-encode a partition $\la$ similarly in two ways, with strings $\la^-$ and $\la^+$. For $\la^-$, label East steps $0$ and North steps $1$. For $\la^+$, label East steps with the symbol $+$ and North steps $0$. Read strings from the lattice path as before. The string $\la^+$ appears in later sections.
    
    \begin{example}\label{ex.compositions}
        Let $\al=(0,3,0,1,3)$ and $\la=(4,4,2,1,0)$. To produce $\al^*$, we enclose the partition $\mu=(3,3,1,0,0)$. In this case, $\al_1=\al_3=0$, so rows of width $0$ are labelled $3$ and $1$ from bottom to top. Similarly, $\al_2=\al_5=3$, so rows of width $3$ are labelled $5$ and $2$ from bottom to top.
        \begin{center}
            \composition{\ascale}
        \end{center}
        Reading the labels starting from the bottom produces the desired strings:
        \begin{align*}
            \al^* &= 31040052 \\
            \la^- &= 101010011 \\
            \la^+ &= \text{0+0+0++00}
        \end{align*}
    \end{example}
    
    We construct our models by tiling a lattice with rhombic tiles, depicted below, where $i,j,f,g$ are in $[n]$, $i\leq j$ and $\si(f)<\si(g)$.
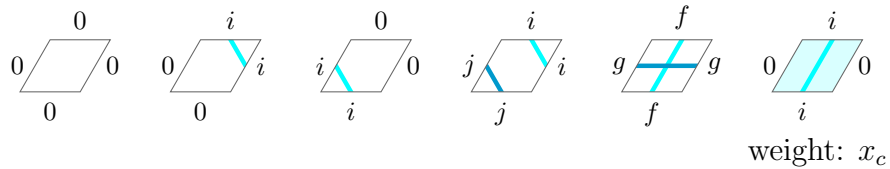
    \begin{center}
        \demazuretiles{\ascale}
    \end{center}
    The final tile has weight $x_c$, where $c$ is the column number that the tile occurs in, and all other tiles have weight $1$. Columns are numbered $1$ to $n$ from left to right. Tiles of weight $1$ are \itc{trivial} and the final tile is \itc{nontrivial}. We define \itc{permuted-basement Demazure atoms (atoms)} $\A_\al^\si(x)$ and \itc{Schur polynomials} $s_\la(x)$ in terms of vertex models with the tiles. We place the tiles within a rhombic lattice with fixed boundary labels. Labels between tiles and boundaries must match to be a legal tiling. We schematically outline both models below:
    \begin{center}
        \asmodels{\bscale}
    \end{center}
    
    We call the model for $\A_\al^\si(x)$ the \itc{atom model} and the model for $s_\la(x)$ the \itc{Schur model}. Label the left boundary of the atom model with $\al^*$ and the bottom with $\si^{-1}$ (this choice makes our polynomials align with the definition in terms of skyline fillings in \Cref{sec.skyline}). Label the left boundary of the Schur model with $\la^-$ and the bottom with all $1$'s, denoted $\bacirc^n$. Label unspecified boundaries with all $0$'s. The weight of a tiling is the product of its tile weights and the sum of all possible tilings is the \itc{partition function} of the model.
    
    \begin{example}\label{ex.models}
        Let $\al=(1,0,2,2)$ and $\si=4123$, so that $\si^{-1}=2341$ and $\al^*=201043$. There are four tilings of the atom model, showing $\A_\al^\si(x_1,x_2,x_3,x_4)=x_1x_2^2x_3^2+x_1x_2^2x_3x_4+x_1x_2x_3^2x_4+x_2^2x_3^2x_4$.
        \begin{center}
            \atomexample{\bscale}
        \end{center}
        Next, let $\la=(2,2,1)$, so that $\la^-=01011$. There are three tilings of the Schur model, showing $s_\la(x_1,x_2,x_3)=x_1^2x_2^2x_3+x_1^2x_2x_3^2+x_1x_2^2x_3^2$.
        \begin{center}
            \schurexample{\bscale}
        \end{center}
    \end{example}
    
    \begin{remark}\label{rem.BBBG}
        The Schur model is the same as the one considered by Zinn-Justin in \cite{Z09} and we can think of the atom model as a ``coloured'' version of the Schur model. Brubaker, Buciumas, Bump and Gustafsson \cite{BBBG21} have a vertex model for Demazure atoms, which is the case where $\si=\id$. Our tiles are essentially the same as their vertices in this case, with a permutation applied to the colours. Their model also has different weights, picking up an extra factor $x^\rho$.
    \end{remark}
    \section{Skyline fillings and $\si$-extendable decompositions}\label{sec.skyline}
    We give an alternative definition of atoms in terms of the semi-skyline augmented fillings introduced by Mason \cite{M08,M09}, only with a permuted basement. We use this definition to give a branching formula for atoms. A \itc{semi-skyline augmented filling (SSAF)} of \itc{shape} $\al$ with \itc{basement} $\si$ is a function 
    \begin{align*}
        F:\{(i,j)\mid0\leq j\leq\al_i,i\in[n]\}\to\Z_+,
    \end{align*}
    which must satisfy certain conditions. We refer to the pairs $(i,j)$ in the domain of $F$ as \itc{cells} and write $F((i,j)):=F(i,j)$ as a shorthand. We say a cell is in $F$ if it is in the domain of $F$. The \itc{entry} of a cell $c$ is the value $F(c)$ and a cell with entry $e$ is an \itc{$e$-cell} of $F$. The \itc{basement} of a filling refers to cells $(i,0)$ for $i\in[n]$, which are always assigned the entry $\si(i)$. SSAFs must have no \itc{descents}, meaning $F(i,j+1)\leq F(i,j)$ for $j\geq0$.
    
    Lastly, SSAFs must satisfy certain conditions for ``triples,'' which come in two flavours. Let $1\leq\ell<r\leq n$; if $\al_\ell\geq\al_r$, a \itc{type A triple} within columns $\ell$ and $r$ is of the form $\{(\ell,i+1),(r,i+1),(\ell,i)\}$ for $i\geq0$. If $\al_\ell<\al_r$, a \itc{type B triple} within columns $\ell$ and $r$ is of the form $\{(r,i+1),(\ell,i),(r,i)\}$ for $i\geq0$. We illustrate below, labelling cells with their entries $a,b,c$:
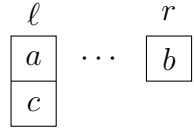
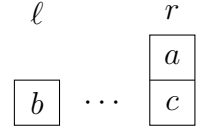
    \begin{center}
        \triples{2*\bscale}
    \end{center}
    In either case, we say a triple is \itc{inversion} if $b\notin[a,c]$ and \itc{coinversion} if $b\in[a,c]$. All triples within an SSAF are inversion. Denote the set of all SSAFs of shape $\al$ with basement $\si$ by $\SSAF_\si(\al)$. The polynomials $\A_\al^\si(x)$ are given by a summation over fillings in $\SSAF_\si(\al)$:
    \begin{align*}
        \A_\al^\si(x) = \sum_{F\in\SSAF_\si(\al)}x^F,
    \end{align*}
    where $x^F$ is the product of $x_{F(c)}$ over all cells $c$ in $F$.
    \begin{example}
        If $\al=(1,0,2,1)$ and $\si=(2,1,4,3)$, there are five fillings in $\SSAF_\si(\al)$:
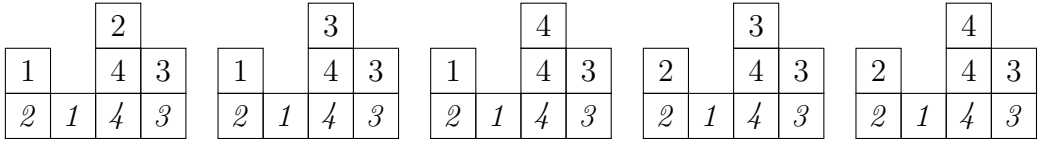
        \begin{center}
            \diagram{2*\bscale}{{{2,1},{1},{4,4,2},{3,3}}}
            \diagram{2*\bscale}{{{2,1},{1},{4,4,3},{3,3}}}
            \diagram{2*\bscale}{{{2,1},{1},{4,4,4},{3,3}}}
            \diagram{2*\bscale}{{{2,2},{1},{4,4,3},{3,3}}}
            \diagram{2*\bscale}{{{2,2},{1},{4,4,4},{3,3}}}
        \end{center}
        This shows that $\A_\al^\si(x)=x_1x_2x_3x_4+x_1x_3^2x_4+x_1x_3x_4^2+x_2x_3^2x_4+x_2x_3x_4^2$.
    \end{example}
    \begin{definition}\label{def.extendable}
        Given weak compositions $\al$ and $\be$ of length $n$ and $\si\in S_n$, we say that $\be$ is \itc{$\si$-extendable} to $\al$ if $\al_i\geq\be_i$ for $i\in[n]$ and the following conditions hold for each pair $(\ell,r)$ satisfying $1\leq\ell<r\leq n$:
        \begin{enumerate}
            \item If $\al_\ell\geq\al_r$ and $\be_\ell\geq\be_r$, then $\al_r\leq\be_\ell$.
            \item If $\al_\ell<\al_r$ and $\be_\ell<\be_r$, then $\al_\ell<\be_r$.
            \item If $\al_\ell\geq \al_r$ and $\be_\ell<\be_r$, then $\al_r=\be_r$ and $\si(\ell)<\si(r)$.
            \item If $\al_\ell<\al_r$ and $\be_\ell\geq\be_r$, then $\al_\ell=\be_\ell$ and $\si(\ell)>\si(r)$.
        \end{enumerate}
    \end{definition}
    \pagebreak
    \begin{lemma}\label{lem.below}
        Let $F\in\SSAF_\si(\al)$ for a weak composition $\al$ of length $n$ and $\si\in S_n$ and let $1\leq\ell<r\leq n$.
        \begin{enumerate}
            \item If $\al_\ell\geq\al_r$ and $0\leq a\leq\al_r$, then:
            \begin{enumerate}
                \item If $F(\ell,a)<F(r,a)$, then $F(\ell,i)<F(r,i+1)$ for all $0\leq i\leq a-1$.
                \item If $F(\ell,a)>F(r,a)$, then $F(\ell,i)>F(r,i)$ for all $a\leq i\leq\al_r$.
            \end{enumerate}
            \item If $\al_\ell<\al_r$ and $0\leq a\leq\al_\ell$, then:
            \begin{enumerate}
                \item If $F(\ell,a)<F(r,a)$, then $F(\ell,i)<F(r,i+1)$ for all $a\leq i\leq\al_\ell$.
                \item If $F(\ell,a)>F(r,a)$, then $F(\ell,i)>F(r,i)$ for all $0\leq i\leq a$.
            \end{enumerate}
        \end{enumerate}
    \end{lemma}
    \begin{proof}
        \begin{enumerate}
            \item Assume $\al_\ell\geq\al_r$ and $0\leq a\leq\al_r$.
            \begin{enumerate}
                \item Suppose $F(\ell,a)<F(r,a)$. If $a=0$, the inequality is vacuously satisfied. Otherwise, if $F(\ell,a-1)\geq F(r,a)$, then $\{(\ell,a),(r,a),(\ell,a-1)\}$ is a type A coinversion triple, so $F(\ell,a-1)<F(r,a)$. Then by the descent condition, $F(\ell,a-1)<F(r,a-1)$ and by induction, we have $F(\ell,i)<F(r,i+1)$ for all $0\leq i\leq a-1$.
                \item Suppose $F(\ell,a)>F(r,a)$. If $a=\al_r$, we are done. Otherwise, if $F(\ell,a+1)\leq F(r,a+1)$, then $\{(\ell,a+1),(r,a+1),(\ell,a)\}$ is a type A coinversion triple, so $F(\ell,a+1)>F(r,a+1)$. By induction, we have $F(\ell,i)>F(r,i)$ for all $a\leq i\leq\al_r$.
            \end{enumerate}
            \item Assume $\al_\ell<\al_r$ and $0\leq a\leq\al_\ell$.
            \begin{enumerate}
                \item Suppose $F(\ell,a)<F(r,a)$. If $F(\ell,a)\geq F(r,a+1)$ then $\{(r,a+1),(\ell,a),(r,a)\}$ is a type B coinversion triple, so $F(\ell,a)<F(r,a+1)$. If $a=\al_\ell$, we are done; otherwise, $F(\ell,a+1)<F(r,a+1)$ by the descent condition and by induction, we have $F(\ell,i)<F(r,i+1)$ for all $a\leq i\leq\al_\ell$.
                \item Suppose $F(\ell,a)>F(r,a)$. If $a=0$, we are done; otherwise, $F(\ell,a-1)\geq F(\ell,a)>F(r,a)$ and thus $F(\ell,a-1)>F(r,a-1)$ since otherwise $\{(r,a),(\ell,a-1),(r,a-1)\}$ is a type B coinversion triple. By induction, we have $F(\ell,i)>F(r,i)$ for all $0\leq i\leq a$.
            \end{enumerate}
        \end{enumerate}
    \end{proof}
    
    We prove that the atoms $\A_\al^\si(x)$ decompose into a summation over compositions $\be$ that are $\si$-extendable to $\al$. Each such $\be$ is the shape of an SSAF that remains after deleting the non-basement $1$-cells from a filling in $\SSAF_\si(\al)$.
    \begin{theorem}\label{thm.decomposition}
        For a weak composition $\al$ of length $n$ and $\si\in S_n$, we have
        \begin{align*}
            \A_\al^\si(x) = \sum_\text{$\be$ $\si$-extendable to $\al$} x_1^{|\al|-|\be|}\A_\be^\si(0,x_2,\ldots,x_n).
        \end{align*}
    \end{theorem}
    \begin{proof}
        For any filling $F\in\SSAF_\si(\al)$, we produce a new filling $F'$ obtained by deleting all $1$-cells other than the basement. This yields a new filling on a weak composition $\be$ with the same basement as in the following example:
        \begin{center}
            \deleteboxes{2*\bscale}
        \end{center}
        
        We first show that every such $\be$ is $\si$-extendable to $\al$ and that $F'$ is in $\SSAF_\si(\be)$. It is clear that $\al_i\geq\be_i$ for all $i\in[n]$, so we need to check the remaining conditions of \Cref{def.extendable} for each pair of columns and verify that the relevant triples are inversion; let $1\leq\ell<r\leq n$.
        \begin{enumerate}
            \item Assume $\al_\ell\geq\al_r$ and $\be_\ell\geq\be_r$. If $\al_r>\be_\ell$, then cells $(\ell,\be_\ell+1)$ and $(r,\be_\ell+1)$ both exist in $F$ and have entry $1$, yielding a type A coinversion triple in $F$, so we have $\al_r\leq\be_\ell$. All type A triples within columns $\ell$ and $r$ in $F'$ are inversion by assumption.
            
            \item Assume $\al_\ell<\al_r$ and $\be_\ell<\be_r$. If $\al_\ell\geq\be_r$, then cells $(\ell,\be_r)$ and $(r,\be_r+1)$ both exist in $F$ and have entry $1$, yielding a type B coinversion triple in $F$, so we have $\al_\ell<\be_r$. All type B triples within columns $\ell$ and $r$ in $F'$ are inversion by assumption.
            
            \item Assume $\al_\ell\geq\al_r$ and $\be_\ell<\be_r$. If $\al_r>\be_r$, then cells $(\ell,\be_r+1)$ and $(r,\be_r+1)$ both exist in $F$ and have entry $1$, yielding a type A coinversion triple in $F$, so we have $\al_r=\be_r$. It is also true that the cell $(\ell,\be_r)$ exists in $F$ and has entry $1$, so $F(\ell,\be_r)<F(r,\be_r)$ and, by case 1a of \Cref{lem.below}, we have $F(\ell,i)<F(r,i+1)\leq F(r,i)$ for all $0\leq i\leq \be_r-1$, which implies $\si(\ell)<\si(r)$ and also that all type B triples within columns $\ell$ and $r$ in $F'$ are inversion.
            
            \item Assume $\al_\ell<\al_r$ and $\be_\ell\geq\be_r$. If $\al_\ell>\be_\ell$, then cells $(\ell,\be_\ell+1)$ and $(r,\be_\ell+2)$ both exist in $F$ and have entry $1$, yielding a type B coinversion triple in $F$, so we have $\al_\ell=\be_\ell$. It is also true that the cell $(r,\be_\ell+1)$ exists in $F$ and has entry $1$. Note that $F(\ell,\be_\ell)>1$ unless $(\ell,\be_\ell)$ is the $1$-cell in the basement, but this is not possible since then we have a type B coinversion triple in $F$. Thus, $F(\ell,\be_\ell)>F(r,\be_\ell+1)$, which entails $F(\ell,\be_\ell)>F(r,\be_\ell)$ since otherwise we have a type B coinversion triple in $F$. By case 2b of \Cref{lem.below}, we have $F(\ell,i)>F(r,i)$ for all $0\leq i\leq\be_\ell$, which implies $\si(\ell)>\si(r)$ and also that all type A triples within columns $\ell$ and $r$ in $F'$ are inversion.
        \end{enumerate}
        
        Conversely, for each $\be$ that is $\si$-extendable to $\al$, each filling $F'\in\SSAF_\si(\be)$ with no $1$-cells (other than the basement) may be extended to a filling $F$ by adding $\al_i-\be_i$ $1$-cells above each column $i\in[n]$ in $F'$. We prove $F\in\SSAF_\si(\al)$ by again considering two columns at a time; let $1\leq\ell<r\leq n$.
        \begin{enumerate}
            \item If $\al_\ell\geq\al_r$ and $\be_\ell\geq\be_r$, then $\al_r\leq\be_\ell$. We only need to confirm that type A triples containing a $1$-cell within columns $\ell$ and $r$ in $F$ are inversion. If column $\ell$ contains the $1$-entry in the basement, then $\be_\ell=0$, so $\al_r=0$, and there are no type A triples to check. Otherwise, since $\al_r\leq\be_\ell$, the $1$-cells in column $\ell$ occur at a height greater than $\al_r$ and thus are not in a type A triple within these columns. Further, all type A triples with a $1$-cell in column $r$ contain two cells in column $\ell$ with entries greater than $1$, so these triples are inversion.
            \item If $\al_\ell<\al_r$ and $\be_\ell<\be_r$, then $\al_\ell<\be_r$. We only need to confirm that type B triples containing a $1$-cell within columns $\ell$ and $r$ in $F$ are inversion. Since $\al_\ell<\be_r$, any $1$-cells in column $r$ occur at a height greater than $\al_\ell+1$ and thus are not in a type B triple within these columns. Further, all type B triples with a $1$-cell in column $\ell$ contain two cells in column $r$ which have entries greater than $1$, so these triples are inversion.
            \item If $\al_\ell\geq\al_r$ and $\be_\ell<\be_r$, then $\al_r=\be_r$ and $\si(\ell)<\si(r)$. Note that $F$ contains no $1$-cells in column $r$ since $\si(r)>1$ and $\al_r=\be_r$. Since $\si(\ell)<\si(r)$, by case 2a of \Cref{lem.below}, we have $F(\ell,i)<F(r,i+1)$ for all $0\leq i\leq\be_\ell$, yielding type A inversion triples $\{(\ell,i+1),(r,i+1),(\ell,i)\}$. The remaining type A triples in $F$ contain two adjacent $1$-cells in column $\ell$ and are inversion since column $r$ contains no $1$-cells.
            \item If $\al_\ell<\al_r$ and $\be_\ell\geq\be_r$, then $\al_\ell=\be_\ell$ and $\si(\ell)>\si(r)$. Note that $F$ contains no $1$-cells in column $\ell$ since $\si(\ell)>1$ and $\al_\ell=\be_\ell$. Since $\si(\ell)>\si(r)$, by case 1b of \Cref{lem.below}, we have $F(\ell,i)>F(r,i)$ for all $0\leq i\leq\be_r$, yielding type B inversion triples $\{(r,i+1),(\ell,i),(r,i)\}$. The remaining type B triples contain two adjacent $1$-cells in column $r$ and are inversion since column $\ell$ contains no $1$-cells.
        \end{enumerate}
        To finish the proof, note that the sum of the monomials $x^{F'}$ over all $F'\in\SSAF_\si(\be)$ where the only $1$-entry is in the basement is $\A_\be^\si(0,x_2,\ldots,x_n)$. The factor $x_1^{|\al|-|\be|}$ accounts for appending $|\al|-|\be|$ cells with entry $1$ to each of these fillings.
    \end{proof}
    Using the $\si$-extendable relation, we can define a poset on weak compositions, which is ranked by length, given by a covering relation depending on $\si$.
    \begin{definition}\label{def.poset}
        Set $\si_0:=\si$ and, for $0\leq i\leq n-1$, let
        \begin{align*}
            \si_{i+1}=(\si_i(1)-1,\ldots,\si_i(s_i-1)-1,\si_i(s_i+1)-1,\ldots,\si_i(n)-1),    
        \end{align*}
        where $s_i=\si_i^{-1}(1)$. In words, $\si_{i+1}$ is the permutation constructed from $\si_i$ by deleting the element $1$ and subtracting $1$ from the remaining elements. If $\len(\be)=\ell$, $\len(\al)=\ell+1$, $\len(\si)=n$ and $\ell<n$, we define the covering relation \itc{$\lessdot_\si$} so that $\be\lessdot_\si\al$ if $(\be_1,\ldots,\be_{s_k-1},0,\be_{s_k},\ldots,\be_\ell)$ is $\si_k$-extendable to $\al$, where $k=n-\ell-1$. Define the poset \itc{$\leq_\si$} as the transitive reflexive closure of $\lessdot_\si$. This poset is clearly ranked by length since $\be\lessdot_\si\al$ only when $\len(\be)=\len(\al)-1$.
    \end{definition}
    \begin{corollary}\label{cor.expansion}
        Atoms expand into summations over saturated chains in $\leq_\si$, given by
        \begin{align*}
            \A_\al^\si(x) = \sum_{\be_{(n)}\lessdot_\si\cdots\lessdot_\si\be_{(1)}\lessdot_\si\al}x_1^{|\al|-|\be_{(1)}|}x_2^{|\be_{(2)}|-|\be_{(1)}|}\cdots x_n^{|\be_{(n)}|-|\be_{(n-1)}|}.
        \end{align*}
    \end{corollary}
    \begin{proof}
        Consider the expansion from \Cref{thm.decomposition}:
        \begin{align*}
            \A_\al^\si(x) = \sum_\text{$\be$ $\si$-extendable to $\al$} x_1^{|\al|-|\be|}\A_\be^\si(0,x_2,\ldots,x_n).
        \end{align*}
        If $\be_{s_0}>0$, then $\be$ does not contribute to this sum, so assume $\be_{s_0}=0$ for each $\be$ summed over. Each $\be_{(1)}\lessdot_\si\al$ is obtained by deleting the part $\be_{s_0}$ from such a $\be$, so we may sum over $\be_{(1)}\lessdot\al$. Then we have that $\A_\be^\si(0,x_2,\ldots,x_n)=\A_{\be_{(1)}}^{\si_1}(x_2,\ldots,x_n)$ where $\si_1$ is obtained by deleting $1$ from $\si$ and subtracting $1$ from its remaining elements. Noting that $|\be_{(1)}|=|\be|$, we have a branching formula:
        \begin{align*}
            \A_\al^\si(x) &= \sum_{\be_{(1)}\lessdot_\si\al}x_1^{|\al|-|\be_{(1)}|}\A_{\be_{(1)}}^{\si_1}(x_2,\ldots,x_n).
        \end{align*}
        Recursively expanding our branching formula, we have
        \begin{align*}
            \A_\al^\si(x) &= \sum_{\be_{(n)}\lessdot_{\si_{n-1}}\cdots\lessdot_{\si_1}\be_{(1)}\lessdot_\si\al}x_1^{|\al|-|\be_{(1)}|}x_2^{|\be_{(2)}|-|\be_{(1)}|}\cdots x_n^{|\be_{(n)}|-|\be_{(n-1)}|} \\
            &= \sum_{\be_{(n)}\lessdot_\si\cdots\lessdot_\si\be_{(1)}\lessdot_\si\al}x_1^{|\al|-|\be_{(1)}|}x_2^{|\be_{(2)}|-|\be_{(1)}|}\cdots x_n^{|\be_{(n)}|-|\be_{(n-1)}|}.
        \end{align*}
        The final equivalence is possible by our definitions which ensure that if $k<\len(\si)-\len(\be)$, then $\be\lessdot_{\si_k}\al$ if and only if $\be\lessdot_\si\al$.
    \end{proof}
    \section{Equivalence of models and bijection with skyline fillings}\label{sec.bijection}
    We show that the atom model satisfies \Cref{cor.expansion}, which also gives a clear bijection with skyline fillings. We call a string of non-negative integers a \itc{descending string} if it does not contain the substring $ij$ for two positive integers $i<j$. Note that, by our encoding procedure, $\al^*$ is descending.
    
    \begin{lemma}\label{lem.descending}
        Within a tiling of the atom model, verticals between columns encode descending strings (moving up).
    \end{lemma}
    \begin{proof}
        Descending strings label the left and right boundaries. Assume that there is not a descending string between two columns, so there is a position in this vertical where colour $i$ is directly below $j$ and $i<j$. Set $i_0:=i$ and $j_0:=j$.
        
        We consider tiling to the left in case 1 and tiling to the right in case 2, both situations having four subcases:
        \begin{center}
            \robust{\ascale}
        \end{center}
        Case 1a transmits $i$ and $j_k$ to the left in the same order. In 1b, we must have $j_{k+1}>j_k>i$. A sequence of cases 1a and 1b then transmits some non-descending string containing $ij_k$ for $i<j_k$ to the left, which cannot continue perpetually since the left boundary is descending. The highlighted tile in $1c$ is not possible since we maintain $i<j_k$. If 1d occurs after a case of 1b, then line $i$ crosses the same line $j_k$ twice, which is not possible. Thus, case 1d must occur, possibly after some cases of 1a, and line $i=i_0$ crosses line $j=j_0$.
        
        A similar argument forces lines $i$ and $j$ to cross moving to the right, but then $i$ and $j$ cross twice, which is not possible. Thus, intermediate verticals are also descending strings.
    \end{proof}
    
    We show that the first column in the atom model essentially sums over all $\be_{(1)}\lessdot_\si\al$, satisfying \Cref{cor.expansion}, thus yielding the same polynomials as in the skyline formulation. Note that \Cref{lem.descending} shows that we may assume the right side of the first column is descending. This side will encode a weak composition $\be$ where $\be_s=0$ with a string similar to $\be^*$, only skipping over the character $s=\si^{-1}(1)$. Precisely, we call this string $\be^{(s)}$, which is obtained by deleting the character $s$ from $\be^*$ and appending $0$ to the end.
    
    \begin{example}
        Let $\al=(3,2,0,2)$ and $\si=4123$, so that $\si^{-1}=2341$ and $s=\si^{-1}(1)=2$. There are two ways to tile column $1$ so that the right side is a descending string. These strings encode weak compositions $\be=(2,0,0,2)$ and $\de=(3,0,0,2)$. We compute the boundary strings:
        \begin{align*}
            \al^* = 3004201 && \be^* = 3200410 && \be^{(2)} = 3004100 \\
            && \de^* = 3200401 && \de^{(2)} = 3004010
        \end{align*}
        \begin{center}
            \firstcolumnex{\bscale}
        \end{center}
    \end{example}
    \pagebreak
    \begin{lemma}\label{lem.columnsum}
        The first column in the atom model is possible to tile if and only if $\be^{(s)}$ labels its right side, where $s=\si^{-1}(1)$ and $\be$ is $\si$-extendable to $\al$.
    \end{lemma}
    \begin{proof}
        We have argued above that we may assume the right side of the first column always encodes a weak composition $\be$, where $\be_s=0$, with the string $\be^{(s)}$. We must show that it is possible to tile the model if and only if $\be$ is $\si$-extendable to $\al$. Note that both strings $\be^*$ and $\be^{(s)}$ will have $\be_i$ instances of the character $0$ before the character $i$, determining the size of part $i$. Suppose it is possible to tile the first column and let $1\leq\ell<r\leq n$.
        \begin{enumerate}
            \item If $\al_\ell\geq\al_r$ and $\be_\ell\geq\be_r$, then the label $r$ occurs below the label $\ell$ on the left side. If $\ell=s$, then $\be_\ell=\be_r=0$ and line $\ell$ begins on the bottom side of the column and crosses line $r$. Then $\al_r=\be_r=0$ and thus $\al_r\leq\be_\ell$. Otherwise, lines $\ell$ and $r$ do not cross. Let $i\leq r$ be the smallest integer where $\al_i=\al_r$ and $j\geq\ell$ be the largest integer where $\be_j=\be_\ell$. If line $i$ exits on the left at height $h$, then line $j$ must exit on the right at height at least $h$, which implies $\al_r\leq\be_\ell$.
            \item If $\al_\ell<\al_r$ and $\be_\ell<\be_r$, then the label $\ell$ occurs below $r$ on both sides of the column and the lines do not cross. Let $i\leq\ell$ be the smallest integer where $\al_i=\al_\ell$ and $j\geq r$ be the largest integer where $\be_j=\be_r$. If line $i$ exits on the left at height $h$, then line $j$ must enter on the right at a height $h'\geq h+1$ and there must be a $0$ between height $h$ and $h'$ on the right, which implies $\al_\ell<\be_r$.
            \item If $\al_\ell\geq \al_r$ and $\be_\ell<\be_r$, we have $r\neq s$ since $\be_r>0$; thus, line $r$ must cross horizontally over line $\ell$, so $\al_r=\be_r$ and $\si(\ell)<\si(r)$. 
            \item If $\al_\ell<\al_r$ and $\be_\ell\geq\be_r$, suppose $\ell=s$; then $\be_r=0$. Line $\ell$ must cross vertically under all lines $i\geq\ell$ where $\al_i=0$, so $\al_r\leq\al_\ell$, contradicting our assumptions. Otherwise, line $\ell$ must cross horizontally over line $r$, so $\al_\ell=\be_\ell$ and $\si(\ell)>\si(r)$.
        \end{enumerate}
        
        Conversely, suppose $\be$ is $\si$-extendable to $\al$ and $\be_s=0$. Label the left of column $1$ with $\al^*$ and the right with $\be^{(s)}$. We show that it is always possible to tile this column from the bottom up, considering one tile at a time. The left, right and bottom edges of the tile will have labels $a,b,c$:
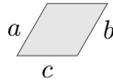
        \begin{center}
            \singletile{\ascale}
        \end{center}
        Let the current tile be at height $h$. Let $L(h)$ be the number of $0$'s below height $h$ on the left side of the column and let $R(h)$ be the number of $0$'s below height $h$ on the right side. If $a$ and $b$ are positive integers, we have $\al_a=L(h)$ and $\be_b=R(h)$ respectively. We say the \itc{zero condition} is satisfied at height $h$ when
        \begin{align*}
            L(h) =
            \begin{cases}
                R(h)-1 &\text{if $c=0$} \\
                R(h) &\text{if $c>0$.}
            \end{cases}
        \end{align*}
        The zero condition is satisfied when $h=1$, where $c=s>0$ and $L(h)=R(h)=0$. Let $a$, $b$ and $c$ be distinct positive integers. There are 15 cases for the tile at height $h$:
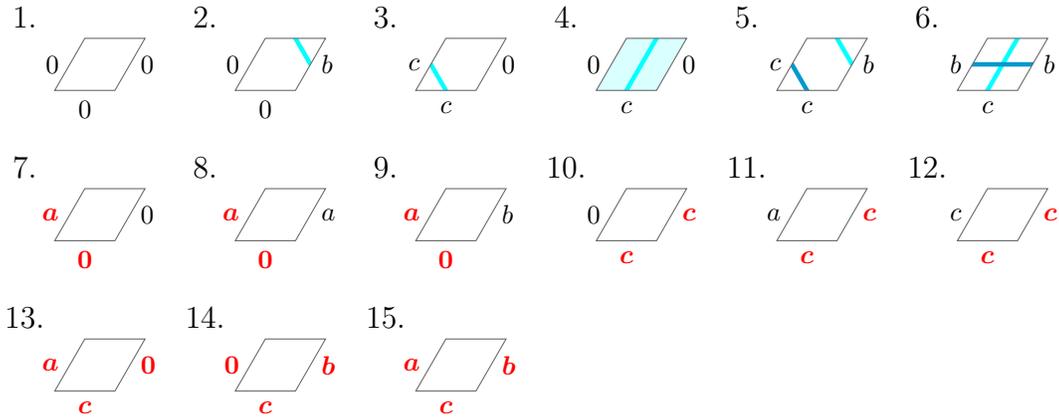
        \begin{center}
            \tilecases{\ascale}
        \end{center}
        We have emphasized problematic clusters of labels in bold red font. We want to show that only cases 1-6 are possible and a tile may always be placed in these cases; this will also ensure we satisfy the zero condition at height $h+1$, so by induction we can tile the entire column.
        
        \begin{enumerate}
            \item[1-4.] It is always possible to place a tile in these cases.
            \item[5.] We must show that $b<c$ to have a valid tile; suppose $c<b$. Then line $b$ exits on the left at a height above $h$ and, since $\al^*$ is a descending string, we must have $\al_c<\al_b$. We must also have that $\be_c\leq\be_b$; if $c=s$ and $\be_c=\be_b$, we have that $\si(c)>\si(b)$ by case 4 of \Cref{def.extendable}, which is not possible since $\si(c)=1$. Otherwise, $c=s$ and $\be_c<\be_b$ or $c\neq s$, and we still have $\be_c<\be_b$ since $\be^{(s)}$ is a descending string. By case 2 of \Cref{def.extendable}, we have that $\al_c<\be_b$, but $\al_c=L(h)=R(h)=\be_b$; hence, $b<c$.
            \item[6.] For the crossing to occur, we must show $\si(c)<\si(b)$. If $c=s$, then $\si(c)=1<\si(b)$, as desired. Othewise, if $c<b$, then $\al_c\geq\al_b$ and $\be_c<\be_b$, so $\si(c)<\si(b)$ by case 3 of \Cref{def.extendable}. If $b<c$, then $\al_b<\al_c$ and $\be_b\geq\be_c$, so $\si(b)>\si(c)$ by case 4 of \Cref{def.extendable}.
            \item[7-9.] These cases cannot occur since then line $a$ enters on the right at height at least $h$; since $L(h)<R(h)$, we have $\al_a<\be_a$.
            \item[10-12.] These cases cannot occur since then line $c$ enters the column on the bottom or right twice, which is not possible.
            \item[13.] Since we assume $\al_a=L(h)=R(h)$ and the right label is $0$, this case implies $\al_a<\be_a$, which is not possible.
            \item[14.] We have that $\al_b$ and $\al_c$ are greater than both $\be_b$ and $\be_c$ because of the $0$ on the left. Thus, we cannot satisfy \Cref{def.extendable} since each case implies $\al_d\leq\be_e$ for at least one pair $d,e\in\{b,c\}$.
            \item[15.] This case implies line $a$ must enter on the right side at a height above $h$. This is only possible if $\be_a=\be_b$ and $a<b$, otherwise $\al_a<\be_a$. On the left side, labels $b$ and $c$ must be above $a$, so we have $\al_b\geq\al_a$ and $\al_c\geq\al_a$.
            
            We cannot have that $\al_b=\al_a$ since $a<b$, so $\al_a<\al_b$. If $\al_c=\al_a$, then $c<a<b$ and $\al_c<\al_b$; note that $\be_c\leq\be_b$. If $\be_c<\be_b$, then by case 2 of \Cref{def.extendable}, we have $\al_c<\be_b$, meaning $c$ must have exited at a previous step. If $\al_c=\al_b$, we are in case 4 of \Cref{def.extendable} and $\si(c)>\si(b)$, meaning it is not possible that $c=s$; otherwise, $\al_c=\be_c$, meaning $\be_c=\be_b$, which is not possible since $\be^{(s)}$ is a descending string.
            
            Thus, $\al_b>\al_a$ and $\al_c>\al_a$, implying that $\al_b$ and $\al_c$ are both greater than $\be_b$ and $\be_c$, and we cannot satisfy \Cref{def.extendable} for the same reason as in case 14.
        \end{enumerate}
    \end{proof}
    We now justify that the atom model yields the same polynomials as those given in \Cref{sec.skyline}.
    \begin{theorem}\label{thm.equal}
        The polynomials given by the atom model satisfy \Cref{cor.expansion}.
    \end{theorem}
    \begin{proof}
        By \Cref{lem.columnsum}, we can consider a sum over $\be$ that are $\si$-extendable to $\al$, where $\be_s=0$, so we may instead sum over all $\be_{(1)}\lessdot_\si\al$, where $\be_{(1)}=(\be_1,\ldots,\be_{s-1},\be_{s+1},\ldots,\be_n)$. For each $\be_{(1)}$, the first column will have weight $x_1^{|\al|-|\be_{(1)}|}$.
        
        Considering the columns $2,\ldots,n$ separately, with $\be^{(s)}$ along the left boundary, we set each colour $i>s$ equal to $i-1$. The left boundary will then be $\be_{(1)}^*$ and the bottom boundary will be $\si_1^{-1}$, where $\si_1$ is as in \Cref{def.poset}. If we let $\si_1$ determine the crossing tiles instead of $\si$, this is the model for $\A_{\be_{(1)}}^{\si_1}(x_2,\ldots,x_n)$. The tilings of this model are identical to the columns $2,\ldots,n$ in the original picture, only with the colours $i>s$ decreased by $1$. Thus, we have the same branching formula as in \Cref{cor.expansion} and can expand the polynomials as before:
        \begin{align*}
            \A_\al^\si(x) &= \sum_{\be_{(1)}\lessdot_\si\al}x_1^{|\al|-|\be_{(1)}|}\A_{\be_{(1)}}^{\si_1}(x_2,\ldots,x_n) \\
            &= \sum_{\be_{(n)}\lessdot_\si\cdots\lessdot_\si\be_{(1)}\lessdot_\si\al}x_1^{|\al|-|\be_{(1)}|}x_2^{|\be_{(2)}|-|\be_{(1)}|}\cdots x_n^{|\be_{(n)}|-|\be_{(n-1)}|}.
        \end{align*}
    \end{proof}
    Similar to \cite{BW22}, we may obtain the SSAF corresponding to a tiling by simply following line $i$ to produce the entries in column $i$. Begin with a skyline diagram of shape $\al$ with basement $\si$ and other entries blank. For each line labelled $i$ in a tiling, follow line $i$ from the bottom, moving up. Each time line $i$ moves through a nontrivial tile within a column $c$, assign the entry $c$ to the bottom-most blank entry of column $i$ in the skyline diagram.

    \begin{example}
        We show the corresponding SSAFs for the tilings in \Cref{ex.models}, where $\al=(1,0,2,2)$ and $\si=4123$, so $\si^{-1}=2341$ and $\al^*=201043$. \textit{Note that columns in SSAFs are numbered $1$ to $n$ from left to right. Columns are not numbered by their basement entries.}
        \begin{center}
            \bijectionex{\cscale}
        \end{center}
    \end{example}
    
    \section{Vertex models for $a_{\al\la}^\be(\si)$}\label{sec.theorem}
    We present our vertex models for the structure coefficients $a_{\al\la}^\be(\si)$. Our models use the ``diamond'' tiles below where $f,g,h,i$ are in $[n]$, $\si(f)<\si(g)$ and $\si(f)<\si(h)$. The label $b^+$ indicates a blue line $b$ and a red line share that edge. The label $ab$ indicates two shades of blue $a$ and $b$ with $\si(a)<\si(b)$ share that edge. All tiles have weight $1$.
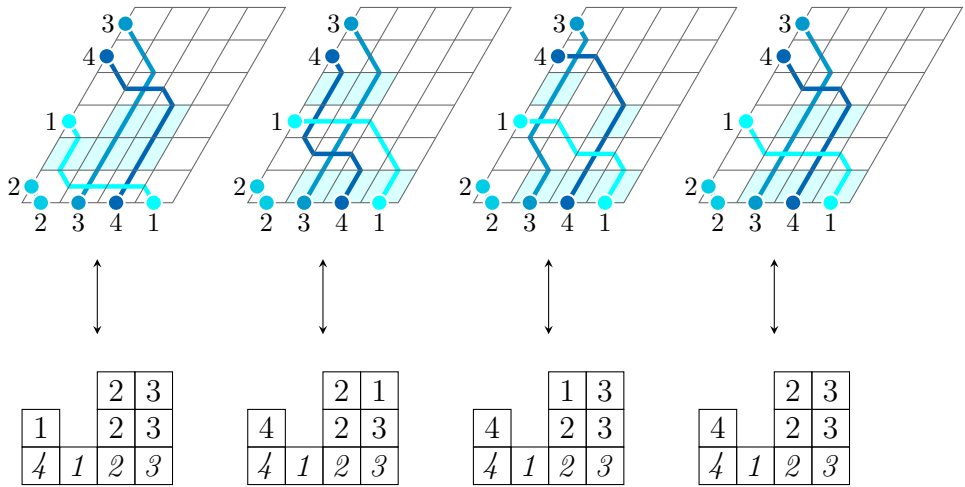
    \begin{center}
        \diamonds{\ascale}
    \end{center}
    In addition to having labels match, we must also ban certain adjacencies not captured by the labels, as depicted in \Cref{fig.banned}. We do not allow a line of colour $i$ to move diagonally directly below a line of colour $j$ if $i<j$. These restrictions are still local and can be imposed with additional labels (see \Cref{app.labels}), but we exclude them to avoid clutter. We now state our theorem.
    
    \begin{theorem}\label{thm.diamonds}
        The structure coefficients $a_{\al\la}^\be(\si)$ are calculated as the partition function of the model
        \begin{center}
            \producttheorem{\bscale}
        \end{center}
        where $k=\max(\be)$ and the restrictions in \Cref{fig.banned} apply within. The symbol $\rcirc^k\wcirc^n$ represents the string $+^k0^n$.
    \end{theorem}
    \begin{proof}
        We reserve this proof for \Cref{sec.proof}.
    \end{proof}
    
    \begin{figure}
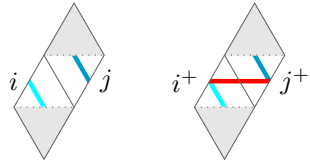

        \stepcounter{theorem}
        \centering
        \banned{\ascale}
        \caption{Banned adjacencies where $1\leq i<j\leq n$.}\label{fig.banned}
    \end{figure}
    
    We call our vertex model for $a_{\al\la}^\be(\si)$ the \itc{atom-Schur model}. Recall that we assume $\al$, $\be$ and $\la$ have length $n$, but we may append zeros to match their lengths. Similarly, we may append zeros to the end of the string $\al^*$ and append $+$ symbols to the end of $\la^+$ so that these strings all have length $n+k$.
    
    \begin{remark}
        The atom-Schur model may be cut in half because the top half is always fixed, giving a triangular puzzle rule as in \cite{KZ20,KZ21,KZ23,WZ19,Z09}.
    \end{remark}
    
    \begin{example}
        For $\al=(1,3,1,0)$, $\la=(3,1,0,0)$ and $\be=(1,4,3,1)$, we have $k=\max(\be)=4$. If $\si(1)<\si(3)$, there are two tilings, showing $a_{\al\la}^\be(\si)=2$, otherwise line $3$ may not cross horizontally over line $1$ and $a_{\al\la}^\be(\si)=0$.
        \begin{center}
            \atomdiamondex{\cscale}
        \end{center}
    \end{example}
    
    \section{Proof of \Cref{thm.diamonds}}\label{sec.proof}
    
    We put the three models in one picture to show \Cref{thm.diamonds}. The crucial machinery for this proof is a restricted Yang--Baxter equation, which we refer to as the Column \Cref{lem.column}. New tiles are necessary, which expand the previous set and introduce a new orientation.
    
    \Cref{fig.alltiles} depicts the complete set of tiles. We call the first, second and third sets of tiles \itc{right-sheared}, \itc{left-sheared} and \itc{diamond} tiles, respectively. The left-sheared tiles contain a new nontrivial tile of weight $-x_c$. We again depict vertex models with grey diagrams tiled only with tiles of matching orientation. Left- and right-sheared tiles share the same column number $c$ if one is on top of the other; columns are numbered $1$ to $n$ from left to right. Some tiles now have ``double'' red and blue lines; these tiles facilitate the proof, but can be excluded in \Cref{thm.diamonds}.
    
    We still ban adjacencies between diamond tiles as in \Cref{fig.banned}, but there are no restrictions between tiles that are not both diamonds. Left-sheared tiles are unaffected; for example, the following configuration is legal for \textit{any} $i$, $j$ and $k$ in $[n]$:
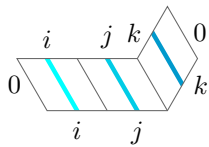
    \begin{center}
        \allowedglue{\ascale}
    \end{center}
    
    \begin{note}\label{not.tiles}
        Diamond tiles are determined by cutting all trivial right-sheared tiles into triangles and gluing them together in all possible ways. This introduces new labels responsible for the restrictions in \Cref{fig.banned}. We suppress these extra labels to clean up analysis, but show them in \Cref{app.labels}. The left-sheared tiles are not determined this way and do not allow crossings between blue lines.
    \end{note}
    
    The key to our proof is a restricted Yang--Baxter equating columns of tiles. The tiles in \Cref{fig.alltiles} are minimal in the sense that this lemma does not hold if we remove any tile.
    
    \begin{figure}
        \stepcounter{theorem}
        \centering
        \xtiles{\ascale}
        \ytiles{\ascale}
        \ztiles{\ascale}
        \caption{The full set of tiles where $f,g,h,i,j,k$ are in $[n]$, $i\leq j$, $\si(f)<\si(g)$ and $\si(f)<\si(h)$.}\label{fig.alltiles}
    \end{figure}
    \pagebreak
    \begin{lemma}[The Column Lemma]\label{lem.column}
        Let $q_1,\ldots,q_m$, $r$, $s$, $t_1,\ldots,t_m,u,v$ be fixed labels where $r$ and $u$ are in $\{0,+\}$. We have the following equality:
        \begin{center}
            \columnlemma{\ascale}
        \end{center}
    \end{lemma}
    \begin{proof}
        We reserve this proof for \Cref{sec.colproof}.
    \end{proof}
    
    \begin{example}
        We consider two examples of \Cref{lem.column}. In the first example, both sides have weight $x_c^2$. In the second, there are two ways to tile the column on the left-hand side which cancel, and there is no way to tile the column on the right-hand side.
        \begin{center}
            \columnexamplea{\cscale}
            \columnexampleb{\cscale}
        \end{center}
    \end{example}
    
    \begin{remark}\label{rmk.YBE}
        In \cite{WZ19,Z09}, the authors proceed similarly with a Yang--Baxter equation that equates unit hexagons with unrestricted boundaries. In contrast, our equation restricts the labels $r$ and $u$, which might suggest a more general framework to explore.
    \end{remark}
    
    Interpreting the next lemma proves our result.
    \pagebreak
    \begin{lemma}\label{lem.prisms}
        Set $k=\max(\al)+\max(\la)$. The following configurations have the same partition function:
        \begin{center}
            \prisms{\bscale}
        \end{center}
    \end{lemma}
    
    \begin{proof}
        Repeatedly applying the Column \Cref{lem.column} to internal columns of length $2(k+n)$, from right to left, top to bottom, shows equality. We illustrate this schematically, omitting the boundaries:
        \begin{center}
            \prismflip{\dscale/2}
        \end{center}
        The boundary conditions ensure that the southwest and northeast labels of columns we equate are always in $\{0,+\}$ at every stage in this process, allowing repetition of \Cref{lem.column}.
    \end{proof}
    
    The proof now follows from examining both sides of the equation in \Cref{lem.prisms}. In short, the left-hand side is manifestly the product $\A_\al^\si(x)s_\la(x)$ and the right-hand side is manifestly the summation $\sum_\be a_{\al\la}^\be(\si)\A_\al^\si(x)$ where $a_{\al\la}^\be(\si)$ is the partition function of the atom-Schur model.

    \begin{figure}
        \stepcounter{theorem}
        \centering
        \prismexample{\cscale}
        \caption{Two tilings of weight $x_1^3x_2^2x_3^4$ from the configurations in \Cref{lem.prisms} where $\la=(2,2,1)$, $\si=231$ and $\al=(2,2,1)$, so that $n=3$, $k=4$ and $\si^{-1}=312$. Regions are labelled to facilitate exposition.}\label{fig.prisms}
    \end{figure}
    \pagebreak
    \begin{proof}[Proof of \Cref{thm.diamonds}]
        Before proceeding, we note a few differences from our main theorem:
        \begin{enumerate}
            \item Here, we include extra diamond tiles to make the Column \Cref{lem.column} hold; these tiles have ``double'' blue and red lines. These tiles will not occur when filling the atom-Schur model, so we omit them in the statement of \Cref{thm.diamonds}.
            \item We have set $k=\max(\la)+\max(\al)$, but when considering only a particular atom-Schur model where $\be$ is given, it suffices to set $k=\max(\be)$, as is done in \Cref{thm.diamonds}.
        \end{enumerate}
        See the illustration in \Cref{fig.prisms} where we label regions ${\color{blue6}\mathbf{A}},{\color{blue6}\mathbf{B}},\ldots,{\color{blue6}\mathbf{K}}$. Within region \BB{A}, all blue lines must move northwest and all red lines must move east. The red lines continue straight northeast through \BB{B}, transmitting the string $\la^+$ to the southwest boundary of \BB{C}. Lemma 10 in \cite{Z09} says that there is only one way to tile \BB{C}, which forces the shared boundary between \BB{C} and \BB{E} to be the string $\la^-$ upside-down. Thus, regions \BB{A}, \BB{B} and \BB{C} have weight $1$.
        
        Next, we recognize region \BB{D} as the atom model $\A_\al^\si(x)$. From the previous paragraph, we have that $\la^-$ is upside-down on the southeast boundary of \BB{E}. Rotating region \BB{E} by $180$ degrees, we see that it is the Schur model for $s_\la(x)$ with the variables in reverse order; since $s_\la(x)$ is symmetric, the weight of the left-hand side is the product $\A_\al^\si(x)s_\la(x)$. We summarize pictorially:
        \begin{center}
            \LHSeq{\cscale}
        \end{center}
        
        On the right-hand side, the regions \BB{G}, \BB{H} and \BB{I} always have weight $1$ and cross as in \Cref{fig.prisms}. Within region \BB{J}, all blue lines must exit the northwest boundary if they are to reach the northwest boundary of \BB{K} (our choice of $k$ makes the boundary long enough to ensure this). The blue lines then travel through the boundary between \BB{J} and \BB{K}, varying over strings $\be^*$ that encode weak compositions $\be$ (these strings will be descending by a similar argument made in \Cref{lem.descending}). For each $\be^*$ along this boundary, we recognize region \BB{J} as $\A_\be^\si(x)$ and region \BB{K} as the atom-Schur model from \Cref{thm.diamonds}. Thus, the right-hand side is a summation over compositions $\be$ where each summand is a product of the atom-Schur model $a_{\al\la}^\be(\si)$ and $\A_\be^\si(x)$. Again, summarizing pictorially:
        \begin{center}
            \RHSeq{\cscale}
        \end{center}
        By \Cref{lem.prisms}, we can equate both sides, completing the proof.
    \end{proof}
    
    \section{Proof of the Column \Cref{lem.column}}\label{sec.colproof}
    We prove the key technical result of this paper, the Column \Cref{lem.column}. For a label $q\in\{b,b^+\}$ with $b\in[n]$, we refer to its colour by $\#(q)=b$ and set $\#(0)=\#(+)=0$. We may mark an edge labelled $q$ with a symbol to indicate restrictions on $q$. The symbol $\wcirc$ indicates $q=0$, $\wrcirc$ indicates $\#(q)=0$ and $\brcirc$ indicates $\#(q)\in[n]$.
    
    We sometimes set the colours of certain labels equal. Doing this has the effect of uncrossing lines, altering the following tiles (dotted red lines indicate a red line may be present or not):
    \begin{center}
        \skeletons{\ascale}
    \end{center}
    Restating \Cref{lem.column} with these conventions, we wish to prove the equation
    \begin{center}
        \columnlemmarestated{\ascale}
    \end{center}
    We first show that \Cref{lem.column} is true when both columns use only trivial tiles, that is, when $x_c=0$.
    
    \begin{lemma}\label{lem.trivial}
        \Cref{lem.column} is true when $x_c=0$.
    \end{lemma}
    \begin{proof}
        As stated in \Cref{not.tiles}, diamond tiles are produced by taking trivial right-sheared tiles, cutting them into triangles, and gluing them together in all possible ways. Left-sheared tiles behave differently, not allowing crossings. The result is that columns with only trivial tiles are decorated identically unless $0<s<\#(t_1)$ or $v>\#(q_m)>0$. We outline the general case where both of these are true:
        \begin{center}
            \unweightedproof{\bscale}
        \end{center}
        In the left column, line $s$ crosses vertically under some contiguous sequence of lines that exit at $q_1,\ldots,q_i$. The index $i$ is the smallest integer where $\#(t_{i+1})\leq s<\#(q_i)$ or $i=m$. Within the right column, line $s$ must turn to the left after exiting the tile labelled by $t_i$. A similar situation is forced in the top region for the line $v$. The grey regions \BB{C} are identical in both columns. Here we set the colours of $q_1,\ldots,q_i$ and $q_j,\ldots q_m$ equal, but these colours may be different, in which case crossings will occur identically in regions \BB{A} or \BB{B}.
    \end{proof}
    
    We proceed by induction on $m$. In light of the following lemma, the base case $m=1$ holds.
    \begin{lemma}\label{lem.emptyhex}
        The unit hexagons with the following restrictions are equivalent:
        \begin{align*}
            \zerohexa{\ascale} \qquad\qquad\qquad \zerohexb{\ascale}
        \end{align*}
    \end{lemma}
    \begin{proof}
        Both equations follow from brute force. Note that $r$ and $u$ are unrestricted in the second equation.
    \end{proof}
    
    Many cases are now possible to prove by induction.
    
    \begin{lemma}\label{lem.empties}
        Assume that \Cref{lem.column} holds for columns of length $m-1\geq1$. Then if $q_i=t_j=0$ for $1\leq i\leq j\leq m$, we have:
        \begin{center}
            \fliplemmaa{\bscale}
        \end{center}
    \end{lemma}
    \begin{proof}
        Note that if the northwest or southeast edge of a diamond tile is $0$, then the northeast or southwest edge, respectively, must be in $\{0,+\}$ since these are the only such tiles:
        \begin{center}
            \forceddiamonds{\ascale}
        \end{center}
        From this fact, we know some internal edges are in $\{0,+\}$, which we label with the $\wrcirc$ symbol. If $i<j$, we have the following equivalences by applying \Cref{lem.column} three times, going backwards part of the way in the second equality:
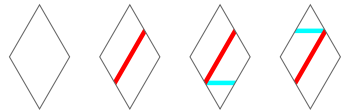
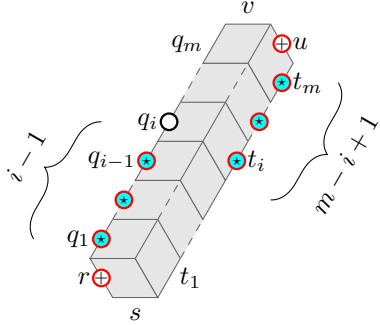
        \begin{center}
            \fliplemmab{\bscale}
        \end{center}
        If $i=j$, we have the following equivalences, where the second equality follows from applying the second equation in \Cref{lem.emptyhex}:
        \begin{center}
            \fliplemmac{\bscale}
        \end{center}
    \end{proof}
    
    Since $r$ and $u$ are in $\{0,+\}$, there can be at most $m+1$ separate blue lines within a column, entering at $s,t_1,\ldots,t_m$ and exiting at $v,q_1,\ldots,q_m$. The next lemma shows that we can assume at least $m-1$ separate blue lines.
    \begin{lemma}\label{lem.minlines}
        \Cref{lem.column} holds for columns of length $m$ containing less than $m-1$ separate blue lines.
    \end{lemma}
    \begin{proof}
        If $q_i\neq0$ for all $1\leq i\leq m$, there are at least $m$ separate blue lines within a column. Otherwise, let $i$ be the smallest integer with $q_i=0$, so there is a blue line at $q_k$ for $1\leq k<i$. By \Cref{lem.empties}, we can assume that $t_j\neq0$ for $i\leq j\leq m$. This accounts for $(i-1)+(m-i+1)=m$ edges incident with a blue line:
        \begin{center}
            \numberofblues{\bscale}
        \end{center}
        
        The line entering at $t_i$ may exit at $q_{i-1}$, but none of the lines entering at $t_{i+1},\ldots,t_m$ may exit at $q_1,\ldots,q_{i-1}$, accounting for at least $m-1$ separate blue lines.
    \end{proof}
    We now prove \Cref{lem.column} by handling the remaining cases.
    \begin{proof}[Proof of the Column \Cref{lem.column}]
        From \Cref{lem.minlines}, we only need to consider the cases of a column with at least $m-1$ separate blue lines. We again set labels equal, uncrossing blue lines, to show the general structure of tilings.
        \begin{enumerate}
            \item If there are $m-1$ separate blue lines, we have $q_i=0$ for some $1\leq i\leq m$. As argued in \Cref{lem.minlines}, we can assume $t_i,\ldots,t_m$ and $q_1,\ldots,q_{i-1}$ are nonzero and a blue line $b$ entering at $t_i$ must exit at $q_{i-1}$. If $b$ crosses a blue line, we have at least $m$ separate blue lines, so $b$ crosses two red lines. These cases only use trivial tiles, covered by \Cref{lem.trivial}. We give a general outline where $1\leq h<i\leq m$:
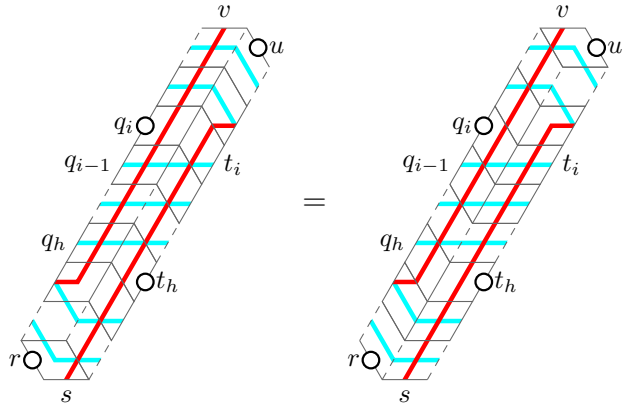
            \begin{center}
                \horizontalcases{\bscale}
            \end{center}
            \item If there are $m+1$ separate blue lines, these cases only use trivial tiles and so are covered by \Cref{lem.trivial}:
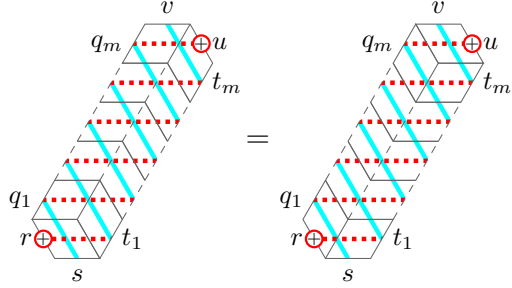
            \begin{center}
                \fullskeleton{\bscale}
            \end{center}
            \item If there are $m$ separate blue lines, there are four outlines using a nontrivial tile.
            \begin{center}
                \partskeletona{\cscale}
                \partskeletonb{\cscale}
                \partskeletonc{\cscale}
                \partskeletond{\cscale}
            \end{center}
            In each case, allowing different colours preserves equality. Case 3a is clear, though the other cases require more careful analysis and examples can get quite large. We outline case 3b in the most general situation, tracking the blue line in the left tiling that crosses through the nontrivial tile labelled by $q_i=0$. Say this line enters at $t_h$ and exits at $q_k$ and let $1\leq h\leq i<j<k\leq m$:
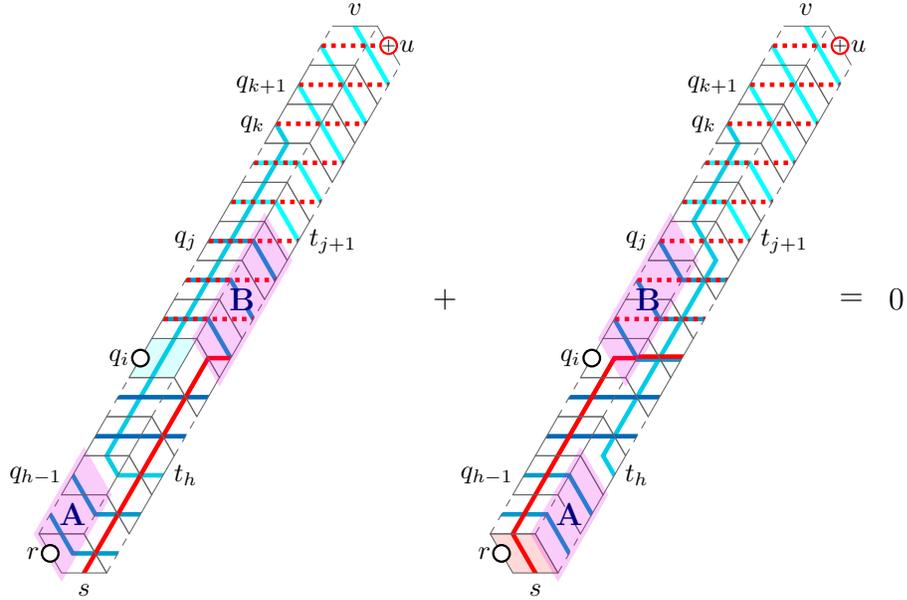
            \begin{center}
                \finalcaseb{\bscale}
            \end{center}
            In this most general case, we assume that there exists some smallest value of $j$ where $\#(q_j)<t_h\leq\#(t_{j+1})$ and $j>i$. We also require that $\#(q_{h-1})\leq t_h$ and $\#(q_k)\leq\#(q_{k+1})$ to tile either column. As before, crossings in regions \BB{A} or \BB{B} will occur identically.
        \end{enumerate}
    \end{proof}
    \section*{Acknowledgements}
    I am deeply grateful to Kevin Purbhoo for introducing me to this approach, for many discussions, and for his boundless encouragement over the course of this project. I thank Oliver Pechenik for valuable feedback on the final draft. Allen Knutson and Paul Zinn-Justin hosted a great workshop on vertex models (LAWRGe 2022) that answered some burning questions. Allen also had good feedback probing this framework. I thank Per Alexandersson, Valentin Buciumas, Olya Mandelshtam and Travis Scrimshaw for their comments. This work was supported by Kevin Purbhoo's NSERC Discovery Grants.
    
    \appendix
    \section*{Appendix}
    \section{Triangle tiles and extra labels}\label{app.labels}
    We may add extra labels to our tiles in \Cref{fig.alltiles} to enforce the restrictions in \Cref{fig.banned} between diamond tiles, only requiring that labels match along edges. We construct rhombic tiles out of triangles with extra labels, emphasized in bold purple font.
    \begin{center}
        \cuttriangles{\ascale}
    \end{center}
    Here, $i,f,g\in[n]$, $\si(f)<\si(g)$, $0\leq a\leq n$ and $0\leq b\leq i$. Rhombic tiles only use a subset of triangular tiles, depending on the orientation of the rhombus. Pairs of triangles within each subset are glued together in all possible ways, where labels match, producing rhombi of the corresponding orientation.
    \begin{itemize}
        \item Triangles \BB{A} and \BB{B} construct diamond tiles.
        \item Triangles \BB{A}, \BB{B} and \BB{C} construct right-sheared tiles.
        \item Triangles \BB{A} and \BB{D} construct left-sheared tiles.
    \end{itemize}
    The resulting tiles are almost identical to \Cref{fig.alltiles}, only with two extra labels on the diamond and left-sheared tiles; right-sheared tiles are the same as before. However, this new labelling scheme implies we should adjust the Column \Cref{lem.column} to sum over these new labels on the right-hand side, giving
    \begin{center}
        \extracolumnlemma{\ascale}
    \end{center}
    where $r$ and $u$ are in $\{0,+\}$. The statement of \Cref{thm.diamonds} will also change, summing over all values of the extra labels along the southwest boundary and setting extra labels along the northeast boundary to $0$.
    \bibliography{bibfile}
\end{document}